\documentclass[12pt]{article}

\usepackage{amsmath,amsthm,amsfonts,amssymb, color,graphicx}

\usepackage[implicit=true]{hyperref}

\allowdisplaybreaks[4]

\newtheorem{theorem}{Theorem}[section]

\newtheorem{example}[theorem]{Example}

\newtheorem{lemma}[theorem]{Lemma}

\newtheorem{remark}[theorem]{Remark}

\def\bD{\mathbb D}
\def\bE{\mathbb E}
\def\bN{\mathbb N}
\def\bP{\mathbb P}
\def\bR{\mathbb R}

\def\cA{\mathcal A}
\def\cB{\mathcal B}
\def\cD{\mathcal D}
\def\cF{\mathcal F}
\def\cH{\mathcal H}

\def\G{\Gamma}

\def\cD{\mathcal{D}}

\def\cF{\mathcal{F}}
\def\cG{\mathcal{G}}
\def\cI{\mathcal{I}}

\def\cH{\mathcal{H}}
\def\cP{\mathcal{P}}

\def\cS{\mathcal{S}}

\def\cX{\mathcal{X}}
\def\cY{\mathcal{Y}}

\def\fX{\mathfrak{X}}

\def\bD{\mathbb{D}}
\def\bE{\mathbb{E}}
\def\bR{\mathbb{R}}

\def\G{\Gamma}

\newcommand{\les}{\lesssim}

\def\e{\varepsilon}

\begin{document}

\title{Hyperbolic Anderson model with time-independent rough noise: Gaussian fluctuations\\}
\date{May 8, 2023}

\author{Raluca M. Balan\footnote{Corresponding author. University of Ottawa, Department of Mathematics and Statistics, STEM Building, 150 Louis-Pasteur Private,
		Ottawa, Ontario, K1N 6N5, Canada. E-mail: rbalan@uottawa.ca.} \footnote{Research supported by a grant from Natural Sciences and Engineering Research Council of Canada.}
	\and Wangjun Yuan\footnote{University of Luxembourg, Department of Mathematics, Maison du Nombre
6, Avenue de la Fonte
L-4364 Esch-sur-Alzette. Luxembourg. E-mail: ywangjun@connect.hku.hk.} \footnote{The author gratefully acknowledges the financial support of ERC Consolidator Grant 815703 "STAMFORD: Statistical Methods for High Dimensional Diffusions"}
}

\maketitle

\begin{abstract}
\noindent In this article, we study the hyperbolic Anderson model in dimension 1, driven by a time-independent rough noise, i.e. the noise associated with the fractional Brownian motion of Hurst index $H \in (1/4,1/2)$. We prove that, with appropriate normalization and centering, the spatial integral of the solution converges in distribution to the standard normal distribution, and we estimate the speed of this convergence in the total variation distance. We also prove the corresponding functional limit result.
Our method is based on a version of the second-order Gaussian Poincar\'e inequality developed recently in \cite{NXZ}, and relies on delicate moment estimates for the increments of the first and second Malliavin derivatives of the solution. These estimates are obtained using a connection with the wave equation with delta initial velocity, a method which is different than the one used in \cite{NXZ} for the parabolic Anderson model.
\end{abstract}

%60H15=spdes
%60H07=stochastic calculus of variations and the Malliavin calculus
%60G15=Gaussian processes
%60F05=central limit theorem and other weak theorems


\pagebreak

\tableofcontents

\section{Introduction}

The study of stochastic partial differential equations (SPDEs) using the random field approach originates in Walsh' lecture notes \cite{walsh86}, which introduced the general framework, focusing mostly on equations driven by space-time Gaussian white noise in dimension 1.
In the seminal article \cite{dalang99}, Dalang extended the martingale measure method of Walsh to equations driven by spatially homogeneous Gaussian noise (white in time), and introduced powerful techniques for analyzing these equations.
Since then, this area has been growing at an accelerated pace. One of the tools that has been used extensively is Malliavin calculus. This tool is especially useful when the noise is colored in time (or time-independent), and It\^o calculus techniques cannot be used, due to a lack of martingale structures. We refer the reader to
\cite{chen-dalang15,CDOT,CHKK,chen-kim,dalang-sanz09,FK09,hu-nualart09,HNS11,HHLNT,HHNT,millet-sanz21,sanz-sus13}
for a small sample of relevant contributions related to SPDEs with various types of Gaussian noise.

In \cite{HNV}, a new line of investigations has been opened up in this area, focusing on the asymptotic behaviour of the spatial integral of the solution of the stochastic heat equation with space-time Gaussian white noise, as the size of the integration region becomes large. The main result of \cite{HNV} states that, with suitable normalization and centering, this integral converges to the standard normal distribution, and gives the speed of this convergence in the total variation distance. This result, called the ``Quantitative Central Limit Theorem'' (QCLT) is obtained by combining Malliavin calculus with Stein's method for normal approximations. Similar results and extensions have been obtained in the subsequent papers \cite{HNVZ,NZ20,NXZ,BY2022} for the solution of the stochastic heat equation with colored noise in space/time (or with time-independent noise), respectively in \cite{BNZ,DNZ,BNQZ,BY} for the solution of the stochastic wave equation. In both cases, the noise enters the equation multiplied by a Lipschitz function $\sigma(u)$ of the solution.
The most difficult case is when the noise is {\em rough in space}, i.e. it behaves in space like the fractional Brownian motion (fBm) with Hurst index $H<1/2$. This case has been studied in \cite{NXZ} for the {\em parabolic Anderson model} (pAm), the stochastic heat equation with a linear term $\sigma(u)=u$ multiplying the noise, using technical arguments that rely heavily on properties of the heat kernel.

The goal of the present article is to present the first study of this problem for the wave equation with rough noise in space. We will assume that the noise is time-independent, and we postpone the treatment of the time-dependent noise for future work. More precisely, in this article we consider the {\em hyperbolic Anderson model} (hAm) in dimension 1:
\begin{align}
\label{HAM}
	\begin{cases}
		\dfrac{\partial^2 u}{\partial t^2} (t,x)
		= \dfrac{\partial^2 u}{\partial x^2} (t,x) + \sqrt{\theta} u(t,x) \dot{W}(x), \ t>0, \ x \in \bR, \\
		u(0,x) = 1, \ \dfrac{\partial u}{\partial t} (0,x) = 0.
	\end{cases}
\end{align}

Often, we are interested in the case $\theta=1$. The reason we introduce a parameter $\theta>0$ is the following. Our main result, the QCLT for the spatial average of the solution $u_{\theta}$ of \eqref{HAM}, is obtained by applying a version of the second-order Gaussian Poincar\'e inequality (Proposition 2.4 of \cite{NXZ}). For this, we need to estimate the {\em fourth} moment of the increments of the Malliavin derivative $Du_{\theta}$ and of the rectangular increments of the second Malliavin derivative $D^2u_{\theta}$. The fact that we include a parameter $\theta>0$ in equation \eqref{HAM} allows us to compare the $p$-th moment (for $p>2$) of the solution $u_{\theta}$, of its Malliavian derivatives, or of their increments, with the {\em second} moment of the similar quantities corresponding to the solution $u_{(p-1)\theta}$, which are then treated using their chaos expansions. This comparison between moments plays a crucial role in the present article, and is derived using a hypercontractivity property for general SPDEs, which is of independent interest and is included in Appendix \ref{section-hyper}.

\medskip

The noise $W$ is time-independent and fractional in space with Hurst index $H \in (\frac{1}{4},\frac{1}{2})$, being given by a zero-mean Gaussian process $\{W(\varphi);\varphi \in \cD(\bR)\}$, defined on a complete probability space $(\Omega,\cF, \bP)$, with covariance
\[
\bE[W(\varphi)W(\psi)]
= c_H \int_{\bR} \cF \varphi(\xi) \overline{\cF \psi(\xi)}|\xi|^{1-2H}d\xi
=:\langle \varphi,\psi\rangle_{\cP_0}.
\]
Here
$\cD(\bR)$ is the space of $C^{\infty}$ functions with compact support in $\bR$, $\cF \varphi(\xi)=\int_{\bR}e^{-i\xi x}\varphi(x)dx$ is the Fourier transform of $\varphi$, and
\[
c_H=\frac{\Gamma(2H+1)\sin(\pi H)}{2\pi}.
\]
By approximation, the noise can be extended to an isonormal Gaussian process $\{W(\varphi);\varphi \in \cP_0\}$, as defined in Malliavin calculus, where $\cP_0$ is the Hilbert space defined as completion of $\cD(\bR)$ with respect to the inner product $\langle \cdot, \cdot \rangle_{\cP_0}$. In particular, $1_{[0,x]}\in \cP_0$ for all $x\in \bR$, and the process $\{W(x)=W(1_{[0,x]});x \in \bR\}$ is a fBm of index $H$.

\medskip

We say that a process $u_{\theta}=\{u_{\theta}(t,x);t\geq 0,x \in \bR\}$ is a (mild Skorohod) {\bf solution} to equation \eqref{HAM} if it satisfies the following integral equation:
\[
u_{\theta}(t,x)=1+ \sqrt{\theta} \int_0^t \int_{\bR}G_{t-s}(x-y)u_{\theta}(s,y)W(\delta y)ds
\]
where the $W(\delta y)$ corresponds to the divergence operator $\delta$ (or Skorohod integral) defined in Section \ref{section-D} below, and $G_t$ is the fundamental solution of the wave equation on $\bR_{+} \times \bR$:
\begin{align}
\label{def-G}
	G_t(x) :=\dfrac{1}{2} 1_{\{ | x| < t\}}, \quad t>0,x\in \bR.
\end{align}
We let $G_t(x) =0$ when $t\leq 0$.
Note that the Fourier transform of $G_t$ is given by:
\begin{align}
\label{Fourier}
\cF G_t(\xi) = \int_{\bR}e^{-i\xi x} G_t(x)dx=\dfrac{\sin(t|\xi|)}{|\xi|}, \quad t\ge 0.
\end{align}

We are interested on the asymptotic behaviour as $R \to \infty$ of the spatial average:
\[
F_{R,\theta}(t)=\int_{-R}^R \big(u_{\theta}(t,x)-1\big)dx.
\]
More precisely, letting $\sigma_{R,\theta}^2(t)={\rm Var}\big(F_{R,\theta}(t)\big)$, we would like to show that
\[
\frac{F_{R,\theta}(t)}{\sigma_{R,\theta}(t)} \stackrel{d}{\to} Z \sim N(0,1) \quad \mbox{as $R \to \infty$},
\]
and to give an estimate for the speed of this convergence in the total variation distance:
\[
d_{\rm TV}(X,Y)=\sup_{B \in \cB(\bR)}|\bP(X \in B)-\bP(Y\in B)|.
\]

The following theorems are the main results of this article.
%They give respectively, the order of magnitude of the %covariance of $F_{R,\theta}(t)$, the estimate for the total %variation distance, and the corresponding functional limit %result.

\begin{theorem}[Limiting Covariance]
\label{limit-cov}
For any $\theta>0$, $t>0$ and $s>0$
\begin{equation}
\label{lim-cov}
\lim_{R \to \infty}\frac{\bE[F_{R,\theta}(t) F_{R,\theta}(s)]}{R}=K_{\theta}(t,s),
\end{equation}
where $K_{\theta}(t,s)$ is finite and is given by \eqref{def-K} below. Moreover, for any $\theta>0$ and $t>0$,
\begin{equation}
\label{lim-cov1}
\lim_{R \to \infty}\frac{\sigma_{R,\theta}^2(t)}{R}=K_{\theta}(t,t)>0.
\end{equation}
%{\red where
%\[
%\rho_{t,s}(x-y) = \bE \left[ \big(u(t,x)-1\big) %\big(u(s,y)-1\big) \right]
%\]
%is a quantity depends on $x,y$ only via their difference.}
\end{theorem}

\begin{theorem}[QCLT]
\label{QCLT}
For any $\theta>0$ and $t>0$,
\[
d_{\rm TV}\left( \frac{F_{R,\theta}(t)}{\sigma_{R,\theta}(t)},Z\right) \leq C_{t,H,\theta} R^{-1/2},
\]
where $C_{t,H,\theta}>0$ is a constant depending on $(t,H,\theta)$ and $Z \sim N(0,1)$.
\end{theorem}

\begin{theorem}[FCLT]
\label{FCLT}
For any $\theta>0$, the process $\{R^{-1/2} F_{R,\theta}(t); t\geq 0\}$ has a continuous modification which converges in distribution in $C([0,\infty))$ as $R \to \infty$, to a zero-mean Gaussian process $\{\cG_{\theta}(t);t\geq 0\}$ with covariance
\[
\bE[\cG_{\theta}(t)\cG_{\theta}(s)]=K_{\theta}(t,s),
\]
where $C([0,\infty))$ is equipped with the topology of uniform convergence on compact sets.
\end{theorem}

For the proof of Theorem \ref{limit-cov} we use the Wiener chaos decomposition of $F_{R,\theta}$, by observing that the projection on the first chaos space does not contribute to the limit, while the projections on the other chaos spaces give rise to a convergent series. A similar phenomena has been observed in \cite{NSZ} for (pAm) with rough noise in space, that is colored in time.  This is in contrast with the QCLT for equations with ``regular'' noise in space (studied in \cite{NZ20,NXZ,BY}) for which only the projection on the first chaos space contributes to the limit. In both cases, the QCLT is {\em non-chaotic}, in the sense described on page 3 of \cite{DNZ}, which means that not all the projections on the chaos spaces contribute to the limit.

Theorem \ref{QCLT} follows by applying a version of the second-order Gaussian Poincar\'e inequality
(Proposition 2.4 of \cite{NXZ}) for the time-independent noise. This leads to a study of the increments of $Du_{\theta}$, as well as the rectangular increments of $D^2u_{\theta}$. Analyzing these increments requires significant effort. For this task, we use a different method than in \cite{NXZ} for (pAm). In addition to the moment comparison technique mentioned above, we develop a connection with the solution $v_{\theta}$ of (hAm) with delta initial velocity. This introduces a new problem, which leads us to examine the increments of $u_{\theta}$ and $v_{\theta}$. A key role in this analysis is played by some translation-invariance properties of these processes. In addition, we exploit the fact that the increments of $v_{\theta}$ are closely related to those of the solution $V_{\theta}$ of (hAm) driven by a Gaussian noise $\fX$, which is {\em white in time} and fractional in space with the same index $H \in (\frac{1}{4},\frac{1}{2})$. The covariance of $\fX$ is given by \eqref{cov-fX} below.

For the proof of Theorem \ref{FCLT}, as usually, we need to show two things : (i) tightness; and (ii) finite-dimensional convergence. (i) follows by Kolmogorov-Centsov theorem, while (ii) follows using the same method as in \cite{BY}, based on a variance estimation that is proved using the same arguments as for the QCLT.

\medskip

This article is organized as follows. In Section \ref{section-prelim}, we include the background and some preliminary results from Malliavin calculus, and we prove the existence of the solution $u_{\theta}$. In Section \ref{section-wave-delta}, we examine some properties of $v_{\theta}$ and $V_{\theta}$.
In Section \ref{section-Malliavin}, we derive some estimates for the increments of $Du_{\theta}$ and the rectangular increments of $D^2u_{\theta}$. The increments of the processes $u_{\theta}$ and $v_{\theta}$ are studied in Section \ref{section-increm}. Finally, in Section \ref{section-proofs}, we present the proofs of Theorems \ref{limit-cov}, \ref{QCLT} and \ref{FCLT}.
The moment comparison result for a general SPDE is included
in Appendix \ref{section-hyper}. In Appendix \ref{section-appA}, we study two stochastic Volterra equations driven by $\fX$, which are of the same form as the equations satisfied by the increments of $V_{\theta}$, and allow us to derive some properties of these increments. Appendix \ref{section-appC-ineq} contains some elementary inequalities which are used in the sequel.

\section{Preliminaries}
\label{section-prelim}

In this section, we include some preliminary results related to Malliavin calculus, and we prove the existence and uniqueness of the solution to \eqref{HAM}.
We let $\|\cdot\|_p$ be the $L^p(\Omega)$-norm.

\subsection{Multiple Wiener integrals}

We start by introduce some basic ingredients of Malliavin calculus, following \cite{nualart06}.

We denote by $J_n$ the projection from $L^2(\Omega)$ to $\cP_{0,n}$, where $\cP_{0,n}$ is the $n$-th Wiener chaos space, defined at the closed linear span in $L^2(\Omega)$ of $\{H_n(W(\varphi);\varphi \in \cP_0)\}$, with $H_n(x)$ being the $n$-th Hermite polynomial.
Let $I_n:\cP_0^{\otimes n} \to \cP_{0,n}$ be
the multiple Wiener integral of order $n$ with respect to $W$,
where $\cP_0^{\otimes n}$ is the $n$-th tensor product of $\cP_0$. Since $I_n$ is surjective, for any $F \in L^2(\Omega)$, $J_n(F)=I_n(f_n)$
for some $f_n \in \cP_0^{\otimes n}$.

The Wiener chaos spaces $\cP_{0,n}$ are orthogonal in $L^2(\Omega)$. More precisely, for any $f \in \cP_{0}^{\otimes n}$ and $g \in \cP_{0}^{\otimes m}$,
\[
 E[I_n(f)I_m(g)]=\left\{
\begin{array}{ll} n! \, \langle \widetilde{f}, \widetilde{g} \rangle_{\cP_{0}^{\otimes n}} & \mbox{if $n=m$} \\
0 & \mbox{if $n \not=m$}
\end{array} \right.
\]
where $\widetilde{f}$ is the symmetrization of $f$ in all $n$
variables:
$$\widetilde{f}(x_1,\ldots,x_n)=\frac{1}{n!}\sum_{\rho \in S_n}f(x_{\rho(1)},\ldots,x_{\rho(n)}),$$
and $S_n$ is the set of all permutations of $\{1,
\ldots,n\}$. Note that $I_n(\widetilde f)=I_n(f)$ and
\begin{equation}
\label{rough-bound}
\|\widetilde f\|_{\cP_0^{\otimes n}} \leq \|f\|_{\cP_0^{\otimes n}}.
\end{equation}

Every random variable $F \in L^2(\Omega)$
which is measurable with respect to $W$ has the {\em Wiener chaos expansion}:
\begin{equation}
\label{F-chaos}
F=E(F)+\sum_{n \geq 1}J_n(F) =E(F)+\sum_{n \geq 1}I_n(f_n),
\end{equation}
where the series converges in $L^2(\Omega)$.
Moreover,
$$E|F|^2=\sum_{n \geq 0}E|I_n(f_n)|^2=\sum_{n \geq 0}n! \, \|\widetilde{f}_n\|_{\cP_0^{\otimes n}}^{2},$$
where $I_0$ is the identity map on $\bR$ and $f_0=E(F)$. Occasionally, we denote $J_0(F)=\bE(F)$.

\medskip

In general, $I_{p+q} (f\otimes g) \not= I_p(f) I_q(g)$. The exact expression of $I_{p+q} (f\otimes g)$ is given by {\em the product formula}, which will not be used in the present paper. Instead, we use the following result.

\begin{lemma}
\label{Lemma-decomp-I}
Let $p,q \in \bN$.
\begin{enumerate}
	\item[(a)] For $f \in \cP_0^{\otimes p}$, $g \in \cP_0^{\otimes q}$, we have
	\begin{align*}
		\left\| I_{p+q}(f \otimes g) \right\|_2^2
		\le \dfrac{(p+q)!}{p!q!} \left\| I_p(f) \right\|_2^2 \left\| I_q(g) \right\|_2^2.
	\end{align*}

	\item[(b)] For $f_1 \in \cP_0^{\otimes p_1}$, $f_2 \in \cP_0^{\otimes p_2}$ and $f_3 \in \cP_0^{\otimes p_3}$, we have
	\begin{align*}
		\| I_{p_1+p_2+p_3}(f_1 \otimes f_2 \otimes f_3) \|_2^2 \le \dfrac{(p_1+p_2+p_3)!}{p_1!p_2!p_3!} \| I_{p_1}(f_1) \|_2^2 \| I_{p_2}(f_2) \|_2^2 \| I_{p_3}(f_3) \|_2^2.
	\end{align*}
\end{enumerate}
\end{lemma}

\begin{proof}
(a) Since $f \otimes g$ and $\tilde f \otimes \tilde g$ have the same symmetrization,
%Indeed, using the notation $f(\sigma) = f(x_{\sigma(1)}, %\ldots, x_{\sigma(p)})$ and $g(\tau) = g(x_{p+\tau(1)}, \ldots, %x_{p+\tau(q)})$ for $\sigma \in S_p$ and $\tau \in S_q$, we %have
%\begin{align*}
%	\widetilde{ \tilde{f} \otimes \tilde{g} }
%	= \dfrac{1}{p!q!} \widetilde {\sum_{\sigma \in S_p, \tau \in %S_q} f(\sigma) g(\tau)}
%	= \dfrac{1}{p!q!(p+q)!} \sum_{\pi \in S_{p+q}} \sum_{\sigma %\in S_p, \tau \in S_q} f(\pi \circ \sigma) g(\pi \circ \tau).
%\end{align*}
%where $f(\pi \circ \sigma) = f(x_{\pi(\sigma(1))}, \ldots, %x_{\pi(\sigma(p))})$ and $g(\pi \circ \tau) = %g(x_{\pi(p+\tau(1))}, \ldots, x_{\pi(p+\tau(q))})$. Note that %for all $\sigma \in S_p$ and $\tau \in S_q$, we have
%\begin{align*}
%	\sum_{\pi \in S_{p+q}} f(\pi \circ \sigma) g(\pi \circ \tau) %= \sum_{\pi \in S_{p+q}} (f \otimes g)(\pi).
%\end{align*}
$I_{p+q} (\tilde{f} \otimes \tilde{g}) = I_{p+q}(f \otimes g)$.
Hence,
\begin{align*}
	&\left\| I_{p+q}(f \otimes g) \right\|_2^2
	= \left\| I_{p+q} (\tilde{f} \otimes \tilde{g}) \right\|_2^2
	= (p+q)! \left\| \widetilde{ \tilde{f} \otimes \tilde{g} } \right\|_{\cP_0^{\otimes (p+q)}}^2 \\
	\le& (p+q)! \left\| \tilde{f} \otimes \tilde{g} \right\|_{\cP_0^{\otimes (p+q)}}^2
	= (p+q)! \int_{\bR^{p+q}} \left| \cF \left( \tilde{f} \otimes \tilde{g} \right) (\xi_1, \ldots, \xi_{p+q}) \right|^2 \prod_{j=1}^{p+q} |\xi_j|^{1-2H} d\xi_j \\
	=& (p+q)! \int_{\bR^{p+q}} \left| \cF \tilde{f} (\xi_1, \ldots, \xi_p) \right|^2 \left| \cF \tilde{g} (\xi_{p+1}, \ldots, \xi_{p+q}) \right|^2 \prod_{j=1}^{p+q} |\xi_j|^{1-2H} d\xi_j \\
	=& \dfrac{(p+q)!}{p!q!} \| \tilde f \|_{\cP_0^{\otimes p}}^2 \| \tilde g \|_{\cP_0^{\otimes q}}^2
	= \dfrac{(p+q)!}{p!q!} \left\| I_p(f) \right\|_2^2 \left\| I_q(g) \right\|_2^2
\end{align*}

(b) We use (a) twice, the first time for $p=p_1$, $q=p_2+p_3$, $f=f_1$, $g=f_2 \otimes f_3$, and the second time for $p=p_2$,  $q=p_3$, $f=f_2$, $g=f_3$.
\end{proof}

Next, we extend Lemma \ref{Lemma-decomp-I} to processes indexed by elements in a measure space.

\begin{lemma}
\label{Lemma-decomp-F}
Let $(E,\cal E,\mu)$ be a measure space. For each $\theta>0$ and $r\in E$, let $F_{\theta}(r), G_{\theta}(r)$ be random variables in $L^2(\Omega)$ with the Wiener chaos expansions
$F_{\theta}(r) = \sum_{p\ge0} \theta^{p/2} I_p(f_p(\cdot,r))$ and
$G_{\theta}(r) = \sum_{q\ge0} \theta^{q/2} I_q(g_q(\cdot,r))$.
%\begin{align*}
%	F_{\theta}(r) = \sum_{p\ge0} \theta^{p/2} I_p(f_p(\cdot,r)), %\quad \quad \quad
%	G_{\theta}(r) = \sum_{q\ge0} \theta^{q/2} I_q(g_q(\cdot,r)),
%\end{align*}
%where $f_p(\cdot,r) \in \cP_0^{\otimes p}$, $g_q(\cdot,r) \in \cP_0^{\otimes q}$ for $p,q \in \bN$.
Then
\begin{align*}
	\sum_{p\ge0} \sum_{q\ge0} \theta^{p+q} \left\| I_{p+q} \left( \int_E f_p(\cdot,r) \otimes g_q(\cdot,r) \mu(dr) \right) \right\|_2^2
	\le \left( \int_E \left\| F_{2\theta}(r) \right\|_2 \left\| G_{2\theta}(r) \right\|_2 \mu(dr) \right)^2.
\end{align*}
Moreover, if for each $\theta>0$ and $r\in E$, $H_{\theta}(r)$ is a random variable in $L^2(\Omega)$ with the Wiener chaos expansion
$H_{\theta}(r) = \sum_{k\ge 0} \theta^{k/2} I_k(h_k(\cdot,r))$,
then
\begin{align*}
&\sum_{p\ge0} \sum_{q\ge0} \sum_{k\ge0} \theta^{p+q+k} \left\| I_{p+q+k} \left( \int_E f_p(\cdot,r) \otimes g_q(\cdot,r) \otimes h_k(\cdot,r) \mu(dr) \right) \right\|_2^2 \\
& \qquad \qquad \qquad \le \left( \int_E \left\| F_{3\theta}(r) \right\|_2 \left\| G_{3\theta}(r) \right\|_2 \left\| H_{3\theta}(r) \right\|_2 \mu(dr) \right)^2.
\end{align*}
\end{lemma}

\begin{proof}
We use the same idea as in the proof of Lemma \ref{Lemma-decomp-I}.(a). More precisely,
\begin{align*}
&\left\| I_{p+q} \left( \int_E f_p(\cdot,r) \otimes g_q(\cdot,r) \mu(dr) \right) \right\|_2^2
	= \left\| I_{p+q} \left( \int_E \tilde f_p(\cdot,r) \otimes \tilde g_q(\cdot,r) \mu(dr) \right) \right\|_2^2 \\
& \quad	\leq (p+q)! \left\| \int_E \tilde f_p(\cdot,r) \otimes \tilde g_q(\cdot,r) \mu(dr) \right\|_{\cP_0^{\otimes (p+q)}} \\
& \quad = (p+q)! \int_{\bR^{p+q}} \left| \int_E \cF \big( \tilde f_p(\cdot,r) \otimes \tilde g_q(\cdot,r) \big) (\xi_1,\ldots,\xi_{p+q}) dr \right|^2 \prod_{j=1}^{p+q} |\xi_j|^{1-2H} d\xi_j \\
& \quad = (p+q)! \int_{\bR^{p+q}} \int_{E^2} \cF \tilde f_p(\cdot,r_1) (\xi_1,\ldots,\xi_p) \cF \tilde g_q(\cdot,r_1) (\xi_{p+1},\ldots,\xi_{p+q}) \\
& \quad \quad \quad \times \overline{\cF \tilde f_p(\cdot,r_2) (\xi_1,\ldots,\xi_p)} \overline{\cF \tilde g_q(\cdot,r_2) (\xi_{p+1},\ldots,\xi_{p+q})} \mu(dr_1)\mu(dr_2)
	\prod_{j=1}^{p+q} |\xi_j|^{1-2H} d\xi_j \\
& \quad = (p+q)! \int_{E^2} \left\langle \tilde f_p(\cdot,r_1), \tilde f_p(\cdot,r_2) \right\rangle_{\cP_0^{\otimes p}} \left\langle \tilde g_q(\cdot,r_1), \tilde g_q(\cdot,r_2) \right\rangle_{\cP_0^{\otimes q}} \mu(dr_1)\mu(dr_2) \\
& \quad = \dfrac{(p+q)!}{p!q!} \int_{E^2} \bE\left[ I_p \left( f_p(\cdot,r_1) \right) I_p \left( f_p(\cdot,r_2) \right) \right] \bE\left[ I_q \left( g_q(\cdot,r_1) \right) I_q \left( g_q(\cdot,r_2) \right) \right] \mu(dr_1)\mu(dr_2).
\end{align*}
Noting that $\frac{(p+q)!}{p!q!} \le 2^{p+q}$, we obtain:
\begin{align*}
	& \sum_{p\ge1} \sum_{q\ge1} \theta^{p+q} \left\| I_{p+q} \left( \int_E f_p(\cdot,r) \otimes g_q(\cdot,r) \mu(dr) \right) \right\|_2^2 \\
	\le& \int_{E^2} \sum_{p\ge1} (2\theta)^p \bE\left[ I_p \left( f_p(\cdot,r_1) \right) I_p \left( f_p(\cdot,r_2) \right) \right] \sum_{q\ge1} (2\theta)^q \bE\left[ I_q \left( g_q(\cdot,r_1) \right) I_q \left( g_q(\cdot,r_2) \right) \right] dr_1dr_2 \\
	=& \int_{E^2} \bE \left[ F_{2\theta}(r_1) F_{2\theta}(r_2) \right] \bE \left[ G_{2\theta}(r_1) G_{2\theta}(r_2) \right] \mu(dr_1)\mu(dr_2)
	\le \left( \int_E \left\| F_{2\theta}(r) \right\|_2 \left\| G_{2\theta}(r) \right\|_2 \mu(dr) \right)^2,
\end{align*}
where we used the Cauchy-Schwarz inequality for the last inequality.

The second estimate is proved in a similar way, using the inequality $\frac{(p+q+k)!}{p!q!k!} \le 3^{p+q+k}$.
\end{proof}

\subsection{The operators $D$ and $\delta$}
\label{section-D}

In this section, we recall briefly the definitions of the Malliavin derivative $D$ and the divergence operator $\delta$. We refer the reader to \cite{nualart06} for more details.

Let $\cS$ be the class of ``smooth'' random variables, i.e variables of the form
\begin{equation}
\label{form-F}F=f(W(\varphi_1),\ldots, W(\varphi_n)),
\end{equation} where $f \in C_{b}^{\infty}(\bR^n)$, $\varphi_i \in \cP_0$, $n \geq 1$, and
$C_b^{\infty}(\bR^n)$ is the class of bounded $C^{\infty}$-functions
on $\bR^n$, whose partial derivatives of all orders are bounded.
The
{\em Malliavin derivative} of a random variable $F$ of the form (\ref{form-F}) is the
$\cP_0$-valued random variable given by:
$$DF:=\sum_{i=1}^{n}\frac{\partial f}{\partial x_i}(W(\varphi_1),\ldots,
W(\varphi_n))\varphi_i.$$ We endow $\cS$ with the norm
$\|F\|_{\bD^{1,2}}:=(E|F|^2)^{1/2}+(E\|D F \|_{\cP_0}^{2})^{1/2}$. The
operator $D$ can be extended to the space $\bD^{1,2}$, the
completion of $\cS$ with respect to $\|\cdot \|_{\bD^{1,2}}$.

For any integer $k\geq 2$, we let $D^k$ be the Malliavin derivative of order $k$, whose domain is denoted by $\bD^{k,2}$.

If $F$ is a random variable in $L^2(\Omega)$ with the Wiener chaos expansion $F=\sum_{n\geq 0}I_n(f_n)=\sum_{n\geq 0}J_n(F)$ for some {\em symmetric} functions $f_n \in \cP_0^{\otimes n}$, then
\begin{equation}
\label{Mal-F}
D_x F=\sum_{n\geq 1}nI_{n-1}(f_n(\cdot,x)).
\end{equation}
In particular, $D_x J_n(F)=nI_{n-1}(f_n(\cdot,x))=J_{n-1}(D_x F)$ for any $n\geq 1$, and $D_x J_0(F)=0$. In general,
\begin{equation}
\label{D-k-J}
D^k J_n(F)=\left\{
\begin{array}{ll}
J_{n-k}(D^k F) & \mbox{if $n\geq k$} \\
0 & \mbox{if $n<k$}
\end{array} \right.
\end{equation}

\medskip

The {\em divergence operator} $\delta$ is the adjoint of
the operator $D$. The domain of $\delta$, denoted by $\mbox{Dom} \
\delta$, is the set of $u \in L^2(\Omega;\cP_0)$ such that
$$|E \langle DF,u \rangle_{\cH}| \leq c (E|F|^2)^{1/2}, \quad \forall F \in \bD^{1,2},$$
where $c$ is a constant depending on $u$. If $u \in {\rm Dom} \
\delta$, then $\delta(u)$ is the element of $L^2(\Omega)$
characterized by the following duality relation:
\begin{equation}
\label{duality}
E(F \delta(u))=E\langle DF,u \rangle_{\cP_0}, \quad
\forall F \in \bD^{1,2}.
 \end{equation}
In particular, $E(\delta(u))=0$. If $u \in \mbox{Dom} \ \delta$, we
use the notation
$$\delta(u)=\int_{\bR}u(x) W(\delta x),$$
and we say that $\delta(u)$ is the {\em Skorohod integral} of $u$ with respect to $W$.

\subsection{OU semigroup and hypercontractivity}

In this section, we review the definition and some properties of the Ornstein-Uhlenbeck semigroup.

The {\em Ornstein-Uhlenbeck (OU) semigroup} is the family $(T_{t})_{t \ge 0}$ of contraction operators on $L^2(\Omega)$ defined by:
\begin{align*}
	T_{t} (F) = \sum_{n\geq 0} e^{-n t} J_n(F).
\end{align*}

Using \eqref{D-k-J}, we derive that
\begin{align}
\label{Prop-OU}
	D^k \left( T_{t}(F) \right)
	= \sum_{n\geq 0} e^{-nt} D^k(J_n(F))
	= \sum_{n\geq k} e^{-nt} J_{n-k}(D^kF)
	= e^{-kt} T_{t} (D^kF).
\end{align}

The OU semigroup has a {\em hypercontractivity property}: for any $t>0$ and $F \in L^p(\Omega)$,
\begin{equation}
\label{OU-hyper}
\|T_t F\|_{q(t)} \leq \|F\|_p, \quad \mbox{where $q(t)=e^{2t}(p-1)+1>p>1$}.
\end{equation}
An important consequence of this property is that for any $1<p<q<\infty$, the norms $\|\cdot\|_p$ and $\|\cdot\|_q$ are equivalent on the same chaos space $\cP_{0,n}$. More precisely,
\begin{equation}
\label{hyper}
\|F\|_q \leq \left( \frac{q-1}{p-1}\right)^{n/2}\|F\|_p, \quad \mbox{for any $F \in \cP_{0,n}$ and $1<p<q$}.
\end{equation}

The generator of the Ornstein-Uhlenbeck semigroup is given by:
\begin{equation}
\label{def-L}
LF=\sum_{n\geq 1}n J_n(F).
\end{equation}
Its domain ${\rm Dom}\ L$ consists of random variables $F \in L^2(\Omega)$ for which the series in \eqref{def-L} converges in $L^2(\Omega)$.
Note that
$F \in {\rm Dom}\ L$ if and only if $F \in \bD^{1,2}$ and $DF \in {\rm Dom} \ \delta$, and in this case, $LF=-\delta (D F)$.

The {\em pseudo-inverse} of $L$ is the operator $L^{-1}$ defined by
\[
L^{-1}F=\sum_{n\geq 1}\frac{1}{n} J_n(F).
\]
For any $F \in \bD^{1,2}$ with $\bE(F)=0$, the process $u=-D L^{-1}F$ belongs to ${\rm Dom} \ \delta$
and
\begin{equation}
\label{F-deltaD}
F=\delta(-D L^{-1} F).
\end{equation}
%(see e.g. Proposition 6.5.1 of \cite{NN}).

\medskip

\subsection{Existence of solution}

In this section, we show that equation \eqref{HAM} has a unique solution.

It is known that, if it exists, the solution to equation \eqref{HAM} has the Wiener chaos expansion:
\begin{equation}
\label{series-conv-u}
u_{\theta}(t,x)=1+\sum_{n\geq 1} \theta^{n/2} I_n(f_n(\cdot,x;t)),
\end{equation}
with kernel $f_n(\cdot,x;t)$ given by:
\begin{align}
\label{def-fk}
	f_n(x_1,\ldots,x_n,x;t)
	=& \int_{T_n(t)} f_n(t_1,x_1,\ldots,t_n,x_n,t,x) d\pmb{t},
\end{align}
where $T_n(t)=\{0<t_1<\ldots<t_n<t\}$ and
\begin{equation}
\label{def-fn-time-dep}
f_n(t_1,x_1,\ldots,t_n,x_n,t,x)=G_{t-t_n}(x-x_n)\ldots G_{t_2-t_1}(x_2-x_1) 1_{\{0<t_1<\ldots<t_n<t\}}.
\end{equation}
is the kernel which appears in the chaos decomposition of the solution of (hAm) with time dependent noise (studied in \cite{BNQZ}). In this case,
\begin{equation}
\label{conv-ser-u-1}
\bE|u_{\theta}(t,x)|^2
=1+\sum_{n\geq 1} n! \theta^n \|\widetilde{f}_n(\cdot,x;t)\|_{\cP_0^{\otimes n}}^2.
\end{equation}
In fact, the solution exists if and only if the series \eqref{series-conv-u} converges in $L^2(\Omega)$, i.e. the series \eqref{conv-ser-u-1} converges. In this case, the solution is unique.

The next result gives the existence of the solution and provides an upper bound for its moments. For this result, we will use the fact that for any $a>0$, there exist some constants $c_{1}>0$ and $c_2>0$ depending on $a$ such that
\begin{equation}
\label{est1}
\sum_{n\geq 0}\frac{x^n}{(n!)^a} \leq c_{1}\exp(c_{2} x^{1/a}) \quad \mbox{for any $x>0$}.
\end{equation}

In addition, we will need the value of the following integral: for any $t>0$,
\begin{equation}
\label{G-int}
\int_{\bR}|\cF G_t(\xi)|^2 |\xi|^{\alpha}d\xi=\int_{\bR}
\frac{\sin^2(t|\xi|)}{|\xi|^2} |\xi|^{\alpha}d\xi=c_{\alpha}t^{1-\alpha}.
\end{equation}
The integral converges if and only if $\alpha \in (-1,1)$; see for instance, Lemma 3.1 of \cite{BJQ15}.

\begin{theorem}
\label{Thm-exist-u}
For any $H \in (\frac{1}{4},\frac{1}{2})$, equation \eqref{HAM} has a unique (Skorohod) solution. Moreover, for any $\theta>0$, $p \geq 2$, $t>0$ and $x \in \bR$,
\begin{align}
\label{mom-p-u}
	\|u_{\theta}(t,x)\|_p \leq c_1\exp\left(c_2 \theta^{\frac{1}{2H+1}}p^{\frac{1}{2H+1}}  t^{\frac{2H+2}{2H+1}}\right),
\end{align}
where $c_1>0$ and $c_2>0$ are constants that only depend on $H$. Consequently, $C_{\theta,p,T,u} := \sup_{(t,x) \in [0,T] \times \bR} \|u_{\theta}(t,x)\|_p<\infty$.
\end{theorem}

\begin{proof}
Note that $\cF f_n(t_1,\cdot,\ldots,t_n,\cdot,t,x)(\pmb{\xi})=
e^{-i(\xi_1+\ldots+\xi_n)x}\prod_{j=1}^n\cF G_{t_{j+1}-t_j}(\xi_1+\ldots+\xi_j)$, where $t_{n+1}=t$ and
$\pmb{\xi}=(\xi_1,\ldots,\xi_n)$.
By Cauchy-Schwarz inequality, we have
\begin{align*}
	\left| \cF f_n(\cdot,x;t) (\pmb{\xi}) \right|^2
	& = \left| \int_{T_n(t)} \prod_{j=1}^n \cF G_{t_{j+1}-t_j} (\xi_1 + \ldots + \xi_j) d\pmb{t} \right|^2 \\
	& \leq \dfrac{t^n}{n!} \int_{T_n(t)} \prod_{j=1}^n \left| \cF G_{t_{j+1}-t_j} (\xi_1 + \ldots + \xi_j) \right|^2 d\pmb{t},
\end{align*}
where $\pmb{t}=(t_1,\ldots,t_n)$. By Fubini's theorem,
\begin{align*}
%\label{ineq-norm-f_n}
	\|f_n(\cdot,x;t)\|_{\cP_0^{\otimes n}}^2
	&= c_H^n \int_{\bR^n} \left| \cF f_n(\cdot,x;t) (\pmb{\xi}) \right|^2 \prod_{j=1}^n |\xi_j|^{1-2H} d\pmb{\xi} \nonumber \\
	&\le  c_H^n \dfrac{ t^n }{n!} \int_{T_n(t)} \int_{\bR^n} \prod_{j=1}^n \left| \cF G_{t_{j+1}-t_j} (\xi_1 + \ldots + \xi_j) \right|^2 \prod_{j=1}^n |\xi_j|^{1-2H}  d\pmb{\xi} d\pmb{t} \nonumber \\
	&= c_H^n \dfrac{  t^n }{n!} \int_{T_n(t)} \int_{\bR^n} \prod_{j=1}^n \left| \cF G_{t_{j+1}-t_j} (\eta_j) \right|^2 |\eta_1|^{1-2H}\prod_{j=2}^n |\eta_j-\eta_{j-1}|^{1-2H} d{\pmb \eta} d\pmb{t},
\end{align*}
where for the last line, we used the change of variables $\eta_j = \xi_1 + \ldots + \xi_j$ for $j=1,\ldots,n$, and we denoted $\pmb{\eta}=(\eta_1,\ldots,\eta_n)$.
We now use the inequality $|a+b|^{1-2H} \leq |a|^{1-2H} + |b|^{1-2H}$, followed by the identity:
\[
x_1 \prod_{j=2}^n (x_j+x_{j-1})=\sum_{a \in A_n}x_j^{a_j},
\]
where $A_n$ is a set of multi-indices $a = (a_1, \ldots, a_n)$ with ${\rm card}(A_n) = 2^{n-1}$ such that $a_1 \in \{1,2\}$, $a_n \in \{0,1\}$, $a_2, \ldots, a_{n-1} \in \{0,1,2\}$ and $\sum_{j=1}^n a_j = n$.
% $a_j + a_{j-1} \in \{1,2,3\}$ for any $j=1,\ldots,n-1$.
The exact description of the set $A_n$ is given in Lemma 2.2 of \cite{BCM}. We obtain that:
\begin{align}
\label{prod}
	|\eta_1|^{1-2H}\prod_{j=2}^n |\eta_j-\eta_{j-1}|^{1-2H}
	\leq \sum_{a \in A_n} \prod_{j=1}^n |\eta_j|^{(1-2H)a_j}=\sum_{\alpha \in D_n} \prod_{j=1}^n |\eta_j|^{\alpha_j},
\end{align}
where $D_n$ is the set of multi-indices $\alpha=(\alpha_1,\ldots,\alpha_n)$ with $\alpha_j=(1-2H)a_j$ for all $j=1,\ldots,n$ and $a=(a_1,\ldots,a_n)\in A_n$. Therefore,
%Substituting \eqref{prod} to \eqref{ineq-norm-f_n}, we obtain:
\begin{align*}
	\|f_n(\cdot,x;t)\|_{\cP_0^{\otimes n}}^2
	%\le& \dfrac{t^n}{n!} \sum_{\alpha\in \cA_n} \int_{T_n(t)} \int_{\bR^n} %\prod_{j=1}^n \left| \cF G_{t_{j+1}-t_j} (\eta_j) \right|^2 \prod_{j=1}^n %|\eta_j|^{(1-2H)\alpha_j} d\eta_1\ldots d\eta_n d\pmb{t} \\
	&\leq c_H^n \dfrac{ t^n}{n!} \sum_{\alpha\in D_n} \int_{T_n(t)} \prod_{j=1}^n \left( \int_{\bR} \left| \cF G_{t_{j+1}-t_j} (\eta_j) \right|^2 |\eta_j|^{\alpha_j} d\eta_j \right) d\pmb{t}.
\end{align*}
The inner integral above is calculated using \eqref{G-int}. In our case, $\alpha_j=2(1-2H)<1$, since $H>1/4$. Using the rough bound \eqref{rough-bound}, we infer that:
\begin{align}
\label{bound-f}
\|\widetilde{f}_n(\cdot,x;t)\|_{\cP_0^{\otimes n}}^2 \leq \|f_n(\cdot,x;t)\|_{\cP_0^{\otimes n}}^2 \leq  C^n \dfrac{t^n}{n!} \sum_{\alpha\in D_n} \int_{T_n(t)} \prod_{j=1}^n(t_{j+1}-t_j)^{1-\alpha_j}d \pmb{t} \leq C^n \frac{t^{(2H+2)n}}{(n!)^{2H+2}},
\end{align}
where $C>0$ is a constant that depends on $H$ which is different in each of its appearances. Using Minkowski inequality, hypercontractivity property \eqref{hyper}, and \eqref{bound-f}, we obtain:
\begin{align*}
\|u_{\theta}(t,x)\|_p & \leq \sum_{n \geq 0} \theta^{n/2} (p-1)^{n/2}
\|I_n(f_n(\cdot,x;t))\|_2 =
\sum_{n \geq 0} \theta^{n/2} (p-1)^{n/2} (n!)^{1/2}
\|\widetilde{f_n}(\cdot,x;t)\|_{\cP_0^{\otimes n}} \\
&
%\leq \sum_{n \geq 0} (p-1)^{n/2}\theta^{n/2}(n!)^{1/2}
% \|f_n(\cdot,x;t)\|_{\cP_0^{\otimes n}}
\leq \sum_{n\geq 0} \theta^{n/2} (p-1)^{n/2}C^{n/2}\frac{t^{(H+1)n}}{(n!)^{H+1/2}}.
\end{align*}
Relation \eqref{mom-p-u} now follows using \eqref{est1}.
\end{proof}

\section{Wave equation with delta initial condition}
\label{section-wave-delta}

In this section, we study the (hAm) on $[r,\infty)$ with zero initial condition  initial velocity given by the measure $\delta_z$. The relation between the solution of this model and the Malliavin derivative $Du_{\theta}$ plays an important role in the present article.

\medskip

For any $\theta>0$, $r>0$ and $z \in \bR$ fixed, we consider the following equation:
\begin{align}
\label{eq-v}
	\begin{cases}
		\dfrac{\partial^2 v}{\partial t^2} (t,x)
		= \dfrac{\partial^2 v}{\partial x^2} (t,x) + \sqrt{\theta}v(t,x) \dot{W}(x), \
		t>r, \ x \in \bR\\
		v(r,\cdot) = 0, \ \dfrac{\partial v}{\partial t} (r,\cdot) = \delta_z.
	\end{cases}
\end{align}
We say that $v_{\theta}^{(r,z)}$ is a (mild Skorohod) {\em solution} of \eqref{eq-v} if it satisfies the following equation:
\begin{equation}
\label{eq-v-int}
	v_{\theta}^{(r,z)}(t,x)=G_{t-r}(x-z)+\sqrt{\theta}\int_r^t \int_{\bR}G_{t-s}(x-y) v_{\theta}^{(r,z)}(s,y)W(\delta y) ds, \quad t \geq r.
\end{equation}

%We denote $v^{(r,z)}=v_{1}^{(r,z)}$.
It can be proved that if a solution $v_{\theta}^{(r,z)}$ exists, then it is unique and it has the chaos expansion:
\begin{align}
\label{chaos-v}
	v_{\theta}^{(r,z)}(t,x)
	= G_{t-r}(x-z)+\sum_{n\geq 1} \theta^{n/2}I_n \left( g_n(\cdot,z,x;r,t) \right),
\end{align}
where
\begin{align}
\label{def-gk}
	g_n(x_1,\ldots,x_n,z,x;r,t)
	=& \int_{r<t_1<\ldots<t_n<t} g_n(t_1,x_1,\ldots,t_n,x_n,r,z,t,x) dt_1 \ldots dt_n,
\end{align}
and
\[
g_n(t_1,x_1,\ldots,t_n,x_n,r,z,t,x) =G_{t-t_n}(x-x_n) \ldots G_{t_1-r}(x_1-z)
\]
We denote $g_0(z,x;r,t)=g_0(r,z,t,x)=G_{t-r}(x-z)$. The necessary and sufficient condition for the existence of the solution $v_{\theta}^{(r,z)}$ is that the series in \eqref{chaos-v} converges in $L^2(\Omega)$, i.e.
\begin{align}
\label{series-conv-v}
\sum_{n\geq 1}\theta^n n!\|\widetilde{g}_n(\cdot,z,x;r,t)\|_{\cP_0^{\otimes n}}^2<\infty,
\end{align}
where $\widetilde{g}_n(\cdot,z,x;r,t)$ is the symmetrization of $g_n(\cdot,z,x;r,t)$.

Due to the special form \eqref{def-G} of $G$, $v_{\theta}^{(r,z)}(t,x)$ satisfies the following identity:
\begin{align}
\label{eq-identity-v}
	v_{\theta}^{(r,z)}(t,x) = 2G_{t-r}(x-z) v_{\theta}^{(r,z)}(t,x).
\end{align}

The existence of the solution $v_{\theta}$ and some of its properties follow from a connection with the (hAm) model with another Gaussian noise (called $\fX$), which is {\em white in time} and fractional in space (with the same Hurst index $H$ as the noise $W$):
\begin{align}
\label{eq-V}
	\begin{cases}
		\dfrac{\partial^2 V}{\partial t^2} (t,x)
		= \dfrac{\partial^2 V}{\partial x^2} (t,x) + \sqrt{\theta} V(t,x) \dot{\fX}(t,x), \
		t>r, \ x \in \bR\\
		V(r,\cdot) = 0, \ \dfrac{\partial V}{\partial t} (r,\cdot) = \delta_z.
	\end{cases}
\end{align}
More precisely, $\{\fX(\varphi);\varphi \in \cD(\bR_{+}\times \bR)\}$ is a zero-mean Gaussian process with covariance
\begin{equation}
\label{cov-fX}
\bE[\fX(\varphi)\fX(\psi)]=c_H\int_0^{\infty}\int_{\bR}\cF
\varphi(t,\cdot) \overline{\cF \psi(t,\cdot)(\xi)} |\xi|^{1-2H}d\xi dt=:\langle \varphi,\psi\rangle_{\cH_0},
\end{equation}
which can be extended to an isonormal Gaussian process $\fX=\{\fX(\varphi);\varphi \in \cH_0\}$, where $\cH_0$ is the completion of $\cD(\bR_{+}\times \bR)$ with respect to $\langle \cdot, \cdot \rangle_{\cH_0}$.
%It can be proved that $\cH_0=L^2(\bR_{+};\cP_0)$.

We say that $V_{\theta}^{(r,z)}$ is a {\em (mild Skorohod) solution} of \eqref{eq-V} if it satisfies:
% the following equation:
\begin{align}
\label{eq-V-int}
	V^{(r,z)}_{\theta}(t,x)
	=G_{t-r}(x-z) + \sqrt{\theta} \int_r^t \int_{\bR} G_{t-s}(x-y) V^{(r,z)}_{\theta}(s,y) \fX(ds,dy), \quad t \geq r.
\end{align}
%We denote $V^{(r,z)}=V_{1}^{(r,z)}$.
Using a similar method to the proof of Theorem 3.8 of \cite{BJQ15}, we show in Appendix \ref{section-appA} that equation \eqref{eq-V} has a unique solution $V_{\theta}^{(r,z)}$, and its moments are uniformly bounded:
\begin{equation}
\label{bound-V}
K_{\theta,T,p}:=\sup_{0\leq r<t\leq T} \sup_{x,z \in \bR}\|V_{\theta}^{(r,z)}(t,x)\|_p<\infty.
\end{equation}
See Example \eqref{ex-v}.
Moreover, the solution $V^{(r,z)}_{\theta}$ has the following chaos expansion
\begin{align}
\label{series-V}
	V^{(r,z)}_{\theta}(t,x) = G_{t-r}(x-z)+\sum_{n\geq 1} \theta^{n/2} I_n^{\fX}(g_n(\cdot,r,z,t,x)),
\end{align}
where $I_n^{\fX}$ is the $n$th multiple integral with respect to $\fX$. Since the solution $V^{(r,z)}_{\theta}$ exists, the series \eqref{series-V} converges in $L^2(\Omega)$. i.e.
\begin{equation}
\label{series-V-conv}
\sum_{n\geq 1}n! \theta^n \|\widetilde{g}_n(\cdot,r,z,t,x)\|_{\cH_0^{\otimes n}}^2<\infty,
\end{equation}
where $\widetilde{g}_n(\cdot,r,z,t,x)$ is the symmetrization of $g_n(\cdot,r,z,t,x)$.

We now establish the existence of the solution $v_{\theta}^{(r,z)}$ and we show that its moments are also uniformly bounded.
%The theorem is the time-independent version of \cite{BHWX22}.

\begin{theorem}
\label{Thm-exist-v}
For any $r >0$, $z \in \bR$, equation \eqref{eq-v} has a unique solution $v_{\theta}^{(r,z)}$. Moreover, for any $\theta>0$, $p\geq 2$  and  $T>0$, we have:
\begin{equation}
\label{mom-v-bound}
	C_{\theta,p,T,v} := \sup_{0\leq r\leq t\leq T} \sup_{z,x \in \bR} \|v_{\theta}^{(r,z)}(t,x)\|_p < \infty.
\end{equation}
\end{theorem}

\begin{proof}
We denote $\pmb{\xi}=(\xi_1\ldots d\xi_n)$. By Cauchy-Schwarz inequality, we have
\begin{align*}
	& \|g_n(\cdot,z,x;r,t)\|_{\cP_0^{\otimes n}}^2 \\
	&=c_H^n \int_{\bR^n} \left| \int_{r<t_1<\ldots<t_n<t} \cF g_n(t_1,\cdot,\ldots,t_n,\cdot,r,z,t,x) (\pmb{\xi}) d\pmb{t} \right|^2 \prod_{j=1}^n |\xi_j|^{1-2H} d\pmb{\xi} \\
	&\leq c_H^n \dfrac{(t-r)^n}{n!} \int_{\bR^n} \int_{r<t_1<\ldots<t_n<t} \left| \cF g_n(t_1,\cdot,\ldots,t_n,\cdot,r,z,t,x) (\pmb{\xi}) \right|^2 d\pmb{t} \prod_{j=1}^n |\xi_j|^{1-2H} d\pmb{\xi}\\
	& =\dfrac{(t-r)^n}{n!} \|g_n(\cdot,r,z,t,x)\|_{\cH_0^{\otimes n}}^2.
%	\le \dfrac{T^n}{n!} \|g_n(\cdot,r,z,t,x)\|_{\cH_0^{\otimes %n}}^2.
\end{align*}
Since $g_n(t_1,x_1,\ldots,t_n,x_n,r,z,t,x)$ contains the indicator of the set $t_1<\ldots<t_n$,
\[
\|g_n(\cdot,r,z,t,x)\|_{\cH_0^{\otimes n}}^2= n! \|\widetilde{g}_n(\cdot,r,z,t,x)\|_{\cH_0^{\otimes n}}^2.
\]
We obtain that:
\[
 \|\widetilde{g}_n(\cdot,z,x;r,t)\|_{\cP_0^{\otimes n}}^2 \leq  \|g_n(\cdot,z,x;r,t)\|_{\cP_0^{\otimes n}}^2 \leq (t-r)^n \|\widetilde{g}_n(\cdot,r,z,t,x)\|_{\cH_0^{\otimes n}}^2,
\]
The converges of the series \eqref{series-conv-v} follows from \eqref{series-V-conv}. This proves the existence of the solution $v_{\theta}^{(r,z)}$. Moreover,
\begin{align}
\nonumber
\bE|v_{\theta}^{(r,z)}(t,x)|^2& =\sum_{n\geq 0}\theta^n n!\|\widetilde{g}_n(\cdot,z,x;r,t)\|_{\cP_0^{\otimes n}}^2 \\
\label{mom-v-V}
& \leq \sum_{n\geq 0} \theta^n n! (t-r)^n \|\widetilde{g}_n(\cdot,r,z,t,x)\|_{\cH_0^{\otimes n}}^2= \bE|V_{\theta (t-r)}^{(r,z)}(t,x)|^2.
\end{align}
To treat the higher moments, we use L\^e's hypercontractivity principle. Note that this principle was originally developed in \cite{le16} for the heat equation, but it is in fact valid for a larger class of SPDEs (see Theorem B.1 of \cite{BCC}). From this principle, combined with \eqref{mom-v-V}, we deduce that
\[
\|v_{\theta}^{(r,z)}(t,x)\|_p \leq \|v_{(p-1)\theta}^{(r,z)}(t,x)\|_2 \leq
\|V_{(p-1)\theta (t-r)}^{(r,z)}(t,x)\|_2 \leq \|V_{(p-1)\theta T}^{(r,z)}(t,x)\|_2.
\]
Relation \eqref{mom-v-bound} follows from \eqref{bound-V}.
%Example \ref{ex-v}.
\end{proof}

We conclude this section with a uniform bound for some integrals of $v_{\theta}$, which will be used in the proof of Theorem \ref{QCLT}.

\begin{lemma}
\label{sup-v-lem}
For any $\theta>0$, $p\geq 2$, $q>0$ and $t>0$,
\begin{align*}
\sup_{r \in [0,t]} \sup_{z \in \bR}\int_{\bR}\|v_{\theta}^{(r,z)}(t,x)\|_p^q dx \leq 2 t C_{\theta,p,t,v}^q \quad \mbox{and} \quad
\sup_{r \in [0,t]} \sup_{x \in \bR}\int_{\bR}\|v_{\theta}^{(r,z)}(t,x)\|_p^q dz \leq 2 t C_{\theta,p,t,v}^q,
\end{align*}
where $C_{\theta,p,t,v}$ is the constant given by relation \eqref{mom-v-bound}.
\end{lemma}

\begin{proof} Using identity \eqref{eq-identity-v} and relation \eqref{mom-v-bound}, we have:
\begin{align*}
\int_{\bR}\|v_{\theta}^{(r,z)}(t,x)\|_p^q dx& =2^q\int_{\bR} G_{t-r}^q(x-z) \|v_{\theta}^{(r,z)}(t,x)\|_p dx \\
& \leq C_{\theta,t,p,v}^q 2^q \int_{\bR}G_{t-r}^q(x-z)dx =2C_{\theta,t,p,v}^q(t-r).
\end{align*}
The other statement is proved similarly.
\end{proof}

\section{Estimates for the Mallivian derivatives}
\label{section-Malliavin}

In this section, we give some estimates for $Du_{\theta}(t,x)$ and $D^2u_{\theta}(t,x)$ and their increments, which are obtained using a connection with $v_{\theta}$.
If we denote by $U_{\theta}$ the solution of (hAm) with white noise in time $\fX$, it can be shown that $D_{r,z}U_{\theta}(t,x)=U_{\theta}(r,z)V_{\theta}^{(r,z)}(t,x)$. This relation does not hold for the time-independent noise. Luckily, with the help of Lemmas \ref{Lemma-decomp-F}
and \ref {Lemma-decomp-I}, we can still develop a connection between $Du_{\theta}(t,x)$ and $v_{\theta}$.

\medskip

%If $u_{\theta}(t,x) \in \bD^{1,2}$ for all $(t,x)$,
We start by observing that for any fixed $z\in \bR$, $D_zu_{\theta}(t,x)$ has the chaos expansion:
\begin{equation}
\label{chaos-D}
D_zu_{\theta}(t,x)= \sum_{n\geq 1}  \theta^{n/2} n I_{n-1}(\tilde f_n(\cdot,z,x;t)).
\end{equation}
See \eqref{Mal-F}. For any $z \in \bR^d$ fixed, we have the decomposition:
\begin{equation}
\label{decomp} \widetilde{f}_n(\cdot,z,x;t)=\frac{1}{n}\sum_{j=1}^{n}h_j^{(n)}(\cdot,z,x;t),
\end{equation}
where $h_j^{(n)}(\cdot,z,x;t)$ is the symmetrization of the function $f_j^{(n)}(\cdot,z,x;t)$ given by:
\begin{align*}
	&f_j^{(n)}(x_1,\ldots,x_{n-1},z,x;t)=
	f_n(x_1,\ldots,x_{j-1},z,x_{j},\ldots,x_{n-1},x;t)\\
	&\quad =\int_{\{0<t_1<\ldots<t_{j-1}<r<t_j<\ldots<t_{n-1}<t\}}
	G_{t-t_{n-1}}(x-x_{n-1})\ldots G_{t_j-r}(x_j-z)G_{r-t_{j-1}}(z-x_{j-1})\ldots \\
	& \qquad \qquad \qquad \qquad \qquad G_{t_2-t_1}(x_2-x_1)dt_1 \ldots dt_{n-1}dr.
\end{align*}

If the series \eqref{chaos-D} converges in $L^2(\Omega)$, then $u_{\theta}(t,x) \in \bD^{1,2}$ and $D_{\cdot}u_{\theta}(t,x)$ is a function in $z$. In this case,
\begin{equation}
\label{D-decomp}
D_zu_{\theta}(t,x)= \sum_{n\geq 1}  \theta^{n/2} \sum_{j=1}^{n}I_{n-1}(f_{j}^{(n)}(\cdot,z,x;t)).
\end{equation}

For the proof of the QCLT in Section \ref{sec:QCLT} below, we will need some estimates for the increments %$\|D_{z}u(t,x)-D_{z'}u(t,x)\|_p$ and for the rectangular increments of the second
of the Malliavin derivatives of $u_{\theta}$. We include these estimates in the next two theorems. Parts a) of these theorems give some estimates
%for $\|D_{z}u(t,x)\|_p$ and $\|D_{z,w}^2u(t,x)\|_p$
which are not needed in the present paper. We include these estimates for the sake of comparison with similar results that exist in the literature, e.g. Theorem 1.3 of \cite{BNQZ} for the colored noise in time,  respectively Theorems 3.1-3.2 of \cite{BY} for the time-independent noise.

\begin{theorem}	
\label{th-D-bound}
a) For any $\theta>0$, $p\geq 2$, $0\leq t\leq T$ and $x,z \in \bR$,
\begin{equation}
\label{D-bound}
	\|D_z u_{\theta}(t,x)\|_p \leq C \int_0^t G_{t-r}(x-z) dr,
\end{equation}
where $C= 2 \sqrt{\theta} C_{4(p-1)\theta,2,T,u} C_{4(p-1)\theta,2,T,v}$, and $C_{\theta,p,T,u}, C_{\theta,p,T,v}$ are the constants given by Theorems \ref{Thm-exist-u} and \ref{Thm-exist-v}.

b) For any $\theta>0$, $p\geq 2$, $0\leq  t\leq T$ and $x,z,z' \in \bR$,
\begin{align*}
%\label{ineq-I1-QCLT-1}
\left\| D_z u_{\theta}(t,x)- D_{z+z'}u_{\theta}(t,x) \right\|_p
	 & \le \sqrt{2\theta} \int_0^t I_1(z,z',x;r,t)dr,
\end{align*}
where
\begin{align*}
I_1(z,z',x;r,t)& = \left\| u_{\eta}(r,z) - u_{\eta}(r,z+z') \right\|_2 \left\| v_{\eta}^{(r,z)}(t,x) \right\|_2  + \\
& \quad \quad \quad \left\| u_{\eta}(r,z+z') \right\|_2 \left\| v_{\eta}^{(r,z)}(t,x) - v_{\eta}^{(r,z+z')}(t,x) \right\|_2
\end{align*}
and $\eta=4(p-1)\theta$.
\end{theorem}

\begin{proof}
a) Using notations \eqref{def-fk} and \eqref{def-gk}, we can write
\begin{equation}
\label{expr-f-jn}
f_j^{(n)}(\cdot,z,x;t)=\int_0^t f_{j-1}(\cdot,z;r)\otimes g_{n-j}(\cdot,z,x;r,t)dr.
\end{equation}

By stochastic Fubini's theorem, we can commute the $dr$ integral with the multiple Wiener integral $I_{n-1}$. We obtain:
\begin{align} \label{eq-I-decomposition-integral}
	I_{n-1}(f_j^{(n)}(\cdot,z,x;t))=\int_0^t I_{n-1}(f_{j-1}(\cdot,z;r)\otimes g_{n-j}(\cdot,z,x;r,t))dr.
\end{align}

By the definition \eqref{def-gk} and the special form $2 G_t(x) = \mathbf{1}_{|x|\le t} = 4 G_t^2(x)$, we have the identity $g_{n-j}(\cdot,z,x;r,t) = 2 G_{t-r}(x-z) g_{n-j}(\cdot,z,x;r,t)$. Hence, by \eqref{eq-I-decomposition-integral} and Cauchy-Schwarz inequality, as well as the Lemma \ref{Lemma-decomp-I}, we have
\begin{align*}
& \left\| I_{n-1}(f_j^{(n)}(\cdot,z,x;t))
	\right\|_2^2
	= \bE \left[ \left( \int_0^t 2 G_{t-r}(x-z) I_{n-1}(f_{j-1}(\cdot,z;r)\otimes g_{n-j}(\cdot,z,x;r,t))dr \right)^2 \right] \\
& \quad \le \int_0^t 2 G_{t-r}(x-z) dr \int_0^t \left\| I_{n-1}(f_{j-1}(\cdot,z;r)\otimes g_{n-j}(\cdot,z,x;r,t)) \right\|_2^2 dr \\
& \quad \le \binom{n-1}{j-1} \int_0^t 2 G_{t-r}(x-z) dr \int_0^t \left\| I_{j-1}(f_{j-1}(\cdot,z;r)) \right\|_2^2 \left\| I_{n-j}(g_{n-j}(\cdot,z,x;r,t)) \right\|_2^2 dr.
\end{align*}
Thus, by \eqref{D-decomp}, orthogonality and Cauchy-Schwarz inequality as well as the inequalities $(\sum_{j=1}^n a_j)^2 \leq n \sum_{j=1}^n a_j^2$,
$n \le 2^{n-1}$ and $\binom{n-1}{j-1} \le 2^{n-1}$, we have
\begin{align*}
& \| D_z u_{\theta}(t,x) \|_2^2
	= \sum_{n\geq 1} \theta^n \left\| \sum_{j=1}^n I_{n-1}(f_j^{(n)}(\cdot,z,x;t)) \right\|_2^2
	\le \sum_{n\geq 1} \theta^n n \sum_{j=1}^n \left\| I_{n-1}(f_j^{(n)}(\cdot,z,x;t)) \right\|_2^2 \\
	& \quad \leq \int_0^t 2 G_{t-r}(x-z) dr \sum_{n\geq 1} 4^{n-1} \theta^n \sum_{j=1}^n \int_0^t \left\| I_{j-1}(f_{j-1}(\cdot,z;r)) \right\|_2^2 \left\| I_{n-j}(g_{n-j}(\cdot,z,x;r,t)) \right\|_2^2 dr \\
	& \quad = \theta \int_0^t 2 G_{t-r}(x-z) dr  \\
	& \quad \int_0^t \sum_{j\geq 1} (4\theta)^{j-1} \left\| I_{j-1}(f_{j-1}(\cdot,z;r)) \right\|_2^2 \left(\sum_{n\geq j} (4\theta)^{n-j} \left\| I_{n-j}(g_{n-j}(\cdot,z,x;r,t)) \right\|_2^2\right) dr \\
	& \quad =\theta \int_0^t 2 G_{t-r}(x-z) dr \int_0^t \| u_{4\theta}(r,z) \|_2^2 \| v_{4\theta}^{(r,z)}(t,x) \|_2^2 dr.
\end{align*}
Using identity \eqref{eq-identity-v}, and Theorems \ref{Thm-exist-u} and \ref{Thm-exist-v}, we have
\begin{align*}
	\| D_z u_{\theta}(t,x) \|_2^2
	\le& 4\theta \int_0^t G_{t-r}(x-z) dr \int_0^t \| u_{4\theta}(r,z) \|_2^2 \| v_{4\theta}^{(r,z)}(t,x) \|_2^2 G_{t-r}(x-z) dr \\
	\le& 4\theta C_{4\theta,2,T,u}^2 C_{4\theta,2,T,v}^2 \left( \int_0^t G_{t-r}(x-z) dr \right)^2.
\end{align*}
By Lemma \ref{Lemma-hypercontractivity}, we have
\begin{align*}
	\| D_z u_{\theta}(t,x) \|_p
	\le \dfrac{1}{\sqrt{p-1}} \| D_z u_{(p-1)\theta}(t,x) \|_2
	\le 2 \sqrt{\theta} C_{4(p-1)\theta,2,t,u} C_{4(p-1)\theta,2,t,v} \int_0^t G_{t-r}(x-z) dr.
\end{align*}

b) By Lemma \ref{Lemma-hypercontractivity}, \eqref{D-decomp}, orthogonality, Cauchy-Schwarz inequality, the inequality $n \le 2^{n-1}$, Lemma \ref{Lemma-decomp-F}, we have:
\begin{align*}
&\left\| D_z u_{\theta}(t,x)- D_{z+z'}u_{\theta}(t,x) \right\|_p^2 \le \dfrac{1}{p-1} \left\| D_z u_{(p-1)\theta}(t,x)- D_{z+z'}u_{(p-1)\theta}(t,x) \right\|_2^2 \\
	& \quad =\dfrac{1}{p-1} \sum_{n\geq 1} \big((p-1)\theta\big)^n \left\| \sum_{j=1}^n I_{n-1} \left( f_j^{(n)}(\cdot,z,x;t) - f_j^{(n)}(\cdot,z+z',x;t) \right) \right\|_2^2 \\
	& \quad \le \dfrac{1}{p-1} \sum_{n\geq 1} \big((p-1)\theta\big)^n n \sum_{j=1}^n \left\| I_{n-1} \left( f_j^{(n)}(\cdot,z,x;t) - f_j^{(n)}(\cdot,z+z',x;t) \right) \right\|_2^2 \\
& \quad \leq \theta \sum_{n\geq 1} \big(2(p-1)\theta\big)^{n-1} \sum_{j=1}^n \left\| I_{n-1} \left( f_j^{(n)}(\cdot,z,x;t) - f_j^{(n)}(\cdot,z+z',x;t) \right) \right\|_2^2.
\end{align*}
Let $\tau=2(p-1)\theta$.
Using \eqref{expr-f-jn} for expressing $f_j^{(n)}(\cdot,z,x;t)$ and $f_j^{(n)}(\cdot,z+z',x;t)$, we have:
\[
\left\| D_z u_{\theta}(t,x)- D_{z+z'}u_{\theta}(t,x) \right\|_p^2 \leq 2\theta (T_1+T_2),
\]
where
\begin{align*}
T_1&=\sum_{n\geq 1} \tau^{n-1} \sum_{j=1}^n
\left\| I_{n-1} \left( \int_0^t \big( f_{j-1}(\cdot,z;r) - f_{j-1}(\cdot,z+z';r) \big) \otimes g_{n-j}(\cdot,z,x;r,t) dr \right) \right\|_2^2\\
T_2&=\sum_{n\geq 1} \tau^{n-1} \sum_{j=1}^n
\left\| I_{n-1} \left( \int_0^t f_{j-1}(\cdot,z+z';r) \otimes \big( g_{n-j}(\cdot,z,x;r,t) - g_{n-j}(\cdot,z+z',x;r,t) \big) dr \right) \right\|_2^2.
\end{align*}
We now use the fact that for any $\theta>0$ and for any functions $(h_n)_{n\geq 0}$ and $(h_n')_{n\geq 0}$ for which the integrals below are well-defined, by Lemma \ref{Lemma-decomp-F}, we have:
\begin{align*}
& \sum_{n\geq 1}\theta^{n-1}  \sum_{j=1}^n
\left\| I_{n-1}  \left(\int_0^t h_{j-1}(\cdot,r) \otimes h_{n-j}'(\cdot,r) dr \right) \right\|_2^2\\
& \quad =\sum_{p\geq 0} \sum_{q \geq 0}\theta^{p+q}
\left\| I_{p+q}  \left(\int_0^t  h_{p}(\cdot,r) \otimes h_{q}'(\cdot,r) dr\right) \right\|_2^2\\
& \quad \leq \left(\int_0^t \left\|\sum_{p\geq 0} (2\theta)^{p/2}I_p\big(h_p(\cdot,r)\big)\right\|_2
\left\|\sum_{q\geq 0} (2\theta)^{q/2}I_q\big(h_q'(\cdot,r)\big)\right\|_2dr \right)^2.
\end{align*}
The conclusion follows using the chaos decomposition of $u_{\theta}(t,x)$ and $v_{\theta}^{(r,z)}(t,x)$.
\end{proof}

We now examine the second Malliavin derivative.
For any fixed $w,z \in \bR$, we have the following chaos expansion
\begin{equation}
\label{D2-series}
	D_{w,z}^2 u_{\theta}(t,x)
	=\sum_{n\geq 2} n(n-1) \theta^{n/2} I_{n-2} \big( \widetilde{f}_n(\cdot,w,z,x;t) \big).
\end{equation}
Note that
\begin{align} \label{eq-hij}
	\widetilde{f}_n(\cdot,w,z,x;t)
	=\frac{1}{n(n-1)} \sum_{i,j=1,i\neq j}^n
	h_{ij}^{(n)}(\cdot,w,z,x;t),
\end{align}
where $h_{ij}^{(n)}(\cdot,w,z,x;t)$ is the symmetrization of the function $f_{ij}^{(n)}(\cdot,w,z,x;t)$ defined as follows. If $i<j$,
\begin{align*}
& f_{ij}^{(n)}(x_1,\ldots,x_{n-2},w,z,x;t) \\
&\quad = f_n(x_1,\ldots,x_{i-1},w,x_{i},	\ldots,x_{j-2},z,x_{j-1},\ldots,x_{n-2},x;t) \\
& \quad =	\int_{\{0<t_1<\ldots<t_{i-1}<s<t_i<\ldots<t_{j-2}<r<t_{j-1}<\ldots<t_{n-2}<t\}}
	G_{t-t_{n-2}}(x-x_{n-2}) \ldots G_{t_{j-1}-r}(x_{j-1}-z) \\
	& \quad G_{r-t_{j-2}}(z-x_{j-2}) \ldots G_{t_i-s}(x_i-w)G_{s-t_{i-1}}(w-x_{i-1}) \ldots G_{t_2-t_1}(x_2-x_1) dt_1 \ldots dt_{n-2}drds \\
& \quad = \int_{0<s<r<t} drds f_{i-1}(x_1,\ldots,x_{i-1},w;s) g_{j-i-1}(x_i,\ldots,x_{j-2},w,z;s,r) g_{n-j}(x_{j-1},\ldots,x_{n-2},z,x;r,t).
\end{align*}
If $j<i$,
\begin{align*}
	& f_{ij}^{(n)}(x_1,\ldots,x_{n-2},w,z,x;t) \\
%	=& %f_n(x_1,\ldots,x_{j-1},z,x_j,\ldots,x_{i-2},w,x_{i-1},\ldots,x_{n-2},x;t) %\\
	& \quad =\int_{0<s<r<t} drds f_{j-1}(x_1,\ldots,x_{j-1},z;s) g_{i-j-1}(x_j,\ldots,x_{i-2},z,w;s,r) g_{n-i}(x_{i-1},\ldots,x_{n-2},w,x;r,t)
\end{align*}
In both cases, $w$ is on position $i$ and $z$ is on position $j$. Consequently,
\begin{equation}
\label{D2-decomp}
D_{w,z}^2 u_{\theta}(t,x)
	=\sum_{n\geq 2} \theta^{n/2} \sum_{i,j=1,i\neq j}^n I_{n-2} \big( f_{ij}^{(n)}(\cdot,w,z,x;t) \big).
\end{equation}

\begin{theorem}	
\label{th-D2-bound}
a) For any $\theta>0$, $p\geq 2$, $0\leq t\leq T$ and $x,z,w \in \bR$,
\begin{equation}
\label{D2-bound}
	\|D_{w,z}^2 u_{\theta}(t,x)\|_p \leq C \widetilde{f}_2(w,z,t,x)
\end{equation}
where $C = 4\theta C_{6\theta(p-1),2,T,u} C_{6\theta(p-1),2,T,v}^2$.

b) For any $\theta>0$, $p\geq 2$, $0\leq t\leq T$ and $x,z,z',w,w' \in \bR$,
\begin{align}
\nonumber
& \left\| D_{z,y}^2 u_{\theta}(t,x)- D_{z,y+y'}^2 u_{\theta}(t,x)- D_{z+z',y}^2 u_{\theta}(t,x)+ D_{z+z',y+y'}^2 u_{\theta}(t,x) \right\|_p \\
\label{ineq-I2-QCLT-1}
& \quad \leq 4\theta \int_{0<s<r<t} \big(I_2(z,z',y,y',x;s,r,t) + I_2(y,y',z,z',x;s,r,t)\big) dsdr,
\end{align}
where $I_2(z,z',y,y',x;s,r,t)=\sum_{k=1}^4 I_{2,k}(z,z',y,y',x;s,r,t)$ and
\begin{align*}
& I_{2,1}(z,z',y,y',x;s,r,t) = \left\| u_{\eta}(s,z) - u_{\eta}(s,z+z') \right\|_2 \left\| v_{\eta}^{(s,z)} (r,y) \right\|_2 \left\| v_{\eta}^{(r,y)} (t,x) - v_{\eta}^{(r,y+y')} (t,x_3) \right\|_2 \\
& I_{2,2}(z,z',y,y',x;s,r,t) = \left\| u_{\eta}(s,z+z') \right\|_2 \left\| v_{\eta}^{(s,z)} (r,y) - v_{\eta}^{(s,z+z')} (r,y) \right\|_2 \left\| v_{\eta}^{(r,y)} (t,x) - v_{\eta}^{(r,y+y')} (t,x) \right\|_2 \\
& I_{2,3}(z,z',y,y',x;s,r,t) =\left\| u_{\eta}(s,z) - u_{\eta}(s,z+z') \right\|_2 \left\| v_{\eta}^{(s,z)} (r,y) - v_{\eta}^{(s,z)} (r,y+y') \right\|_2 \left\| v_{\eta}^{(r,y+y')} (t,x) \right\|_2 \\
& I_{2,4}(z,z',y,y',x;s,r,t) = \left\| u_{\eta}(s,z+z') \right\|_2 \\
& \quad \left\| v_{\eta}^{(s,z)} (r,y) - v_{\eta}^{(s,z)} (r,y+y') - v_{\eta}^{(s,z+z')} (r,y) + v_{\eta}^{(s,z+z')} (r,y+y') \right\|_2  \left\| v_{\eta}^{(r,y+y')} (t,x) \right\|_2.
\end{align*}
and $\eta=6(p-1)\theta$.

\end{theorem}

\begin{proof}
a) The proof is similar to the proof of Theorem \ref{th-D-bound}.a). We omit the details.

b) By Lemma \ref{Lemma-hypercontractivity}, \eqref{D2-decomp},
%expansion \eqref{D2-series} and \eqref{eq-hij}
orthogonality, and the inequality $(\sum_{k=1}^{N} a_k)^2 \leq N \sum_{k=1}^N a_k^2$,
%and Cauchy-Schwarz inequality and the two %estimates above,
\begin{align}
& \left\| D_{z,y}^2 u_{\theta}(t,x)- D_{z,y+y'}^2 u_{\theta}(t,x)- D_{z+z',y}^2 u_{\theta}(t,x)+ D_{z+z',y+y'}^2 u_{\theta}(t,x) \right\|_p^2 \nonumber \\
& \quad \leq \dfrac{1}{(p-1)^2} \left\| D_{z,y}^2 u_{(p-1)\theta}(t,x)- D_{z,y+y'}^2 u_{(p-1)\theta}(t,x)- D_{z+z',y}^2 u_{(p-1)\theta}(t,x)+ D_{z+z',y+y'}^2 u_{(p-1)\theta}(t,x) \right\|_2^2 \nonumber \\
& \quad = \dfrac{1}{(p-1)^2} \sum_{n\geq 2} \big((p-1)\theta\big)^n \bigg\| \sum_{i,j=1,i\neq j}^n I_{n-2} \Big( f_{ij}^{(n)} (\cdot,z,y,x;t) - f_{ij}^{(n)} (\cdot,z,y+y',x;t) \nonumber \\
& \quad \quad \quad - f_{ij}^{(n)} (\cdot,z+z',y,x;t) + f_{ij}^{(n)} (\cdot,z+z',y+y',x;t) \Big) \bigg\|_2^2 \nonumber \\
& \quad \le \dfrac{1}{(p-1)^2} \sum_{n\geq 2} \big((p-1)\theta\big)^n n(n-1) \sum_{i,j=1,i\neq j}^n \bigg\| I_{n-2} \Big( f_{ij}^{(n)} (\cdot,z,y,x;t) - f_{ij}^{(n)} (\cdot,z,y+y',x;t) \nonumber \\
& \quad \quad \quad - f_{ij}^{(n)} (\cdot,z+z',y,x;t) + f_{ij}^{(n)} (\cdot,z+z',y+y',x;t) \Big) \bigg\|_2^2 =: \frac{1}{(p-1)^2} (S_1+S_2).
\label{S1-S2}
\end{align}
where $S_1,S_2$ denote the sums corresponding to $i<j$, respectively $j<i$.

We first treat the sum $S_1$. When $i<j$, we can write
\begin{align*}
	f_{ij}^{(n)} (\cdot,z,y,x;t) = \int_{0<s<r<t} f_{i-1}(\cdot,z;s) \otimes g_{j-i-1}(\cdot,z,y;s,r) \otimes g_{n-j}(\cdot,y,x;r,t) dsdr.
\end{align*}
Besides, we have the following decomposition (we pair the first two terms and the last two terms in the first step, and then pair the first and the third term in the second step):
\begin{align*}
	&a_z \otimes b_{z,y} \otimes c_y - a_z \otimes b_{z,y+y'} \otimes c_{y+y'} - a_{z+z'} \otimes b_{z+z',y} \otimes c_y + a_{z+z'} \otimes b_{z+z',y+y'} \otimes c_{y+y'} \\
	=& a_z \otimes b_{z,y} \otimes (c_y -c_{y+y'}) + a_z \otimes (b_{z,y} - b_{z,y+y'}) \otimes c_{y+y'} \\
	&- a_{z+z'} \otimes b_{z+z',y} \otimes (c_y -c_{y+y'}) - a_{z+z'} \otimes (b_{z+z',y} - b_{z+z',y+y'}) \otimes c_{y+y'} \\
	=& (a_z - a_{z+z'}) \otimes b_{z,y} \otimes (c_y -c_{y+y'}) + a_{z+z'} \otimes (b_{z,y} - b_{z+z',y}) \otimes (c_y -c_{y+y'}) \\
	& + (a_z-a_{z+z'}) \otimes (b_{z,y} - b_{z,y+y'}) \otimes c_{y+y'} + a_{z+z'} \otimes (b_{z,y} - b_{z,y+y'} - b_{z+z',y} + b_{z+z',y+y'}) \otimes c_{y+y'}.
\end{align*}
Hence, for any $1\leq i<j\leq n$,
\begin{align*}
	&f_{ij}^{(n)} (\cdot,z,y,x;t) - f_{ij}^{(n)} (\cdot,z,y+y',x;t) - f_{ij}^{(n)} (\cdot,z+z',y,x;t) + f_{ij}^{(n)} (\cdot,z+z',y+y',x;t) \\
	& \qquad \qquad \qquad \qquad \qquad \qquad =\sum_{k=1}^4 \int_{0<s<r<t}A_{ij,k}^{(n)}(s,r)dsdr,
\end{align*}
where
\begin{align*}
A_{ij,1}^{(n)}(s,r)& :=\big( f_{i-1}(\cdot,z;s) - f_{i-1}(\cdot,z+z';s) \big) \otimes g_{j-i-1}(\cdot,z,y;s,r) \\
	& \otimes \big( g_{n-j}(\cdot,y,x;r,t) - g_{n-j}(\cdot,y+y',x;r,t) \big) \\
A_{ij,2}^{(n)}(s,r)& := f_{i-1}(\cdot,z+z';s) \otimes \big( g_{j-i-1}(\cdot,z,y;s,r) - g_{j-i-1}(\cdot,z+z',y;s,r) \big) \\
	& \otimes \big( g_{n-j}(\cdot,y,x;r,t) - g_{n-j}(\cdot,y+y',x;r,t) \big) \\
A_{ij,3}^{(n)}(s,r)& := \big( f_{i-1}(\cdot,z;s) - f_{i-1}(\cdot,z+z';s) \big) \otimes \big( g_{j-i-1}(\cdot,z,y;s,r) - g_{j-i-1}(\cdot,z,y+y';s,r) \big) \\
	& \otimes g_{n-j}(\cdot,y+y',x;r,t) \\
A_{ij,4}^{(n)}(s,r)& := f_{i-1}(\cdot,z+z';s) \otimes \big( g_{j-i-1}(\cdot,z,y;s,r) - g_{j-i-1}(\cdot,z,y+y';s,r) - g_{j-i-1}(\cdot,z+z',y;s,r) \\
	& + g_{j-i-1}(\cdot,z+z',y+y';s,r) \big) \otimes g_{n-j}(\cdot,y+y',x;r,t)
\end{align*}
Using inequality $n(n-1) \le 2^n$ and the inequality $(\sum_{k=1}^4 a_k)^4 \leq 4 \sum_{k=1}^4 a_k^2$
we have
\begin{align*}
S_1	& \le  \sum_{n\geq 2}\sum_{1\le i<j \le n} \big(2(p-1)\theta\big)^{n} \bigg\| I_{n-2} \Big( f_{ij}^{(n)} (\cdot,z,y,x;t) - f_{ij}^{(n)} (\cdot,z,y+y',x;t) - f_{ij}^{(n)} (\cdot,z+z',y,x;t)  \\
	& \qquad \qquad \qquad \qquad \qquad+ f_{ij}^{(n)} (\cdot,z+z',y+y',x;t) \Big) \bigg\|_2^2 \\
& \leq 16 (p-1)^2 \theta^2  \sum_{k=1}^4 \sum_{n\geq 2} \big(2(p-1)\theta\big)^{n-2}\sum_{1\le i<j \le n}  \Big\| I_{n-2} \Big( \int_{0<s<r<t} A_{ij,k}^{(n)}(s,r)drdr \Big) \Big\|_2^2.
\end{align*}

We now use the fact that for any $\theta>0$ and for any functions $h_n,h_n',h_n''$ for which the integrals below are well-defined, by Lemma \ref{Lemma-decomp-F}, we have:
\begin{align*}
& \sum_{n\geq 2}\theta^{n-2} \sum_{i,j=1,i<j}^n \left\| I_{n-2}\left(\int_{0<s<r<t} h_{j-1}(\cdot,s)\otimes h_{j-i-1}'(\cdot,s,r)\otimes h_{n-j}''(\cdot,r)dsdr \right)\right\|_2^2\\
& = \sum_{p \geq 0}\sum_{q\geq 0} \sum_{k\geq 0}\theta^{p+q+k}\left\| I_{p+q+k}\left(\int_{0<s<r<t} h_{p}(\cdot,s)\otimes h_{q}'(\cdot,s,r)\otimes h_{k}''(\cdot,r)dsdr \right)\right\|_2^2\\
&\leq \left(\int_{0<s<r<t} \left\| \sum_{p\geq 0}(3\theta)^{p/2} I_p(h_p(\cdot,s))\right\|_2
\left\| \sum_{q\geq 0}(3\theta)^{q/2} I_q(h_q'(\cdot,s,r))\right\|_2
\left\| \sum_{k\geq 0}(3\theta)^{k/2} I_k(h_k''(\cdot,r))\right\|_2
dsdr \right)^2.
\end{align*}

Using the chaos decompositions of $u_{\theta}(t,x)$ and $v_{\theta}^{(r,z)}(t,x)$, we infer that:
\[
S_1 \le   16 (p-1)^2 \theta^2 \sum_{k=1}^4 \left(\int_{0<s<r<t}
 I_{2,k}(z,z',y,y',x;s,r,t) dsdr  \right)^2.
 \]

Using the inequality $\sum_{i=1}^4 a_i^2 \le (\sum_{i=1}^4 a_i)^2$ for $a_i>0$, we obtain:
\begin{equation}
\label{S1}
S_1 \leq  \left( 4(p-1) \theta \int_{0<s<r<t}
I_{2}(z,z',y,y',x;s,r,t) dsdr  \right)^2.
\end{equation}

A similar formula holds for the sum $S_2$ (which corresponds to the case $j<i$), which is obtained by swapping $(y,y',i)$ and $(z,z',j)$:
\begin{equation}
\label{S2}
S_2 \leq \left( 4(p-1) \theta \int_{0<s<r<t} I_2(y,y',z,z',x_3;s,r,t) dsdr \right)^2.
\end{equation}

Relation \eqref{ineq-I2-QCLT-1} follows from \eqref{S1-S2}, \eqref{S1} and \eqref{S2}.
%using the fact that $a^2+b^2 \leq (a+b)^2$.
\end{proof}

\section{Increments of $u_{\theta}$ and $v_{\theta}$}
\label{section-increm}

In this section, study some integrals involving the increments of the solutions $u_{\theta}$ and $v_{\theta}$, which will be used in the proof of the QCLT. A key role is played by some stationarity properties of these processes.

The first result examines the increments of $u_{\theta}$.

\begin{lemma}
\label{Lem-int-u-u}
For any $\theta>0$, $p \geq 2$ and $T>0$,
\begin{align*}
	C'_{u,p,H,T,\theta} := \sup_{(t,x) \in [0,T] \times \bR} \int_{\bR}\|u_{\theta}(t,x)-u_{\theta}(t,x+h)\|_p^2 |h|^{2H-2}dh<\infty.
\end{align*}
\end{lemma}

\begin{proof}
By Theorem \ref{Thm-exist-u}, since $H < 1/2$, we have:
\begin{align*}
	\sup_{(t,x) \in [0,T]\times \bR} \int_{|h|>1}\|u_{\theta}(t,x)-u_{\theta}(t,x+h)\|_p^2 |h|^{2H-2}dh
	\leq 2C_{\theta,p,T,u} \int_{|h|>1}|h|^{2H-2}dh<\infty.
\end{align*}
It remains to treat the integral over the set $|h|\leq 1$.

We use the chaos expansion $u_{\theta}(t,x)-u_{\theta}(t,x+h) = \sum_{n\geq 1} \theta^{n/2} I_n(f_{n,h}(\cdot,x;t))$ where $f_{n,h}(\cdot,x;t)=f_{n}(\cdot,x;t)-f_{n}(\cdot,x+h;t)$. By hypercontractivity property \eqref{hyper}, %for any $p\geq 2$,
\begin{align}
\label{bound-dif-u}
	\|u(t,x)-u(t,x+h)\|_p
	\leq \sum_{n\geq 1} (p-1)^{n/2} \theta^{n/2} \|I_n(f_{n,h}(\cdot,x;t))\|_2
	= \sum_{n\geq 1} (p-1)^{n/2} \theta^{n/2} [J_{n,h}(t)]^{1/2}
\end{align}
where $J_{n,h}(t) = n! \|\widetilde{f}_{n,h}(\cdot,x;t) \|_{\cP_0^{\otimes n}}^2$. Note that
\[
f_{n,h}(x_1,\ldots,x_n,x;t)=\int_{T_n(t)}f_{n,h}(t_1,x_1,\ldots,t_n,x_n,t,x)
d{\bf t},
\]
where $f_{n,h}(\cdot,t,x)=f_{n}(\cdot,t,x)-f_{n}(\cdot,t,x+h)$ and $f_n(\cdot,t,x)$ is given by \eqref{def-fn-time-dep}.
Using the rough bound \eqref{rough-bound}, %$\|\widetilde{f}\|_{\cP_0^{\otimes n}} \leq %\|f\|_{\cP_0^{\otimes n}}$,
followed by the Cauchy-Schwarz inequality, we have
\begin{align*}
	&J_{n,h}(t)
	\le n! \|f_{n,h}(\cdot,x;t) \|_{\cP_0^{\otimes n}}^2 \\
	=& c_H^n n! \int_{\bR^n} \left| \int_{T_n(t)} \cF f_{n,h}(t_1,\cdot,\ldots,t_n,\cdot,t,x)(\xi_1,\ldots,\xi_n) d\pmb{t} \right|^2 \prod_{j=1}^n |\xi_j|^{1-2H} d\pmb{\xi} \\
	\le& c_H^n t^n \int_{T_n(t)} \int_{\bR^n} \left| \cF f_{n,h}(t_1,\cdot,\ldots,t_n,\cdot,t,x)(\xi_1,\ldots,\xi_n) \right|^2 \prod_{j=1}^n |\xi_j|^{1-2H} d\pmb{\xi} d\pmb{t} \\
	\le& c_H^n t^n \int_{T_n(t)} \int_{\bR^n} |1-e^{-i(\xi_1+\ldots+\xi_n)h}|^2 \prod_{j=1}^n |\cF G_{t_{j+1}-t_j}(\xi_1+\ldots+\xi_j)|^2 \prod_{j=1}^n |\xi_j|^{1-2H} d\pmb{\xi} d\pmb{t} \\
	=& c_H^n t^n \int_{T_n(t)} \int_{\bR^n} |1-e^{-i\eta_n h}|^2 \prod_{j=1}^n |\cF G_{t_{j+1}-t_j}(\eta_j)|^2 \prod_{j=1}^n |\eta_j-\eta_{j-1}|^{1-2H} d\pmb{\eta} d\pmb{t},
\end{align*}
where we change the variable $\xi_1 + \ldots + \xi_j = \eta_j$ with the convention $t_{n+1}=t$ and $\eta_0=0$.

The product above is estimated as usually, using inequality \eqref{prod}. Then, using identity \eqref{G-int}, we obtain:
\begin{align*}
 J_{n,h}(t) &\leq c_H^n t^n \sum_{\alpha \in D_n} \int_{T_n(t)} \int_{\bR^n} |1-e^{-i\eta_n h}|^2 \prod_{j=1}^{n}\int_{\bR}|\cF G_{t_{j+1}-t_j}(\eta_j)|  \prod_{j=1}^n|\eta_j|^{\alpha_j}d\pmb{\eta} d\pmb{t}\\
& \leq c^n t^n \sum_{\alpha \in D_n} \int_0^t
\left( \int_{\bR}|1-e^{-i\eta_n h}|^2  |\cF G_{t-t_n}(\eta_n)|^2 |\eta_n|^{\alpha_n}\eta_n\right) \\
& \quad \quad \quad \left(\int_{T_{n-1}(t_n)}
\prod_{j=1}^{n-1}(t_{j+1}-t_j)^{1-\alpha_j}dt_1 \ldots dt_{n-1}\right) dt_n\\
& \leq c_{t,H} \frac{c^n t^n }{(n!)^{2H+1}} \int_0^t \int_{\bR}
|1-e^{-i\eta_n h}|^2  (1+|\eta_n|^{1-2H}) \frac{\sin^2((t-t_n)|\eta_n|)}{|\eta_n|^2} t_n^{(2H+1)n-2}d\eta_n dt_n,
\end{align*}
where $c_{t,H}=t^{1-2H}+1$ and $c>0$ is a constant depending on $H$
(which may be different from line to line).
Using the bound
$\sin^2((t-t_n)|\eta_n|) \leq 1$, we obtain:
\begin{align*}
J_{n,h}(t) & \leq c_{t,H} \frac{c^n t^n}{(n!)^{2H+1}} t^{(2H+1)n-1}
\int_{\bR}
|1-e^{-i\eta h}|^2 (1+|\eta|^{1-2H}) \frac{1}{|\eta|^2} d\eta \\
& \leq c_{t,H}\frac{c^{n} }{(n!)^{2H+1}} t^{(2H+2)n-1}  (|h|+|h|^{2H}).
\end{align*}
The estimate above is now inserted in \eqref{bound-dif-u}. We obtain:
\begin{align*}
\|u(t,x)-u(t,x+h)\|_p & \leq c_{t,H}^{1/2} \sum_{n\geq 1}(p-1)^{n/2}\theta^{n/2} \frac{c^{n/2}}{(n!)^{H+1/2}} t^{(H+1)n-1/2} (|h|^{1/2}+|h|^{H}) \\
& =: C_{t,p,\theta,H} (|h|^{1/2}+|h|^{H}),
\end{align*}
where $C_{t,p,\theta,H}>0$ is a constant which depends in $t,p,\theta,H$ and is increasing in $t$. Hence,
\begin{align*}
\int_{|h|\leq 1}\|u(t,x)-u(t,x+h)\|_p^2 |h|^{2H-2}dh
\leq  C_{t,p,\theta,H}^2 \int_{|h|\leq 1} (|h|+|h|^{2H}) |h|^{2H-2}dh.
\end{align*}
The last integral is finite due to the condition $H > 1/4$.
\end{proof}

The following lemma gives some translation invariance properties of $u_{\theta}$ and $v_{\theta}$.

\begin{lemma}
\label{translation-lemma}
%\label{Lem-transition invariant}
For any $\theta>0$, $0\leq r \leq t$, $x, z, h \in \bR$, we have:
\begin{align*}
& u_{\theta}(t,x)-u_{\theta}(t,x+h) \stackrel{d}{=}
u_{\theta}(t,0)-u_{\theta}(t,h) \\
& v_{\theta}^{(r,z)}(t,x)-v_{\theta}^{(r,z+h)}(t,x) \stackrel{d}{=}
v_{\theta}^{(r,0)}(t,x-z)-v_{\theta}^{(r,h)}(t,x-z)\\
& v_{\theta}^{(r,z)}(t,x)-v_{\theta}^{(r,z)}(t,x+h) \stackrel{d}{=}
v_{\theta}^{(r,0)}(t,x-z)-v_{\theta}^{(r,0)}(t,x+h-z)
\end{align*}
Moreover, for the rectangular difference, we have:
\begin{align}
\label{eq-transition-vv}
& v^{(r,z)}_{\theta}(t,x) - v^{(r,z+h)}_{\theta}(t,x)-v^{(r,z)}_{\theta}(t,x+k) + v^{(r,z+h)}_{\theta}(t,x+k) \stackrel{d}{=} \nonumber \\
& \quad  v^{(r,0)}_{\theta}(t,x-z) - v^{(r,h)}_{\theta}(t,x-z)- v^{(r,0)}_{\theta}(t,x+k-z)  + v^{(r,h)}_{\theta}(t,x+k-z).
\end{align}
\end{lemma}

\begin{proof}
The lemma follows using the chaos expansions of these differences and the fact that the noise $W$ is spatially homogeneous, i.e.
$W \stackrel{d}{=}W^{(x)}$ for any $x \in \bR$, where
$W^{(x)}(\varphi)=W(\varphi(\cdot-x))$.
\end{proof}

\begin{remark}
\label{trans-rem}
{\rm For any $r\in [0,t], h \in \bR$, the following integrals do not depend on $x,z\in \bR$:
\begin{align}
\label{xz-integrals}
& \int_{\bR}\left\| v_{\theta}^{(r,z)}(t,x') - v_{\theta}^{(r,z+h)}(t,x') \right\|_p^{q} dx'=
\int_{\bR}\left\| v_{\theta}^{(r,z')}(t,x) - v_{\theta}^{(r,z'+h)}(t,x) \right\|_p^{q} dz'= I_{r,t,h}^{(\theta,p,q)}, \\
\nonumber
& \int_{\bR}\left\| v_{\theta}^{(r,z)}(t,x')-v_{\theta}^{(r,z)}(t,x'+h)  \right\|_p^{q} dx'=
\int_{\bR}\left\| v_{\theta}^{(r,z')}(t,x)-v_{\theta}^{(r,z')}(t,x+h)  \right\|_p^{q} dz'= J_{r,t,h}^{(\theta,p,q)},
\end{align}
where
\begin{align*}
I_{r,t,h}^{(\theta,p,q)} =\int_{\bR}\left\| v_{\theta}^{(r,0)}(t,x) - v_{\theta}^{(r,h)}(t,x) \right\|_p^{q} dx, \quad
J_{r,t,h}^{(\theta,p,q)}=\int_{\bR}\left\| v_{\theta}^{(r,0)}(t,x) - v_{\theta}^{(r,0)}(t,x+h) \right\|_p^{q} dx.
\end{align*}
To see this, we apply Lemma \ref{translation-lemma}, followed by the changes of variables $x''=x'-z$ (for the integrals on the left), respectively $z''=x-z'$ (for the integrals on the right).
Moreover, the same argument shows that the two integrals below also do not depend on $x,z \in \bR$:
\begin{align*}
& \int_{\bR} \left\| v^{(r,z)}_{\theta}(t,x') - v^{(r,z+h)}_{\theta}(t,x')-v^{(r,z)}_{\theta}(t,x'+k) + v^{(r,z+h)}_{\theta}(t,x'+k) \right\|_p^q dx'  \\
& = \int_{\bR} \left\| v^{(r,z')}_{\theta}(t,x) - v^{(r,z'+h)}_{\theta}(t,x)-v^{(r,z')}_{\theta}(t,x+k) + v^{(r,z'+h)}_{\theta}(t,x+k) \right\|_p^q dz'.
\end{align*}
}
\end{remark}

The next result examines the increments of $v_{\theta}$. Its proof uses the connection between $v_{\theta}$ and $V_{\theta}$, and the fact that similar inequalities hold for $V_{\theta}$.

\begin{lemma}
\label{v-lem1}
For any $\theta>0$, $p\geq 2$ and $t>0$,
\begin{align*}
a) &	\quad	\sup_{r \in [0,t]}\sup_{z \in \bR}  \int_{\bR^2}\left\| v_{\theta}^{(r,z)}(t,x) - v_{\theta}^{(r,z+h)}(t,x) \right\|_p^2 |h|^{2H-2} dx  dh \leq C_{v,p,H,t,\theta}, \\
b) & \quad \sup_{r \in [0,t]} \sup_{x \in \bR} \int_{\bR^2}  \left\| v_{\theta}^{(r,z)}(t,x) - v_{\theta}^{(r,z+h)}(t,x) \right\|_p^2 |h|^{2H-2} dz   dh \leq C_{v,p,H,t,\theta}, \\
c) &	\quad	\sup_{r \in [0,t]}\sup_{z \in \bR}  \int_{\bR^2}\left\| v_{\theta}^{(r,z)}(t,x) - v_{\theta}^{(r,z)}(t,x+h) \right\|_p^2 |h|^{2H-2} dx  dh \leq C_{v,p,H,t,\theta},\\
d) &	\quad	\sup_{r \in [0,t]}\sup_{x \in \bR}  \int_{\bR^2}\left\| v_{\theta}^{(r,z)}(t,x) - v_{\theta}^{(r,z)}(t,x+h) \right\|_p^2 |h|^{2H-2} dz  dh \leq C_{v,p,H,t,\theta},\\
e) & \quad \sup_{r \in [0,t]} \sup_{z\in \bR} \int_{\bR^3}\big\|v_{\theta}^{(r,z)}(t,x)-
v_{\theta}^{(r,z+h)}(t,x)-v_{\theta}^{(r,z)}(t,x+k)+
v_{\theta}^{(r,z+h)}(t,x+k)\big\|_p^2 \\
& \qquad \quad \quad \quad \qquad |h|^{2H-2}
|k|^{2H-2}dxdh dk  \leq C_{v,p,H,t,\theta},\\
f) & \quad \sup_{r \in [0,t]} \sup_{x\in \bR} \int_{\bR^3}\big\|v_{\theta}^{(r,z)}(t,x)-
v_{\theta}^{(r,z+h)}(t,x)-v_{\theta}^{(r,z)}(t,x+k)+
v_{\theta}^{(r,z+h)}(t,x+k)\big\|_p^2 \\
& \qquad \quad \quad \quad \qquad |h|^{2H-2}
|k|^{2H-2}dzdh dk  \leq C_{v,p,H,t,\theta},
\end{align*}
where $C_{v,p,H,t,\theta}>0$ is a constant that depends on $(p,H,t,\theta)$ and is increasing in $t$.
\end{lemma}

\begin{proof} Due to Remark \ref{trans-rem} and \eqref{eq-transition-vv}, we only have to prove a), c) and e).

a) We use the following fact: for any $x,z \in \bR$,
\begin{align}
\label{ineq-compare-norm}
\left\| v_{\theta}^{(r,z)}(t,x) - v_{\theta}^{(r,z+h)}(t,x) \right\|_p^2
		\le \left\| V_{(p-1)\theta t}^{(r,z)}(t,x) - V_{(p-1)\theta t}^{(r,z+h)}(t,x) \right\|_2^2.
\end{align}
This can be shown as in the proof of Theorem \ref{Thm-exist-v}: we first prove it for $p=2$ (replacing $g_n(\cdot,z,x;r,t)$ by $g_n(\cdot,z,x;r,t) -g_n(\cdot,z+h,x;r,t)$), and then we use the fact that $\|v_{\theta}^{(r,z)}(t,x) - v_{\theta}^{(r,z+h)}(t,x)\|_p \leq \|v_{(p-1)\theta}^{(r,z)}(t,x) - v_{(p-1)\theta}^{(r,z+h)}(t,x)\|_2$, due to L\^e's hypercontractivity principle (Theorem B.1 of \cite{BCC}), noting that $v_{\theta}^{(r,z)}(t,x) - v_{\theta}^{(r,z+h)}(t,x)$ satisfies an equation of the form given in Appendix B.1 of \cite{BCC}.

Then, the inequality in part a) follows from relation \eqref{V-int1} (Example \ref{TheoremF-ex}), with the constant
$C_{v,p,H,t,\theta}=C_{V,2,H,t,(p-1)\theta t}$, where $C_{V,p,H,t,\theta}$ is given by \eqref{TheoremF-const}.

c) For any $x,z \in \bR$,
\begin{align*}
\left\| v_{\theta}^{(r,z)}(t,x) - v_{\theta}^{(r,z)}(t,x+h) \right\|_p^2 & \leq
\left\| v_{(p-1)\theta}^{(r,z)}(t,x) - v_{(p-1)\theta}^{(r,z)}(t,x+h) \right\|_2^2 \\
& \leq \left\| V_{(p-1)\theta t}^{(r,z)}(t,x) - V_{(p-1)\theta t}^{(r,z)}(t,x+h) \right\|_2^2,
\end{align*}
where the first inequality is due to Lemma \ref{Lemma-hypercontractivity} and the second inequality is proved similarly to \eqref{ineq-compare-norm}.
Therefore, it suffices to show that the integral
\[
\int_{\bR^2} \left\| V_{\theta t}^{(r,z)}(t,x) - V_{\theta t}^{(r,z)}(t,x+h) \right\|_2^2 |h|^{2H-2}dx dh
\]
is uniformly bounded, for all $r \in [0,t]$ and $z \in \bR$.
To treat this integral, we cannot use the method of Example \ref{TheoremF-ex}, since $V_{\theta}^{(r,z)}(t,x) - V_{\theta}^{(r,z)}(t,x+h)$ does not satisfy an integral equation similar to the one given in Theorem \ref{theoremF}. So we need to proceed differently. The idea is to move the increment from the $x$ variable to the $z$ variable, and then use the bound given by \eqref{V-int1} for the resulting integral.
From the chaos expansion, we have:
\begin{align*}
\left\|V_{\theta t}^{(r,z)}(t,x) - V_{\theta t}^{(r,z)}(t,x+h)\right\|_2^2&=\big|G_{t-r}(x-z)-G_{t-r}(x+h-z) \big|^2+\\
& \quad \sum_{n\geq 1}(\theta t)^n n! \, \|
\widetilde{g}_n(\cdot,r,z,t,x)- \widetilde{g}_n(\cdot,r,z,t,x+h)\|_{\cH^{\otimes n}}^2.
\end{align*}
By direct calculation, it can be shown that
\[
\|
\widetilde{g}_n(\cdot,r,z,t,x)- \widetilde{g}_n(\cdot,r,z,t,x+h)\|_{\cH^{\otimes n}}^2=
\|
\widetilde{g}_n(\cdot,0,x,t-r,z)- \widetilde{g}_n(\cdot,0,x  +h,t-r,z)\|_{\cH^{\otimes n}}^2,
\]
and hence
\[
\left\|V_{\theta t}^{(r,z)}(t,x) - V_{\theta t}^{(r,z)}(t,x+h)\right\|_2^2=
\left\|V_{\theta t}^{(0,x)}(t-r,z) - V_{\theta t}^{(0,x+h)}(t-r,z)\right\|_2^2.
\]
We conclude that for any $z,x\in \bR$,
\begin{align*}
& \int_{\bR^2} \left\| V_{\theta t}^{(r,z)}(t,x') - V_{\theta t}^{(r,z)}(t,x'+h) \right\|_2^2 |h|^{2H-2}dx' dh\\
& \quad =\int_{\bR^2} \left\| V_{\theta t}^{(0,x')}(t-r,z) - V_{\theta t}^{(0,x'+h)}(t-r,z) \right\|_2^2 |h|^{2H-2}dx' dh \\
& \quad  =\int_{\bR^2} \left\| V_{\theta t}^{(0,x)}(t-r,z') - V_{\theta t}^{(0,x+h)}(t-r,z') \right\|_2^2 |h|^{2H-2}dz' dh,
\end{align*}
where the last equality is proved exactly as  \eqref{xz-integrals}, using some translation invariance properties of $V_{\theta}$ (similar to those given in Lemma \ref{translation-lemma} for $v_{\theta}$). Finally, due to relation \eqref{V-int1}, the last integral above is bounded by the constant $C_{V,2,H,t, \theta t}$.

e) For fixed $r>0$, $z\in \bR$ and $h\in\bR$, the process $U_{\theta}^{(r,z,h)}(t,x):=v_{\theta}^{(r,z)}(t,x)-
v_{\theta}^{(r,z+h)}(t,x)$ satisfies the following equation: for any $t \geq r$ and $x \in \bR$,
\[
U_{\theta}^{(r,z,h)}(t,x)=G_{t-r}(x-z)-G_{t-r}(x-z-h)+\int_r^t \int_{\bR} G_{t-s}(x-y)U_{\theta}^{(r,z,h)}(s,y)W(\delta y)ds.
\]
Therefore, by applying Lemma \ref{Lemma-hypercontractivity}, we infer that
\[
\|U_{\theta}^{(r,z,h)}(t,x)-U_{\theta}^{(r,z,h)}(t,x+k)\|_p \leq
\|U_{(p-1)\theta}^{(r,z,h)}(t,x)-U_{(p-1)\theta}^{(r,z,h)}(t,x+k)\|_2,
\]
which reduces the problem to the case $p=2$.
Next, we use the fact that for any $x,z \in \bR$,
\begin{align*}
& \left\| v_{\theta}^{(r,z)}(t,x) - v_{\theta}^{(r,z+h)}(t,x) - v_{\theta}^{(r,z)}(t,x+k) + v_{\theta}^{(r,z+h)}(t,x+k) \right\|_2^2 \leq \\
& \quad \quad \quad  \left\| V_{\theta t}^{(r,z)}(t,x) - V_{\theta t}^{(r,z+h)}(t,x) - V_{\theta t}^{(r,z)}(t,x+k) + V_{\theta t}^{(r,z+h)}(t,x+k) \right\|_2^2,
\end{align*}
which is proved similarly to \eqref{ineq-compare-norm}.
The conclusion follows from relation \eqref{V-int2}.
\end{proof}

The next result is derived using Lemma \ref{v-lem1}.

\begin{lemma}
	\label{Lem-int-v-v'}
	For any $\theta>0$, $p\geq 2$ and $t>0$,
	\begin{align*}
a) &	\quad	\sup_{r \in [0,t]} \int_{\bR} \sup_{z \in \bR} \left(  \int_{\bR}\left\| v_{\theta}^{(r,z)}(t,x) - v_{\theta}^{(r,z+h)}(t,x) \right\|_p dx \right)^2 |h|^{2H-2}  dh \leq C_{v,p,H,t,\theta}', \\
b) & \quad \sup_{r \in [0,t]}  \int_{\bR} \sup_{x \in \bR} \left( \int_{\bR}\left\| v_{\theta}^{(r,z)}(t,x) - v_{\theta}^{(r,z+h)}(t,x) \right\|_p dz \right)^2 |h|^{2H-2}  dh \leq C_{v,p,H,t,\theta}', \\
c) & \quad \sup_{r \in [0,t]}  \int_{\bR} \sup_{z \in \bR} \left( \int_{\bR}\left\| v_{\theta}^{(r,z)}(t,x) - v_{\theta}^{(r,z)}(t,x+h) \right\|_p dx \right)^2 |h|^{2H-2}  dh \leq C_{v,p,H,t,\theta}', \\
d) & \quad \sup_{r \in [0,t]} \int_{\bR} \sup_{x \in \bR}  \left( \int_{\bR}\left\| v_{\theta}^{(r,z)}(t,x) - v_{\theta}^{(r,z)}(t,x+h) \right\|_p dz \right)^2 |h|^{2H-2}  dh \leq C_{v,p,H,t,\theta}',\\
e)	& \quad \sup_{r \in [0,t]}   \int_{\bR^2} \sup_{z \in \bR} \left( \int_{\bR}\left\| v_{\theta}^{(r,z)}(t,x) - v_{\theta}^{(r,z+h)}(t,x) - v_{\theta}^{(r,z)}(t,x+k) + v_{\theta}^{(r,z+h)}(t,x+k) \right\|_p dx \right)^2 \\
	& \qquad\qquad\qquad \times |h|^{2H-2} |k|^{2H-2} dhdk \leq C_{v,p,H,t,\theta}'\\
f)	& \quad \sup_{r \in [0,t]}   \int_{\bR^2} \sup_{x \in \bR}\left( \int_{\bR}\left\| v_{\theta}^{(r,z)}(t,x) - v_{\theta}^{(r,z+h)}(t,x) - v_{\theta}^{(r,z)}(t,x+k) + v_{\theta}^{(r,z+h)}(t,x+k) \right\|_p dz \right)^2 \\
	& \qquad\qquad\qquad \times |h|^{2H-2} |k|^{2H-2} dhdk \leq C_{v,p,H,t,\theta}',
	\end{align*}
where $C_{v,p,H,t,\theta}'>0$ is a constant that depends on $(p,H,t,\theta)$ and is increasing in $t$.
\end{lemma}

\begin{proof}
By Remark \ref{trans-rem}, it suffices to prove a), c) and e).

a) We denote by $I$ the integral appearing in the this inequality. Then $I=I_1+I_2$, where $I_1$ and $I_2$ are the integrals on $|h|>1$, respectively $|h|\leq 1$. By triangle inequality and Lemma \ref{sup-v-lem}, we have
	\[
	\int_{\bR} \left\| v_{\theta}^{(r,z)}(t,x) - v_{\theta}^{(r,z+h)}(t,x) \right\|_p dx \leq \int_{\bR}\left\| v_{\theta}^{(r,z)}(t,x)\right\|_p dx+\int_{\bR} \left\| v_{\theta}^{(r,z+h)}(t,x) \right\|_p dx \leq 4t C_{\theta,p,t,v},
	\]
	and hence $I_1 \leq (4t^2 C_{\theta,p,t,v})^2 \int_{|h|>1}|h|^{2H-2}dh<\infty$.
	
	It remains to treat $I_2$. For this, we insert artificially the $G$ function, using identity \eqref{eq-identity-v}.
%More precisely, by \eqref{eq-identity-v} and
Then, by triangle inequality,
\begin{align*}
%\label{ineq-norm-decom}
\left\| v_{\theta}^{(r,z)}(t,x) - v_{\theta}^{(r,z+h)}(t,x) \right\|_p
%		& = 2\left\| G_{t-r}(x-z) v_{\theta}^{(r,z)}(t,x) - %G_{t-r}(x-z-h) v_{\theta}^{(r,z+h)}(t,x) \right\|_p  \nonumber %\\
& \leq 2G_{t-r}(x-z) \left\| v_{\theta}^{(r,z)}(t,x) - v_{\theta}^{(r,z+h)}(t,x) \right\|_p  \\
%\nonumber \\
& + 2\big|G_{t-r}(x-z) - G_{t-r}(x-z-h)\big| \left\| v_{\theta}^{(r,z+h)}(t,x) \right\|_p.
\end{align*}

Hence $I_2 \leq  8 \big(I_{2,1}+I_{2,2}\big)$, where
\begin{align*}
I_{2,1} & = \int_{|h|\leq 1}\left(  \int_{\bR}
G_{t-r}(x-z) \left\| v_{\theta}^{(r,z)}(t,x) - v_{\theta}^{(r,z+h)}(t,x) \right\|_p dx   \right)^2 |h|^{2H-2} dh\\
I_{2,2}&=  \int_{|h|\leq 1}\left(  \int_{\bR}
\big|G_{t-r}(x-z) - G_{t-r}(x-z-h)\big| \left\| v_{\theta}^{(r,z+h)}(t,x) \right\|_p dx   \right)^2 |h|^{2H-2} dh.
\end{align*}
	
We treat $I_{2,1}$. The the presence of $G$ allows us to apply the Cauchy-Schwarz inequality to bring the square inside the integral. Using the fact that $\int_{\bR}G_{t-r}^2(x-z)dx=(t-r)/2$, we obtain:
\begin{align*}
		& \left(  \int_{\bR}
		G_{t-r}(x-z) \left\| v_{\theta}^{(r,z)}(t,x) - v_{\theta}^{(r,z+h)}(t,x) \right\|_p dx   \right)^2 \leq t \int_{\bR} \left\| v_{\theta}^{(r,z)}(t,x) - v_{\theta}^{(r,z+h)}(t,x) \right\|_p^2 dx.
\end{align*}
The desired bound for $I_{2,1}$ follows now from Lemma \ref{v-lem1}.a).
	
For $I_{2,2}$, we use Theorem \ref{Thm-exist-v}, and $\int_{\bR}  \big|G_{t-r}(x-z) - G_{t-r}(x-z-h)\big|  dx \leq |h|$. Hence,
\begin{align*}
		I_{2,2}	& \leq  C_{\theta,p,t,v}^2 \int_{|h|\leq 1} \left( \int_{\bR} \big|G_{t-r}(x-z) - G_{t-r}(x-z-h)\big|  dx  \right)^2  |h|^{2H-2}dh \\
		& \leq C_{\theta,p,t,v}^2 \int_{|h|\leq 1} |h|^{2H} dh<\infty.
\end{align*}

c) We use the same argument as in part a), using Lemma \ref{v-lem1}.c).

e) We denote by $I'$ the integral appearing in this inequality. Inserting artificially the term $G_{t-r}(x-z)$ (using identity \eqref{eq-identity-v}), we obtain:
\begin{align*}
& v_{\theta}^{(r,z)}(t,x) - v_{\theta}^{(r,z+h)}(t,x) - v_{\theta}^{(r,z)}(t,x+k) + v_{\theta}^{(r,z+h)}(t,x+k)= \\
& \quad \quad \quad 2G_{t-r}(x-z) v_{\theta}^{(r,z)}(t,x) - 2G_{t-r}(x-z-h) v_{\theta}^{(r,z+h)}(t,x)- \\
& \quad \quad \quad  2G_{t-r}(x+k-z) v_{\theta}^{(r,z)}(t,x+k) + 2G_{t-r}(x+k-z-h) v_{\theta}^{(r,z+h)}(t,x+k).
\end{align*}
We now use the following identity:
\begin{align*}
& a_1b_1 - a_2b_2 - a_3b_3 + a_4b_4 \\
& \quad =a_1(b_1-b_2) + (a_1-a_2)b_2 - a_3(b_3-b_4) - (a_3-a_4)b_4 \\
& \quad =a_1(b_1-b_2-b_3+b_4) + (a_1-a_3)(b_3-b_4) + (a_1-a_2-a_3+a_4)b_2 + (a_3-a_4)(b_2-b_4).
\end{align*}
Using this identity together with triangle inequality, we have
\begin{align*}
& \left\| v_{\theta}^{(r,z)}(t,x) - v_{\theta}^{(r,z+h)}(t,x) - v_{\theta}^{(r,z)}(t,x+k) + v_{\theta}^{(r,z+h)}(t,x+k)\right\|_p  \\
& \quad \leq 2G_{t-r}(x-z) \left\| v_{\theta}^{(r,z)}(t,x) - v_{\theta}^{(r,z+h)}(t,x) - v_{\theta}^{(r,z)}(t,x+k) + v_{\theta}^{(r,z+h)}(t,x+k) \right\|_p \\
& \quad + |2G_{t-r}(x-z) - 2G_{t-r}(x+k-z)| \left\| v_{\theta}^{(r,z)}(t,x+k) - v_{\theta}^{(r,z+h)}(t,x+k) \right\|_p \\
& \quad + |2G_{t-r}(x-z) - 2G_{t-r}(x-z-h) - 2G_{t-r}(x+k-z) + 2G_{t-r}(x+k-z-h)| \left\| v_\theta^{(r,z+h)}(t,x) \right\|_p \\
& \quad + |2G_{t-r}(x+k-z) - 2G_{t-r}(x+k-z-h)| \left\| v_\theta^{(r,z+h)}(t,x) - v_\theta^{(r,z+h)}(t,x+k) \right\|_p \\
& := \sum_{j=1}^4 F_j(x,h,k). %F_1+F_2+F_3+F_4.
\end{align*}
Using the inequality $(\sum_{j=1}^4 e_j)^2 \le 4 \sum_{j=1}^4 e_j^2$, we have
\begin{align*}
%	&\int_{\bR^2} \left( \int_{\bR}\left\| %v_{\theta}^{(r,z)}(t,x) - v_{\theta}^{(r,z+h)}(t,x) - %v_{\theta}^{(r,z)}(t,x+k) + v_{\theta}^{(r,z+h)}(t,x+k) \right\|_p dx \right)^2  |h|^{2H-2} |k|^{2H-2} dhdk \\
I' \leq 4 \sum_{j=1}^4 \int_{\bR^2} \left( \int_{\bR} F_j(x,h,k) dx \right)^2 |h|^{2H-2} |k|^{2H-2} dhdk =:4 \sum_{j=1}^4 I_j'.
\end{align*}

Since $F_j$ is a product of two terms (which involve $G$, respectively $v_{\theta}$), we bound $\left( \int_{\bR} F_j(x,h,k) dx \right)^2$ using the Cauchy-Schwarz inequality. For $I_1'$, we use Lemma \ref{v-lem1}.e) and $\int_{\bR} G_{t-r}^2(x-r)dx=(t-r)/2$.
For $I_2'$, we use Lemma \ref{v-lem1}.a), and
\begin{equation}
\label{J1-G}
\int_{\bR^2}|G_{t-r}(x-z)-G_{t-r}(x+k-z)|^2 |k|^{2H-2} dxdk=c_H (t-r)^{2H},
\end{equation}
where $c_H>0$ is a constant depending on $H$. For $I_3'$, we use Lemma \ref{sup-v-lem} and
\begin{align}
\nonumber
& \int_{\bR^3} |G_{t-r}(x-z) - G_{t-r}(x-z-h) - G_{t-r}(x+k-z) + G_{t-r}(x+k-z-h)|^2 \\
\label{J3-G}
& \quad \quad \quad |h|^{2H-2}|k|^{2H-2} dhdk dx= c_H'(t-r)^{4H-1},
\end{align}
where $c_H'>0$ is a constant that depends on $H$.
For $I_4'$, we use Lemma \ref{v-lem1}.c) and \eqref{J1-G}.
\end{proof}

\begin{remark}
{\rm
As stated in Remark \ref{trans-rem}, the integral $dx$ which appears in part a) of Lemma \ref{Lem-int-v-v'} does not depend on $z$, so instead of taking the supremum over $z \in \bR$, we could have evaluated this integral at one particular point $z_0$, for instance $z_0=0$. We chose to present the result in this way, since this will be the form that we will use in Section \ref{sec:QCLT} below, for the proof of the QCLT. The same comment applies to parts b)-f).
}
\end{remark}

\section{Proofs of the main results}
\label{section-proofs}

In this section, we include the proofs of Theorems \ref{limit-cov}, \ref{QCLT} and \ref{FCLT}.

\subsection{Limiting covariance}

In this section, we give the proof of Theorem \ref{limit-cov}.
Let
\begin{equation}
\label{def-rho}
	\rho_{t,s,\theta}(x-y)
	:=\bE \left[ \big(u_{\theta}(t,x)-1\big) \big(u_{\theta}(s,y)-1\big) \right]
	= \sum_{n\geq 1}\frac{1}{n!}\theta^n \alpha_n(x-y;t,s),
\end{equation}
where
\begin{align}
\nonumber
	& \alpha_n(x-y;t,s)
	=(n!)^2 \langle \widetilde f_n(\cdot,x;t), \widetilde f_n(\cdot,y;s) \rangle_{\cP_0^{\otimes n}}
	=(n!)^2 \langle f_n(\cdot,x;t), \tilde f_n(\cdot,y;s) \rangle_{\cP_0^{\otimes n}} \nonumber \\
	=& n! \sum_{\sigma \in S_n} c_H^n \int_{\bR^n}
	\cF f_n(\cdot,x;t)(\xi_1,\ldots,\xi_n) \overline{\cF f_n(\cdot,y;s)(\xi_{\sigma(1)},\ldots,\xi_{\sigma(n)})}
	\prod_{j=1}^{n}|\xi_j|^{1-2H} d\pmb{\xi} \nonumber \\
	=& n! \sum_{\sigma \in S_n} c_H^n \int_{\bR^n} e^{-i \sum_{j=1}^n \xi_j (x-y)} \prod_{j=1}^n |\xi_j|^{1-2H} d\xi_1 \ldots d\xi_n
	\int_{T_n(t)} d\pmb{t} \int_{T_n(s)} d\pmb{s} \nonumber \\
\label{eq-inner}
	& \times \prod_{j=1}^n \cF G_{t_{j+1}-t_j} (\xi_1+\ldots+\xi_j) \prod_{j=1}^n \cF G_{s_{j+1}-s_j} (\xi_{\sigma(1)}+\ldots+\xi_{\sigma(j)}).
\end{align}
Here we use the convention $t_{n+1}=t$ and $s_{n+1}=s$. This shows that $\alpha_n(x-y;t,s)$ and $\rho_{t,s}(x-y)$ depend on $x$ and $y$ only through the difference $x-y$. In particular, $\{u(t,x);x \in \bR\}$ is a (wide-sense) stationary process with covariance function $\rho_{t,s,\theta}$.

By Fubini theorem,
\begin{align}
\nonumber
\frac{1}{R}\bE[F_{R,\theta}(t) F_{R,\theta}(s)]&=\frac{1}{R}\int_{-R}^R \int_{-R}^R \rho_{t,s,\theta}(x-y)dxdy=
\sum_{n\geq 1}\theta^n \frac{1}{Rn!} \int_{-R}^R \int_{-R}^R
\alpha_n(x-y;t,s)dxdy\\
\label{series-Q}
&=: \sum_{n\geq 1}\theta^n Q_{n,R}(t,s),
\end{align}
The application of Fubini theorem is justified by \eqref{bound-f}, using that fact that:
\[
\alpha_n(x-y;t,s) \leq (n!)^2
\| f_n(\cdot,x;t)\|_{\cP_0^{\otimes n}}\|\tilde f_n(\cdot,y;s) \|_{\cP_0^{\otimes n}} \leq c^n \frac{t^{(2H+2)n}}{(n!)^{2H}}.
\]

We now show that the first term in series \eqref{series-Q} does not contribute to the limit.

\begin{lemma}
For any $t,s>0$,
\[
Q_{1,R}(t,s)=\frac{1}{R}\int_{-R}^{R} \int_{-R}^{R}
\alpha_1(x-y;t,s)dxdy \to 0 \quad \mbox{as $R \to \infty$}.
\]
\end{lemma}

\begin{proof} By Fubini theorem,
\begin{align*}
\alpha_1(x-y;t,s)&=c_H \int_{\bR} \left( \int_0^t e^{-i\xi x}\frac{\sin((t-t_1)|\xi|)}{|\xi|}dt_1\right)
\left( \int_0^s e^{i\xi y}\frac{\sin((s-s_1)|\xi|)}{|\xi|}ds_1\right)\\
&=c_H \int_0^t \int_0^s \int_{\bR}e^{-i \xi(x-y)}
\frac{\sin((t-t_1)|\xi|)}{|\xi|}
\frac{\sin((s-s_1)|\xi|)}{|\xi|} d\xi ds_1 dt_1.
\end{align*}
The application of this theorem is justified using the inequality:
\begin{equation}
\label{Fub-just}
\frac{|\sin((t-t_1)|\xi|)|}{|\xi|}
\frac{|\sin((s-s_1)|\xi|)|}{|\xi|} \leq (t-t_1)(s-s_1)1_{\{ |\xi|\leq 1\}}+\frac{1}{|\xi|^2}1_{\{ |\xi|> 1\}}.
\end{equation}
Another application of Fubini theorem, justified also by \eqref{Fub-just}, shows that:
\begin{align*}
& \int_{-R}^{R}
\int_{-R}^{R} \alpha_1(x-y;t,s)dxdy \\
& \quad = c_H \int_0^t \int_0^s \int_{\bR} \frac{4\sin^2(R|\xi|)}{|\xi|^2}
\frac{\sin((t-t_1)|\xi|)}{|\xi|}
\frac{\sin((s-s_1)|\xi|)}{|\xi|} |\xi|^{1-2H} d\xi ds_1 dt_1,
\end{align*}
where we used the fact that:
\[
\left(\int_{-R}^{R}
\int_{-R}^{R} e^{-i \xi(x-y)} dxdy  \right)=
\left|\int_{-R}^{R} e^{-i \xi x} dx\right|^2 =\frac{4\sin^2(R|\xi|)}{|\xi|^2} :=4 \pi R \ell_{R}(\xi),
\]

We denote
\[
I_1(t_1,s_1):=\int_{\bR} \frac{\sin^2(R|\xi|)}{|\xi|^2}
\frac{\sin((t-t_1)|\xi|)}{|\xi|}
\frac{\sin((s-s_1)|\xi|)}{|\xi|} |\xi|^{1-2H} d\xi=I_1(t_1,s_1)+I_2(t_1,s_1),
\]
where $I_1(t_1,s_1)$ and $I_2(t_1,s_1)$ correspond to the integration regions $\{|\xi|\leq \e\}$, respectively $\{|\xi|>\e\}$. To estimate $I_2(t_1,s_1)$, we bound all three $\sin$ functions by $1$, obtaining
\[
|I_2(t_1,s_1)| \leq \int_{|\xi|>\e} |\xi|^{-3-2H}d\xi=:C_{\e,H}.
\]
To bound $I_1(t_1,s_1)$, we choose $\e>0$ such that $(t+s)\e<\frac{\pi}{2}$, and hence $\sin((t-t_1)|\xi|)\geq 0$ and
$\sin((s-s_1)|\xi|)\geq 0$ for any $t_1 \in [0,t],s_1 \in[0,s]$ and $\xi$ with $|\xi|\leq \e$. Using the inequality $|\sin(x)|\leq |x|$ for the last two $\sin$ functions, followed by the inequality $|\sin x|\leq |x|^{\theta}$ with $\theta \in [0,1]$ for the last $\sin$ function, we obtain that:
\begin{align*}
I_1(t_1,s_1) &\leq (t-t_1)(s-s_1)\int_{|\xi|\leq \e}\frac{\sin^2(R|\xi|)}{|\xi|^2}|\xi|^{1-2H} d\xi\\
& \leq (t-t_1)(s-s_1)R^{2\theta} \int_{|\xi|\leq \e}|\xi|^{2\theta-2H-1} d\xi=:(t-t_1)(s-s_1)R^{2\theta} C_{\e,\theta,H},
\end{align*}
provided that $\theta>H$. Summarizing, for any $t_1 \in [0,t]$, $s_1 \in [0,s]$ and $\theta \in (H,1]$,
\[
|I(t_1,s_1)|\leq (t-t_1)(s-s_1)R^{2\theta} C_{\e,\theta,H} +C_{\e,H}.
\]
Therefore,
\[
\left|\int_{-R}^{R}
\int_{-R}^{R} \alpha_1(x-y;t,s)dxdy\right| \leq 4c_H \int_0^t \int_0^s |I(t_1,s_1)| ds_1 dt_1 \leq C (t^2 s^2 R^{2\theta}+ts).
\]
Then $|Q_{1,R}(t,s)| \leq C (t^2 s^2 R^{2\theta-1}+tsR^{-1}) \to 0$ as $R \to \infty$, if we choose $\theta \in (H,\frac{1}{2})$.
\end{proof}

Next, we study the terms corresponding to $n\geq 2$. We use the fact that the function
\[
\ell_R(\xi)=\frac{\sin^2(R|\xi|)}{\pi R |\xi|^2}, \quad \xi \in \bR
\]
is an approximation of the identity, when $R \to \infty$ (see Lemma 2.1 of \cite{NZ20}).

\begin{lemma}
For any $t>0$ and $s>0$,
\[
K_{\theta}(t,s):=\lim_{R \to \infty}\sum_{n\geq 2}\theta^n Q_{n,R}(t,s) \quad
\mbox{exists and is finite}.
\]
\end{lemma}

\begin{proof}
We use expression \eqref{eq-inner} for $\alpha_n(x-y;t,s)$, and we integrate $dxdy$ on $[-R,R]^2$. Using Fubini's theorem, we obtain the following expression for $Q_{n,R}(t,s)$:
\begin{align}
\nonumber
Q_{n,R}(t,s) &= \frac{1}{R}\frac{1}{n!}\int_{-R}^{R}
\int_{-R}^{R} \alpha_n(x-y;t,s) dxdy \\
\nonumber
&=4\pi c_H^n \sum_{\sigma \in S_n}  \int_{T_n(t)}  \int_{T_n(s)}  \int_{\bR^n} \ell_R(\xi_1+\ldots+\xi_n) \prod_{j=1}^n |\xi_j|^{1-2H}  \\
\nonumber
& \times \dfrac{\sin((t-t_n)|\xi_1+\ldots+\xi_n|)}{|\xi_1+\ldots+\xi_n|} \dfrac{\sin((s-s_n)|\xi_1+\ldots+\xi_n|)}{|\xi_1+\ldots+\xi_n|}\\
\nonumber
& \times \prod_{j=1}^{n-1} \dfrac{\sin((t_{j+1}-t_j)|\xi_1+\ldots+\xi_j|)}{|\xi_1+\ldots+\xi_j|} \prod_{j=1}^{n-1} \dfrac{\sin((s_{j+1}-s_j)|\xi_{\sigma(1)}+\ldots+\xi_{\sigma(j)}|)}{|\xi_{\sigma(1)}+\ldots+\xi_{\sigma(j)}|} d\pmb{\xi}d\pmb{s} d\pmb{t}\\
\label{eq-Q}
& :=4\pi c_H^n \sum_{\sigma \in S_n}  \int_{T_n(t)}\int_{T_n(s)}  Q_{n,R}(\pmb{t},\pmb{s},\sigma)
d\pmb{s} d\pmb{t}
\end{align}

\noindent To evaluate $Q_{n,R}(\pmb{t},\pmb{s},\sigma)$,
we use the change the variables $\xi_n \to \eta_n = \xi_1+\ldots+\xi_n$. For any fixed permutation $\sigma$, we denote $j_{0} = \sigma^{-1}(n)$, so that $\sigma(j_{0}) = n$. Then for all $j<j_0$, the sum $\xi_{\sigma(1)}+\ldots+\xi_{\sigma(j)}$ does not involve the variable $\xi_n$, and remains the same when we perform the change of variables. For $j \ge j_0$, the sum
$\xi_{\sigma(1)}+\ldots+\xi_{\sigma(j)}$ contains $\xi_n$ and will be replaced after the change of variables by
\[
\sum_{k=1}^{j_0-1}\xi_{\sigma(k)}+\big(\eta_n-\sum_{j=1}^{n-1}\xi_j\big)+
\sum_{k=j_0+1}^{j}\xi_{\sigma(k)}
= \eta_n - \xi_{\sigma(j+1)}-\ldots-\xi_{\sigma(n)},
\]
where $\sum_{\emptyset}=0$.
Using the convention $\prod_{\emptyset}=1$, we obtain the expression:
\begin{align*}
& Q_{n,R}(\pmb{t},\pmb{s},\sigma) =\int_{\bR^n} \ell_R(\eta_n)
    \dfrac{\sin((t-t_n)|\eta_n|)}{|\eta_n|} \dfrac{\sin((s-s_n)|\eta_n|)}{|\eta_n|} \\
& \quad  \times \prod_{j=1}^{n-1} \dfrac{\sin((t_{j+1}-t_j)|\xi_1+\ldots+\xi_j|)}{|\xi_1+\ldots+\xi_j|} \prod_{j<\sigma^{-1}(n)} \dfrac{\sin((s_{j+1}-s_j)|\xi_{\sigma(1)}+\ldots+\xi_{\sigma(j)}|)}{|\xi_{\sigma(1)}+\ldots+\xi_{\sigma(j)}|} \\
&  \quad  \times \prod_{j\ge \sigma^{-1}(n)} \dfrac{\sin((s_{j+1}-s_j)|\eta_n - \xi_{\sigma(j+1)} - \ldots - \xi_{\sigma(n)}|)}{|\eta_n - \xi_{\sigma(j+1)} - \ldots - \xi_{\sigma(n)}|} \\
&  \quad \times \prod_{j=1}^{n-1} |\xi_j|^{1-2H} |\eta_n - \xi_1 - \ldots - \xi_{n-1}|^{1-2H} d\xi_1\ldots d\xi_{n-1} d\eta_n  :=
\left( \ell_R * g_{{\bf t},{\bf s},\sigma}^{(n)}\right)(0),
\end{align*}
where the function $g_{{\pmb t},\pmb{s},\sigma}^{(n)}$ is given by:
\begin{align*}
g_{\pmb{t},\pmb{s},\sigma}^{(n)}(x) &:= \dfrac{\sin((t-t_n)|x|)}{|x|} \dfrac{\sin((s-s_n)|x|)}{|x|} \\
    & \times \int_{\bR^{n-1}}  \prod_{j=1}^{n-1} \dfrac{\sin((t_{j+1}-t_j)|\xi_1+\ldots+\xi_j|)}{|\xi_1+\ldots+\xi_j|} \prod_{j<\sigma^{-1}(n)} \dfrac{\sin((s_{j+1}-s_j)|\xi_{\sigma(1)}+\ldots+\xi_{\sigma(j)}|)}{|\xi_{\sigma(1)}+\ldots+\xi_{\sigma(j)}|} \\
    & \times \prod_{j\ge \sigma^{-1}(n)} \dfrac{\sin((s_{j+1}-s_j)|x - \xi_{\sigma(j+1)} - \ldots - \xi_{\sigma(n)}|)}{|x - \xi_{\sigma(j+1)} - \ldots - \xi_{\sigma(n)}|} \\
    & \times \prod_{j=1}^{n-1} |\xi_j|^{1-2H} |x - \xi_1 - \ldots - \xi_{n-1}|^{1-2H} d\xi_1\ldots d\xi_{n-1}.
\end{align*}
To define $g_{{\bf t},{\bf s},\sigma} (0)$, we use the convention $\frac{\sin x}{x}\big|_{x=0}=1$.

Since $\ell_R$ is an approximation of the identity as $R \to \infty$, by the Dominated Convergence Theorem, we obtain that
\begin{align}
\nonumber
\sum_{n\geq 2}\theta^n Q_{n,R}(t,s)&= 4\pi \sum_{n\geq 2}\theta^n c_H^n \int_{T_n(t)}  \int_{T_n(s)} \sum_{\sigma \in S_n} \left( \ell_R * g_{{\bf t},{\bf s},\sigma} \right)(0) d{\bf s}  d{\bf t}\\
\label{def-K}
   & \to 4\pi \sum_{n\geq 2}\theta^n c_H^n \int_{T_n(t)} \int_{T_n(s)} \sum_{\sigma \in S_n} g_{{\bf t},{\bf s},\sigma} (0) d{\bf s} d{\bf t}  =:K_{\theta}(t,s), \quad \mbox{as $R \to \infty$}.
\end{align}

It remains to justify the application of the Dominated Convergence Theorem. For this, we will prove that there exists a function $H_n(\pmb{t},\pmb{s})$ such that
\[
\left|\sum_{\sigma \in S_n} \left( \ell_R * g_{{\bf t},{\bf s},\sigma} \right)(0)\right| \leq H_n(\pmb{t},\pmb{s})
\]
for any $n\geq 2$, $R\geq 1$, $\pmb{t}\in T_n(t)$, $\pmb{s}\in T_n(s)$, and
\begin{equation}
\label{bound-H}
\sum_{n\geq 2}\theta^n c_H^n \int_{T_n(t)} \int_{T_n(s)}H_n(\pmb{t},\pmb{s}) d \pmb{t} d\pmb{s}<\infty.
\end{equation}
In particular, this will imply that the limit $K_{\theta}(t,s)$ is finite.

We will use the following inequality: for any functions $h_1,h_2$ on $\bR^n$ and for any symmetric measure $\mu_n$ on $\bR^n$,
\begin{align*}
& \left|\sum_{\sigma \in S_n} \int_{\bR^n}h_1(\xi_1,\ldots,\xi_n)
h_2(\xi_{\sigma(1)},\ldots,\xi_{\sigma(n)})
\mu_n(d\xi_1,\ldots,d\xi_n) \right| \leq \\
& \quad \frac{n!}{2} \left\{
\int_{\bR^n}|h_1(\xi_1,\ldots,\xi_n)|^2 \mu_n(d\xi_1,\ldots,d\xi_n)+\int_{\bR^n}|h_2(\xi_1,\ldots,\xi_n)|^2 \mu_n(d\xi_1,\ldots,d\xi_n) \right\},
\end{align*}
which can be proved applying the inequality $ab \leq \frac{1}{2}(a^2+b^2)$. It follows that
\begin{align*}
\left|\sum_{\sigma \in S_n} \left( \ell_R * g_{{\bf t},{\bf s},\sigma} \right)(0) \right|
& \leq \frac{n!}{2}\left\{\int_{\bR^n}
\ell_{R}(\xi_1+\ldots+\xi_n)
\prod_{j=1}^{n} \frac{\sin^2((t_{j+1}-t_j)|\xi_1+\ldots+\xi_j|)}{|\xi_1+\ldots+\xi_j|^2}
\prod_{j=1}^{n}|\xi_j|^{1-2H}d\pmb{\xi}\right.\\
& \quad \left. +\int_{\bR^n}
\ell_{R}(\xi_1+\ldots+\xi_n)
\prod_{j=1}^{n} \frac{\sin^2((s_{j+1}-s_j)|\xi_1+\ldots+\xi_j|)}{|\xi_1+\ldots+\xi_j|^2}
\prod_{j=1}^{n}|\xi_j|^{1-2H} d\pmb{\xi} \right\}\\
&=:\frac{n!}{2}\Big(\cI_{n,R}({\pmb t})+\cI_{n,R}({\pmb s})\Big).
\end{align*}
We only evaluate the first integral, the second one being similar. We use the change of variables $\eta_j=\xi_1+\ldots+\xi_j$ for $j=1,\ldots,n$ (with $\eta_0=0$), followed by inequality \eqref{prod}:
\begin{align*}
\cI_{n,R}({\pmb t})&\leq \sum_{\alpha \in D_n} \prod_{j=1}^{n-1}\left( \int_{\bR} \frac{\sin^2((t_{j+1}-t_j)|\eta_j|)}{|\eta_j|^2}|\eta_j|^{\alpha_j}
d\eta_j\right)
\left( \int_{\bR} \frac{\sin^2((t-t_n)|\eta_n|)}{|\eta_n|^2}
\ell_R(\eta_n) |\eta_n|^{\alpha_n}
d\eta_n\right)\\
&\leq c^{n-1}\sum_{\alpha \in D_n} \prod_{j=1}^{n-1}(t_{j+1}-t_j)^{1-\alpha_j}(t-t_n)^2 \frac{1}{\pi R}\int_{\bR}
\frac{\sin^2(R|\eta_n|)}{|\eta_n|^2} |\eta_n|^{\alpha_n}
d\eta_n\\
&=c^n \sum_{\alpha \in D_n} \prod_{j=1}^{n-1}(t_{j+1}-t_j)^{1-\alpha_j} (t-t_n)^2 R^{-\alpha_n} \leq c^n \sum_{\alpha \in D_n} t^{1+\alpha_n} \prod_{j=1}^{n}(t_{j+1}-t_j)^{1-\alpha_j}
\end{align*}
Therefore, assuming that $s\leq t$, we obtain:
\begin{align*}
\left|\sum_{\sigma \in S_n} \left( \ell_R * g_{{\bf t},{\bf s},\sigma} \right)(0) \right|
& \leq n!
 c^n \sum_{\alpha \in D_n}t^{1+\alpha_n}\Big(\prod_{j=1}^{n}(t_{j+1}-t_j)^{1-\alpha_j}
 +
 \prod_{j=1}^{n}(s_{j+1}-s_j)^{1-\alpha_j}\Big)=:H_n(\pmb{t},\pmb{s}).
\end{align*}
To see that \eqref{bound-H} holds, we note that
\begin{align*}
 \int_{T_n(t)}\int_{T_n(s)}H_n(\pmb{t},\pmb{s}) d\pmb{s}
d\pmb{t}
& \leq 2c^n  \sum_{\alpha \in D_n}t^{1+\alpha_n}\int_{T_n(t)}
 \prod_{j=1}^{n}(t_{j+1}-t_j)^{1-\alpha_j}d\pmb{t}\\
& \leq 2c^n \sum_{\alpha \in D_n}t^{1+\alpha_n} \frac{t^{(2H+1)n}}{\Gamma((2H+1)n+1)}.
\end{align*}

\end{proof}

Finally, we prove that when $s=t$, the limit $K_{\theta}(t,t)$ is non-zero. For this, we use the following lemma.

\begin{lemma}
\label{sym-lem}
If $h$ is a measurable function on $(\bR^d)^n$ and $\mu_n$ is a symmetric measure on $(\bR^d)^n$, then
%in the sense that $\mu_n(d\xi_1,\ldots,d\xi_n) = %\mu_n(d\xi_{\sigma(1)},\ldots,d\xi_{\sigma(n)})$ for any %permutation $\sigma \in S_n$. Then we have
\begin{align}
\label{sym-sum}
& \sum_{\sigma \in S_n} \int_{(\bR^d)^n} h(\xi_1,\ldots,\xi_n) h(\xi_{\sigma(1)},\ldots,\xi_{\sigma(n)}) \mu_n(d\xi_1,\ldots,d\xi_n) \\
\nonumber
& \quad \quad \quad =\frac{1}{n!} \int_{(\bR^d)^n} |\widetilde{h}(\xi_1,\ldots,\xi_n)|^2 \mu_n(d\eta_1,\ldots,d\eta_n).
\end{align}
%where $\widetilde{h}$ is the symmetrization of $h$.
\end{lemma}

\begin{proof}
We denote by $I$ the left-hand side of \eqref{sym-sum}.
For any $\rho \in S_n$, let $h_{\rho}(\xi_1,\ldots,\xi_n)=h(\xi_{\rho(1)},\ldots,\xi_{\rho(n)})$.
Then for any $\rho \in S_n$, $I=n!\int h \widetilde{h} d\mu_n=n!\int h_{\rho} \widetilde{h} d\mu_n$, where the second equality is due to the symmetry of $\mu_n$. Taking the sum over all $\rho \in S_n$, we obtain $n! I=n!\int \sum_{\rho \in S_n}h_{\rho} \widetilde{h} d\mu_n=(n!)^2 \int \widetilde{h}^2 d\mu_n$.
\end{proof}

\begin{lemma}
For any $t>0$, $K_{\theta}(t,t)>0$.
\end{lemma}

\begin{proof}
We will first prove that $Q_{n,R}(t,t) \geq 0$ for all $n\geq 1$ and $R>0$. For this, we use \eqref{eq-Q}. Applying Lemma \ref{sym-lem} to the measure
\[
\mu_n (d\xi_1,\ldots,d\xi_n) = \ell_R(\xi_1+\ldots+\xi_n) \prod_{j=1}^n |\xi_j|^{1-2H} d\xi_1\ldots d\xi_n
\]
and the function
\[
h(\xi_1,\ldots,\xi_n) = \int_{T_n(t)} \dfrac{\sin((t-t_n)|\xi_1+\ldots+\xi_n|)}{|\xi_1+\ldots+\xi_n|} \prod_{j=1}^{n-1} \dfrac{\sin((t_{j+1}-t_j)|\xi_1+\ldots+\xi_j|)}{|\xi_1+\ldots+\xi_j|} d\pmb{t},
\]
we obtain:
\begin{align*}
    Q_{n,R}(t,t)
    %&= 4 \pi c_H^n \sum_{\sigma \in S_n} \int_{\bR^n} %h(\xi_1,\ldots,\xi_n) %h(\xi_{\sigma(1)},\ldots,\xi_{\sigma(n)}) %\mu_n(d\xi_1,\ldots,d\xi_n) \\
    &= 4 \pi c_H^n n! \int_{\bR^n} |\widetilde{h}(\xi_1,\ldots,\xi_n)|^2 \mu_n(d\xi_1,\ldots,d\xi_n) \geq 0.
\end{align*}

This implies that $Q_n(t,t):=\lim_{R \to \infty} Q_{n,R}(t,t) \geq 0$ for all $n\geq 1$, and hence
$K_{\theta}(t,t)=\sum_{n\geq 2} \theta^n Q_n(t,t)\geq \theta^2 Q_2(t,t)$. Therefore, it is enough to prove that $Q_2(t,t)>0$.

To show this, we will use an alternative expression of $Q_n(t,t)$. We denote $H(\xi_1,\ldots,\xi_n)=|\widetilde{h}(\xi_1,\ldots,\xi_n)|^2
\prod_{j=1}^n |\xi_j|^{1-2H}$, and we use the change of variables $\xi_n \mapsto \eta_n=\xi_1+\ldots+\xi_n$:
\begin{align*}
Q_{n,R}(t,t) &= 4 \pi c_H^n n! \int_{\bR^n} H(\xi_1,\ldots,\xi_n) \ell_R(\xi_1+\ldots+\xi_n)d\xi_1 \ldots d\xi_n \\
&= 4 \pi c_H^n n! \int_{\bR^n} H(-\xi_1,\ldots,-\xi_n) \ell_R(\xi_1+\ldots+\xi_n)d\xi_1 \ldots d\xi_n\\
&= 4 \pi c_H^n n! \int_{\bR^n} H(-\xi_1,\ldots,-\xi_{n-1},\xi_1+\ldots+\xi_{n-1}-\eta_n) \ell_R(\eta_n)d\xi_1 \ldots d\xi_{n-1}d\eta_n\\
&=4\pi c_H^n n! \int_{\bR^{n-1}} \Big(H(-\xi_1,\ldots,-\xi_{n-1},\cdot)*\ell_R\Big)
(\xi_1+\ldots+\xi_{n-1})d\xi_1\ldots d\xi_{n-1}.
\end{align*}
Taking $R \to \infty$, we obtain:
\[
Q_n(t,t)=
4\pi c_H^n n! \int_{\bR^{n-1}} H(-\xi_1,\ldots,-\xi_{n-1},\xi_1+\ldots+\xi_{n-1})d\xi_1\ldots d\xi_{n-1}.
\]
We consider now the case $n=2$. In this case,
\[
\widetilde{h}(\xi_1,\xi_2)=\frac{1}{2} \int_{T_2(t)} \frac{\sin((t-t_2)|\xi_1+\xi_2|)}{|\xi_1+\xi_2|}
\left(\frac{\sin((t_2-t_1)|\xi_1|)}{|\xi_1|} +
\frac{\sin((t_2-t_1)|\xi_2|)}{|\xi_2|}\right) dt_1 dt_2
\]
and
\begin{align*}
H(\xi_1,\xi_2)&=\frac{1}{4} \left(
\int_{T_2(t)} \frac{\sin((t-t_2)|\xi_1+\xi_2|)}{|\xi_1+\xi_2|}
\left(\frac{\sin((t_2-t_1)|\xi_1|)}{|\xi_1|} +
\frac{\sin((t_2-t_1)|\xi_2|)}{|\xi_2|}\right) dt_1 dt_2 \right)^2 \\
& |\xi_1|^{1-2H}|\xi_2|^{1-2H}.
\end{align*}
Using the convention $\frac{\sin x}{x}|_{x=0}=1$, by direct calculation, we have:
\begin{align*}
H(-\xi_1,\xi_1)&=\frac{1}{4}
\left( \int_{T_2(t)}
\frac{\sin((t_2-t_1)|\xi_1|)}{|\xi_1|} dt_1 dt_2 \right)^2 |\xi_1|^{2(1-2H)} = \left(t-\frac{\sin(t|\xi_1|)}{|\xi_1|} \right)^2 |\xi_1|^{-4H-2}.
\end{align*}
Since $H(-\xi_1,\xi_1)>0$  for almost all $\xi_1\in \bR$, $Q_2(t,t)=\int_{\bR}H(-\xi_1,\xi_1)d\xi_1>0$.
\end{proof}

\subsection{Quantitative CLT}
\label{sec:QCLT}

In this section, we prove Theorem \ref{QCLT}. We divide the proof into several steps. To simplify the notation, we drop the dependence on $\theta$ in this part, and we write $F_{R}(t)$ and $\sigma_{R}(t)$ instead of $F_{R,\theta}(t)$ and $\sigma_{R,\theta}(t)$.

\medskip

{\bf Step 1.} By applying a version of Proposition 2.4 of \cite{NXZ} for the time-independent noise, to $F=F_R(t)$, we get:
\[
d_{TV}\left(\frac{F_{R}(t)}{\sigma_{R}(t)},Z \right) \leq \frac{2\sqrt{3}}{\sigma_R^2(t)} \sqrt{C_H^3 \cA^*},
\]
where $Z \sim N(0,1)$ and
\begin{align*}
	\cA^*=&
	\int_{\bR^6} \|D_zF_R(t)- D_{z'}F_R(t)\|_4  \|D_wF_R(t)-D_{w'}F_R(t)\|_4 \\
	& \qquad \qquad \| D_{z,y}^2 F_R(t)- D_{z,y'}^2 F_R(t)- D_{z',y}^2 F_R(t)+ D_{z',y'}^2 F_R(t) \|_4 \\
	& \qquad \qquad \| D_{w,y}^2 F_R(t)- D_{w,y'}^2 F_R(t)- D_{w',y}^2 F_R(t)+ D_{w',y'}^2 F_R(t) \|_4 \\
	& \qquad \qquad |y-y'|^{2H-2}|z-z'|^{2H-2}|w-w'|^{2H-2}dydy' dzdz' dw dw'.
\end{align*}
By applying Minkowski's inequality to the four norms and change of variables $y' \rightarrow y+y', z' \rightarrow z+z', w' \rightarrow w+w'$, we have
\begin{align} \label{ineq-cA}
	\cA^* \le \cA
	=& \int_{[-R,R]^4} \int_{\bR^6} \left\| D_z u_{\theta}(t,x_1)- D_{z+z'} u_{\theta}(t,x_1) \right\|_4 \left\| D_w u_{\theta}(t,x_2) - D_{w+w'} u_{\theta}(t,x_2) \right\|_4 \nonumber \\
	& \left\| D_{z,y}^2 u_{\theta}(t,x_3)- D_{z,y+y'}^2 u_{\theta}(t,x_3)- D_{z+z',y}^2 u_{\theta}(t,x_3)+ D_{z+z',y+y'}^2 u_{\theta}(t,x_3) \right\|_4 \nonumber \\
	& \left\| D_{w,y}^2 u_{\theta}(t,x_4)- D_{w,y+y'}^2 u_{\theta}(t,x_4)- D_{w+w',y}^2 u_{\theta}(t,x_4)+ D_{w+w',y+y'}^2 u_{\theta}(t,x_4) \right\|_4 \nonumber \\
	& \qquad \qquad \qquad |y'|^{2H-2}|z'|^{2H-2}|w'|^{2H-2} dydy' dzdz' dwdw' dx_1 dx_2 dx_3 dx_4.
\end{align}
Since $\sigma_R^2(t) \sim C_t R$ as $R \to \infty$, it is enough to prove that $\cA \leq C_t R$.

For the first two norms, we use Theorem \ref{th-D-bound}.b) with $p=4$ (and hence $\eta=12\theta$). We obtain:
\begin{align}
\label{ineq-I1-QCLT-1}
	\left\| D_z u_{\theta}(t,x_1)- D_{z+z'}u_{\theta}(t,x_1) \right\|_4
	& \le \sqrt{2\theta} \int_0^t I_1(z,z',x_1;r,t) dr \\
\label{ineq-I1-QCLT-2}
	\left\| D_w u_{\theta}(t,x_2) - D_{w+w'} u_{\theta}(t,x_2) \right\|_4
	& \le \sqrt{2\theta} \int_0^t I_1(w,w',x_2;r,t) dr.
\end{align}

%For the second norm, replacing $(z,z',x_1)$ by $(w,w',x_2)$, we %have

For the last two norms, we use Theorem \ref{th-D2-bound}.b) with $p=4$ (and hence $\eta=18\theta$). We obtain:
\begin{align}
\nonumber
& \left\| D_{z,y}^2 u_{\theta}(t,x_3)- D_{z,y+y'}^2 u_{\theta}(t,x_3)- D_{z+z',y}^2 u_{\theta}(t,x)+ D_{z+z',y+y'}^2 u_{\theta}(t,x_3) \right\|_4 \\
\label{ineq-I2-QCLT-1}
& \quad \quad \quad \leq 4\theta \int_{0<s<r<t} \big(I_2(z,z',y,y',x_3;s,r,t) + I_2(y,y',z,z',x_3;s,r,t)\big) dsdr \\
\label{ineq-I2-QCLT-2}
& \left\| D_{w,y}^2 u_{\theta}(t,x_4)- D_{w,y+y'}^2 u_{\theta}(t,x_4)- D_{w+w',y}^2 u_{\theta}(t,x_4)+ D_{w+w',y+y'}^2 u_{\theta}(t,x_4) \right\|_4 \nonumber \\
&	\quad \quad \quad \leq  4\theta \int_{0<s<r<t} \big(I_2(w,w',y,y',x_4;s,r,t) + I_2(y,y',w,w',x_4;s,r,t) \big)dsdr .
\end{align}

%For the fourth norm, replacing $(z,z',x_3)$ by $(w,w',x_4)$, we %have

Hence, substituting the identities \eqref{ineq-I1-QCLT-1}, \eqref{ineq-I1-QCLT-2}, \eqref{ineq-I2-QCLT-1} and \eqref{ineq-I2-QCLT-2} to \eqref{ineq-cA}, we have
\begin{align} \label{eq-cA-bound}
	\cA \le& (4 \theta)^3 \int_{[-R,R]^4} dx_1dx_2dx_3dx_4 \int_{\bR^3} dydzdw \int_{\bR^3} |y'|^{2H-2} |z'|^{2H-2} |w'|^{2H-2} dy'dz'dw' \nonumber \\
	& \int_0^t dr_1 \int_0^t dr_2 \int_{0<s_3<r_3<t} ds_3dr_3 \int_{0<s_4<r_4<t} ds_4dr_4 I_1(z,z',x_1;r_1,t) I_1(w,w',x_2;r_2,t) \nonumber \\
	& \times \big( I_2(z,z',y,y',x_3;s_3,r_3,t) + I_2(y,y',z,z',x_3;s_3,r_3,t) \big) \nonumber \\
	& \times \big( I_2(w,w',y,y',x_4;s_4,r_4,t) + I_2(y,y',w,w',x_4;s_4,r_4,t) \big).
\end{align}

%Comparing $I_1, I_2$ with the terms on page 15 of your %manuscript with Jingyu, panqiu, xiong, $I_1(z,z',x_1;r_1,t)$ %coincide with $A_1+A_2$, $I_1(w,w',x_2;r_2,t)$ coincide with %$B_1+B_2$, $I_2(z,z',y,y',x_3;s_3,r_3,t)$ match the term %$C_1+C_2$ and $I_2(w,w',y,y',x_4;s_4,r_4,t)$ match the term %$D_1+D_2$.

\medskip

{\bf Step 2.} In this step, we evaluate the spatial integral in \eqref{eq-cA-bound}. We have
\begin{align} \label{eq-int-spatial}
	 & \int_{[-R,R]^4} dx_1dx_2dx_3dx_4 \int_{\bR^3} dydzdw \int_{\bR^3} |y'|^{2H-2} |z'|^{2H-2} |w'|^{2H-2} dy'dz'dw' \nonumber \\
	& I_1(z,z',x_1;r_1,t) I_1(w,w',x_2;r_2,t) \big( I_2(z,z',y,y',x_3;s_3,r_3,t) + I_2(y,y',z,z',x_3;s_3,r_3,t) \big) \nonumber \\
	& \times \big( I_2(w,w',y,y',x_4;s_4,r_4,t) + I_2(y,y',w,w',x_4;s_4,r_4,t) \big) \nonumber \\
	&\le \int_{[-R,R]} dx_4 \int_{\bR^3} |y'|^{2H-2} |z'|^{2H-2} |w'|^{2H-2} dy'dz'dw' \nonumber \\
	& \sup_{z \in \bR} \int_{\bR} I_1(z,z',x_1;r_1,t) dx_1 \sup_{w \in \bR} \int_{\bR} I_1(w,w',x_2;r_2,t) dx_2 \nonumber \\
	& \times \sup_{y \in \bR} \int_{\bR^2} \big( I_2(z,z',y,y',x_3;s_3,r_3,t) + I_2(y,y',z,z',x_3;s_3,r_3,t) \big) dz dx_3 \nonumber \\
	& \times \int_{\bR^2} \big( I_2(w,w',y,y',x_4;s_4,r_4,t) + I_2(y,y',w,w',x_4;s_4,r_4,t) \big) dwdy \nonumber \\
	&\le \int_{[-R,R]} dx_4
	\left( \int_{\bR^2} \left( \sup_{z \in \bR} \int_{\bR} I_1(z,z',x_1;r_1,t) dx_1 \sup_{w \in \bR} \int_{\bR} I_1(w,w',x_2;r_2,t) dx_2 \right)^2 |z'|^{2H-2} |w'|^{2H-2} dz'dw' \right)^{1/2} \nonumber \\
	& \left( \int_{\bR^2} \left( \sup_{y \in \bR} \int_{\bR^2} \big( I_2(z,z',y,y',x_3;s_3,r_3,t) + I_2(y,y',z,z',x_3;s_3,r_3,t) \big) dz dx_3 \right)^2 |y'|^{2H-2} |z'|^{2H-2} dy'dz' \right)^{1/2} \nonumber \\
	& \left( \int_{\bR^2} \left( \int_{\bR^2} \big( I_2(w,w',y,y',x_4;s_4,r_4,t) + I_2(y,y',w,w',x_4;s_4,r_4,t) \big) dwdy \right)^2 |y'|^{2H-2} |w'|^{2H-2} dy'dw' \right)^{1/2}.
\end{align}
where for the last inequality, we use Lemma \ref{Lem-Cauchy-Schwarz-3} for the integral $\iiint dy'dz'dw'$

Next, we need to compute the three integrals $dz'dw', dy'dz', dy'dw'$ one by one. We start with the integral $dz'dw'$. By Theorem \ref{Thm-exist-u}, Lemma \ref{sup-v-lem} and Lemma \ref{translation-lemma}, for any $r \in [0,t]$
\begin{align*}
	\int_{\bR} I_1(z,z',x_1;r,t) dx_1
	=& \int_{\bR} \left\| u_{12\theta}(r,z) - u_{12\theta}(r,z+z') \right\|_2 \left\| v_{12\theta}^{(r,z)}(t,x_1) \right\|_2 dx_1  \\
	& + \int_{\bR} \left\| u_{12\theta}(r,z+z') \right\|_2 \left\| v_{12\theta}^{(r,z)}(t,x_1) - v_{12\theta}^{(r,z+z')}(t,x_1) \right\|_2 dx_1 \\
	\le& 2tC_{12\theta,2,t,v} \left\| u_{12\theta}(r,0) - u_{12\theta}(r,z') \right\|_2 \\
	& + C_{12\theta,2,t,u} \int_{\bR} \left\| v_{12\theta}^{(r,z)}(t,x_1) - v_{12\theta}^{(r,z+z')}(t,x_1) \right\|_2 dx_1.
\end{align*}
By Remark \ref{trans-rem}, the integral on the right-hand side does not depend on $z$. Hence, by Lemma \ref{Lem-int-u-u} and Lemma \ref{Lem-int-v-v'}, we have
\begin{align*}
	& \int_{\bR} \left( \sup_{z \in \bR} \int_{\bR} I_1(z,z',x_1;r,t) dx_1 \right)^2 |z'|^{2H-2} dz' \\
	\le& 8t^2C_{12\theta,2,t,v}^2 \int_{\bR} \left\| u_{12\theta}(r,0) - u_{12\theta}(r,z') \right\|_2^2 |z'|^{2H-2} dz' \\
	&+ 2 C_{12\theta,2,t,u}^2 \int_{\bR} \left( \sup_{z \in \bR} \int_{\bR} \left\| v_{12\theta}^{(r,z)}(t,x_1) - v_{12\theta}^{(r,z+z')}(t,x_1) \right\|_2 dx_1 \right)^2 |z'|^{2H-2} dz' \\
	\le& 8t^2C_{12\theta,2,t,v}^2 C'_{u,2,H,t,12\theta} + 2 C_{12\theta,2,t,u}^2 C'_{v,2,H,t,12\theta}.
\end{align*}
Similarly, for any $r \in [0,t]$, we have
\begin{align*}
	\int_{\bR} \left( \sup_{w \in \bR} \int_{\bR} I_1(w,w',x_2;r,t) dx_2 \right)^2 |w'|^{2H-2} dw'
	\le 8t^2C_{12\theta,2,t,v}^2 C'_{u,2,H,t,12\theta} + 2 C_{12\theta,2,t,u}^2 C'_{v,2,H,t,12\theta}.
\end{align*}
Therefore, for any $r_1,r_2 \in [0,t]$, we have
\begin{align} \label{eq-dz'dw'}
	& \int_{\bR^2} \left( \sup_{z \in \bR} \int_{\bR} I_1(z,z',x_1;r_1,t) dx_1 \sup_{w \in \bR} \int_{\bR} I_1(w,w',x_2;r_2,t) dx_2 \right)^2 |z'|^{2H-2} |w'|^{2H-2} dz'dw' \nonumber \\
	& \le \left( 8t^2C_{12\theta,2,t,v}^2 C'_{u,2,H,t,12\theta} + 2 C_{12\theta,2,t,u}^2 C'_{v,2,H,t,12\theta} \right)^2.
\end{align}

Secondly, we deal with the integral $dy'dz'$, which consist of two terms. For the first term, by Theorem \ref{Thm-exist-u}, Lemma \ref{sup-v-lem} and Lemma \ref{translation-lemma}, for any $0\leq s<r\leq t$,
\begin{align} \label{eq-int-I2-1}
	& \int_{\bR^2} I_2(z,z',y,y',x_3;s,r,t) dzdx_3 \nonumber \\
	\le & 2t C_{18\theta,2,t,v} \left\| u_{18\theta}(s,0) - u_{18\theta}(s,z') \right\|_2 \int_{\bR} \left\| v_{18\theta}^{(r,y)} (t,x_3) - v_{18\theta}^{(r,y+y')} (t,x_3) \right\|_2 dx_3 \nonumber \\
	& + C_{18\theta,2,t,u} \int_{\bR} \left\| v_{18\theta}^{(s,z)} (r,y) - v_{18\theta}^{(s,z+z')} (r,y) \right\|_2 dz \int_{\bR} \left\| v_{18\theta}^{(r,y)} (t,x_3) - v_{18\theta}^{(r,y+y')} (t,x_3) \right\|_2 dx_3 \nonumber \\
	& + 2t C_{18\theta,2,t,v} \left\| u_{18\theta}(s,0) - u_{18\theta}(s,z') \right\|_2 \int_{\bR} \left\| v_{18\theta}^{(s,z)} (r,y) - v_{18\theta}^{(s,z)} (r,y+y') \right\|_2 dz \nonumber \\
	&+ 2t C_{18\theta,2,t,v} C_{18\theta,2,t,u} \int_{\bR} \left\| v_{18\theta}^{(s,z)} (r,y) - v_{18\theta}^{(s,z)} (r,y+y') - v_{18\theta}^{(s,z+z')} (r,y) + v_{18\theta}^{(s,z+z')} (r,y+y') \right\|_2 dz.
\end{align}

For the second term, we integrate $dx_3$ first and then $dz$. We have
\begin{align} \label{eq-int-I2-2}
	& \int_{\bR^2} I_2(y,y',z,z',x_3;s,r,t) dzdx_3 \nonumber \\
	\le & 2t C_{18\theta,2,t,v} \left\| u_{18\theta}(s,0) - u_{18\theta}(s,y') \right\|_2 \sup_{z \in \bR} \int_{\bR} \left\| v_{18\theta}^{(r,z)} (t,x_3) - v_{18\theta}^{(r,z+z')} (t,x_3) \right\|_2 dx_3 \nonumber \\
	& + C_{18\theta,2,t,u} \int_{\bR} \left\| v_{18\theta}^{(s,y)} (r,z) - v_{18\theta}^{(s,y+y')} (r,z) \right\|_2 dz \sup_{z \in \bR} \int_{\bR} \left\| v_{18\theta}^{(r,z)} (t,x_3) - v_{18\theta}^{(r,z+z')} (t,x_3) \right\|_2 dx_3 \nonumber \\
	& + 2t C_{18\theta,2,t,v} \left\| u_{18\theta}(s,0) - u_{18\theta}(s,y') \right\|_2 \int_{\bR} \left\| v_{18\theta}^{(s,y)} (r,z) - v_{18\theta}^{(s,y)} (r,z+z') \right\|_2 dz \nonumber \\
	&+ 2t C_{18\theta,2,t,v} C_{18\theta,2,t,u} \int_{\bR} \left\| v_{18\theta}^{(s,y)} (r,z) - v_{18\theta}^{(s,y)} (r,z+z') - v_{18\theta}^{(s,y+y')} (r,z) + v_{18\theta}^{(s,y+y')} (r,z+z') \right\|_2 dz.
\end{align}
One can see from the Remark \ref{trans-rem} that the integrals on the right-hand side of \eqref{eq-int-I2-1} and \eqref{eq-int-I2-2} do not depend on $y$. Hence, by Lemma \ref{Lem-int-u-u} and Lemma \ref{Lem-int-v-v'} together with the inequality $(\sum_{i=1}^4 a_i)^2 \le 4 \sum_{i=1}^4 a_i^2$, we have
\begin{align*}
	& \max \bigg\{ \int_{\bR^2} \left( \sup_{y \in \bR} \int_{\bR^2} I_2(z,z',y,y',x_3;s_3,r_3,t) dzdx_3 \right)^2 |y'|^{2H-2} |z'|^{2H-2} dy'dz', \\
	& \int_{\bR^2} \left( \sup_{y \in \bR} \int_{\bR^2} I_2(y,y',z,z',x_3;s_3,r_3,t) dzdx_3 \right)^2 |y'|^{2H-2} |z'|^{2H-2} dy'dz' \bigg\} \\
	\le& 8 \left( 2t C_{18\theta,2,t,v} \right)^2 C'_{u,2,H,t,18\theta} C'_{v,2,H,t,18\theta} + 4 C_{18\theta,2,t,u}^2 (C'_{v,2,H,t,18\theta})^2 + 4 \left( 2t C_{18\theta,2,t,v} C_{18\theta,2,t,u} \right)^2 C'_{v,2,H,t,18\theta}.
\end{align*}
Therefore, using the inequality $(a_1+a_2)^2 \le 2a_1^2 + 2a_2^2$, we have
\begin{align} \label{eq-dy'dz'}
	& \int_{\bR^2} \left( \sup_{y \in \bR} \int_{\bR^2} \big( I_2(z,z',y,y',x_3;s_3,r_3,t) + I_2(y,y',z,z',x_3;s_3,r_3,t) \big) dz dx_3 \right)^2 |y'|^{2H-2} |z'|^{2H-2} dy'dz' \nonumber \\
	& \le 32 \left( 2t C_{18\theta,2,t,v} \right)^2 C'_{u,2,H,t,18\theta} C'_{v,2,H,t,18\theta} + 16 C_{18\theta,2,t,u}^2 (C'_{v,2,H,t,18\theta})^2 + 16 \left( 2t C_{18\theta,2,t,v} C_{18\theta,2,t,u} \right)^2 C'_{v,2,H,t,18\theta}.
\end{align}

Thirdly, we deal with the integral $dy'dw'$. By change of variables $(w,w') \leftrightarrow (y,y')$, one has
\begin{align*}
	& \int_{\bR^2} \left( \int_{\bR^2} I_2(w,w',y,y',x_4;s_4,r_4,t) dwdy \right)^2 |y'|^{2H-2} |w'|^{2H-2} dy'dw' \nonumber \\
	= & \int_{\bR^2} \left( \int_{\bR^2} I_2(y,y',w,w',x_4;s_4,r_4,t) dwdy \right)^2 |y'|^{2H-2} |w'|^{2H-2} dy'dw'.
\end{align*}
Thus, it is enough to compute the first integral. Let $0\leq s<r\leq t$ be arbitrary. By Theorem \ref{Thm-exist-u}, Lemma \ref{sup-v-lem} and Lemma \ref{translation-lemma}, we integrate $dw$ first and then $dy$ to obtain
\begin{align*}
	& \int_{\bR^2} I_2(w,w',y,y',x_4;s,r,t) dwdy \nonumber \\
	&\le 2t C_{18\theta,2,t,v} \left\| u_{18\theta}(s,0) - u_{18\theta}(s,w') \right\|_2 \int_{\bR} \left\| v_{18\theta}^{(r,y)} (t,x_4) - v_{18\theta}^{(r,y+y')} (t,x_4) \right\|_2 dy \nonumber \\
	& + C_{18\theta,2,t,u} \sup_{y \in \bR} \int_{\bR} \left\| v_{18\theta}^{(s,w)} (r,y) - v_{18\theta}^{(s,w+w')} (r,y) \right\|_2 dw \int_{\bR} \left\| v_{18\theta}^{(r,y)} (t,x_4) - v_{18\theta}^{(r,y+y')} (t,x_4) \right\|_2 dy \nonumber \\
	& + 2t C_{18\theta,2,t,v} \left\| u_{18\theta}(s,0) - u_{18\theta}(s,w') \right\|_2 \sup_{y \in \bR} \int_{\bR} \left\| v_{18\theta}^{(s,w)} (r,y) - v_{18\theta}^{(s,w)} (r,y+y') \right\|_2 dw \nonumber \\
	&+ 2t C_{18\theta,2,t,v} C_{18\theta,2,t,u} \sup_{y \in \bR} \int_{\bR} \left\| v_{18\theta}^{(s,w)} (r,y) - v_{18\theta}^{(s,w)} (r,y+y') - v_{18\theta}^{(s,w+w')} (r,y) + v_{18\theta}^{(s,w+w')} (r,y+y') \right\|_2 dw.
\end{align*}
Thus, using Lemma \ref{Lem-int-u-u} and Lemma \ref{Lem-int-v-v'} with the inequality $(\sum_{i=1}^4 a_i)^2 \le 4 \sum_{i=1}^4 a_i^2$, we have
\begin{align} \label{eq-int-I2-3}
	& \int_{\bR^2} \left( \int_{\bR^2} I_2(w,w',y,y',x_4;s_4,r_4,t) dwdy \right)^2 |y'|^{2H-2} |w'|^{2H-2} dy'dw' \nonumber \\
	\le& 8 \left( 2t C_{18\theta,2,t,v} \right)^2 C'_{u,2,H,t,18\theta} C'_{v,2,H,t,18\theta} + 4 C_{18\theta,2,t,u}^2 (C'_{v,2,H,t,18\theta})^2 + 4 \left( 2t C_{18\theta,2,t,v} C_{18\theta,2,t,u} \right)^2 C'_{v,2,H,t,18\theta}.
\end{align}
Therefore, using the inequality $(a_1+a_2)^2 \le 2a_1^2 + 2a_2^2$, we have
\begin{align} \label{eq-dy'dw'}
	& \int_{\bR^2} \left( \int_{\bR^2} \big( I_2(w,w',y,y',x_4;s_4,r_4,t) + I_2(y,y',w,w',x_4;s_4,r_4,t) \big) dwdy \right)^2 |y'|^{2H-2} |w'|^{2H-2} dy'dw' \nonumber \\
	& \le 32 \left( 2t C_{18\theta,2,t,v} \right)^2 C'_{u,2,H,t,18\theta} C'_{v,2,H,t,18\theta} + 16 C_{18\theta,2,t,u}^2 (C'_{v,2,H,t,18\theta})^2 + 16 \left( 2t C_{18\theta,2,t,v} C_{18\theta,2,t,u} \right)^2 C'_{v,2,H,t,18\theta}.
\end{align}

Lastly, substituting \eqref{eq-dz'dw'}, \eqref{eq-dy'dz'} and \eqref{eq-dy'dw'} to \eqref{eq-int-spatial}, we have the following estimate for spatial integral:
\begin{align*}
	& \int_{[-R,R]^4} dx_1dx_2dx_3dx_4 \int_{\bR^3} dydzdw \int_{\bR^3} |y'|^{2H-2} |z'|^{2H-2} |w'|^{2H-2} dy'dz'dw' \nonumber \\
	& I_1(z,z',x_1;r_1,t) I_1(w,w',x_2;r_2,t) \big( I_2(z,z',y,y',x_3;s_3,r_3,t) + I_2(y,y',z,z',x_3;s_3,r_3,t) \big) \nonumber \\
	& \times \big( I_2(w,w',y,y',x_4;s_4,r_4,t) + I_2(y,y',w,w',x_4;s_4,r_4,t) \big) \nonumber \\
	&\le \int_{[-R,R]} dx_4 \left( 8t^2C_{12\theta,2,t,v}^2 C'_{u,2,H,t,12\theta} + 2 C_{12\theta,2,t,u}^2 C'_{v,2,H,t,12\theta} \right) \nonumber \\
	& \times \left( 32 \left( 2t C_{18\theta,2,t,v} \right)^2 C'_{u,2,H,t,18\theta} C'_{v,2,H,t,18\theta} + 16 C_{18\theta,2,t,u}^2 (C'_{v,2,H,t,18\theta})^2 + 16 \left( 2t C_{18\theta,2,t,v} C_{18\theta,2,t,u} \right)^2 C'_{v,2,H,t,18\theta} \right).
\end{align*}
Therefore, coming back to \eqref{eq-cA-bound}, we obtain $\cA \le C(\theta,t,H) R$, where $C(\theta,t,H)$ is a positive constant only depends on $\theta,t,H$ and is increasing in $t$.
This concludes the proof.

\medskip

In the argument above, we used the following generalized version of the Cauchy-Schwarz inequality.

\begin{lemma}
\label{Lem-Cauchy-Schwarz-3}
Let $n_1,n_2,n_3 \in \bN$, and let $f \in L^2(\bR^{n_1+n_2}), g \in L^2(\bR^{n_2+n_3}), h \in L^2(\bR^{n_3+n_1})$. Then the following inequality holds.
\begin{align*}
	& \int f(x_1,x_2) g(x_2,x_3) h(x_3,x_1) dx_1dx_2dx_3 \\
	\le& \left( \int f^2(x_1,x_2) dx_1dx_2 \right)^{1/2} \left( \int g^2(x_2,x_3) dx_2dx_3 \right)^{1/2} \left( \int h^2(x_3,x_1) dx_3dx_1 \right)^{1/2}.
\end{align*}
\end{lemma}

\begin{proof}
Applying Cauchy-Schwarz inequality for the integral $dx_1dx_2$ first, and then Cauchy-Schwarz inequality to $dx_3$, we have
\begin{align*}
	& \int f(x_1,x_2) g(x_2,x_3) h(x_3,x_1) dx_1dx_2dx_3 \\
	\le& \left( \int f^2(x_1,x_2) dx_1dx_2 \right)^{1/2} \left( \int dx_1dx_2 \left( \int g(x_2,x_3) h(x_3,x_1) dx_3 \right)^2 \right)^{1/2} \\
	\le& \left( \int f^2(x_1,x_2) dx_1dx_2 \right)^{1/2} \left( \int dx_1dx_2 \left( \int g^2(x_2,x_3) dx_3 \right) \left( \int h^2(x_3,x_1) dx_3 \right) \right)^{1/2} \\
	=& \left( \int f^2(x_1,x_2) dx_1dx_2 \right)^{1/2} \left( \int g^2(x_2,x_3) dx_2dx_3 \right)^{1/2} \left( \int h^2(x_3,x_1) dx_3dx_1 \right)^{1/2}.
\end{align*}
\end{proof}

\subsection{Functional CLT}

In this section, we give the proof of Theorem \ref{FCLT}.

\medskip
%\subsection{Tightness}

{\bf Step 1. (tightness)} For this, we show that for any $p\geq 2$, $0\leq s<t\leq T$ and $R\geq 1$,
\begin{equation}
\label{tight-eq}
\|F_R(t)-F_R(s)\|_p \leq C R^{1/2}(t-s),
\end{equation}
where $C=C_{p,\theta,T}>0$ is a constant that depends on $(p,\theta,T)$. By Kolmogorov's continuity theorem, it will follow that the process $F_R=\{F_R(t)\}_{t\in [0,T]}$ has a continuous modification (which we denote also by $F_R$). Moreover, by Kolmogorov-Censtov Theorem (Theorem 12.3 of \cite{billingsley68}), the family $\{F_R\}_{R>1}$ is tight in $C([0,T])$.

To prove \eqref{tight-eq}, we proceed as in Section 4.3 of \cite{BY}.
Using the chaos expansion, we have: $F_R(t)-F_R(s)= \sum_{n\geq 1} \theta^{n/2} I_n (g_{n,R}(\cdot;t,s))$, where
\begin{align*}
	& g_{n,R}(x_1,\ldots,x_n;t,s)
	= \int_{B_R} \big( f_n(x_1,\ldots,x_n,x;t) - f_n(x_1,\ldots,x_n,x;s) \big)dx.
\end{align*}
%	& \quad =\int_{B_R} \int_{T_n(t)} \prod_{j=1}^{n-1} %G_{t_{j+1}-t_j} (x_{j+1} - x_j) \big( G_{t-t_n}(x-x_n) - %G_{s-t_n}(x-x_n) \big)  d\pmb{t} dx

By hypercontractivity, for any $p\geq 2$,
\begin{equation}
\label{increm-F}
\|F_R(t)-F_R(s)\|_p \leq \sum_{n\geq 1}\theta^{n/2}(p-1)^{n/2}(n!)^{1/2}\| \widetilde{g}_{n,R}(\cdot;t,s)\|_{\cP_0^{\otimes n}}.
\end{equation}

Using the same argument as for (54)-(55) of \cite{BY}, followed by the change of variables $\eta_j=\xi_1+\ldots+\xi_j$ and inequality \eqref{prod}, we obtain:
\begin{align*}
& n! \| \widetilde{g}_{n,R}(\cdot;t,s)\|_{\cP_0^{\otimes n}}^2
\leq (t-s)^2 t^n c_H^n  \\
& \quad \times \int_{T_n(t)} \int_{\bR^n}  |\cF \mathbf{1}_{[-R,R]} (\xi_1 + \ldots + \xi_n)|^2 \left| \prod_{j=1}^{n-1} \cF G_{t_{j+1}-t_j} (\xi_1 + \ldots + \xi_j) \right|^2   \prod_{j=1}^n |\xi_j|^{1-2H} d\pmb{\xi} d\pmb{t} \\
& \quad \leq c_H^n (t-s)^2 t^n \sum_{\alpha\in D_n} \int_{T_n(t)} \prod_{j=1}^{n-1} \left( \int_{\bR} \left| \cF G_{t_{j+1}-t_j} (\eta_j) \right|^2 |\eta_j|^{\alpha_j} d\eta_j\right) \left(\int_{\bR} |\cF \mathbf{1}_{[-R,R]} (\eta_n)|^2 |\eta_n|^{\alpha_n} d\eta_n \right)d\pmb{t} .
\end{align*}
Note that
$\cF 1_{[-R,R]}(\eta)=\frac{2\sin(\eta R)}{\eta}$. Using Lemma \ref{Lem-ineq-integral-1} and the fact that $\alpha_n \in \{0,1-2H\}$, we obtain:
\begin{align*}
\int_{\bR} |\cF \mathbf{1}_{[-R,R]} (\eta_n)|^2 |\eta_n|^{\alpha_n} d\eta_n& =\int_{\bR}\left|\frac{2\sin(\eta_n R)}{\eta_n}\right|^2 |\eta|^{\alpha_n}d\eta=4R^{1-\alpha_n}\int_{\bR}\frac{\sin^2 \eta_n}{|\eta_n|^2} |\eta_n|^{\alpha_n} \\
& \leq 8 R^{1-\alpha_n}\left( \frac{1}{1-\alpha_n}+\frac{1}{1+\alpha_n}\right)\leq 8 R\left(\frac{1}{H}+1\right).
\end{align*}
Similarly, for any $j=1,\ldots, n-1$, using Lemma \ref{Lem-ineq-integral-1}, we have:
\begin{align*}
	\int_{\bR} \left| \cF G_{t_{j+1}-t_j} (\eta_j) \right|^2 |\eta_j|^{\alpha_j} d\eta_j
	&= \int_{\bR} \dfrac{\sin^2\left( (t_{j+1}-t_j) |\eta_j| \right)}{|\eta_j|^2} |\eta_j|^{\alpha_j} d\eta_j
	\le \dfrac{2}{1-\alpha_j} + \dfrac{2(t_{j+1}-t_j)^2}{1+\alpha_j} \\
&\leq 2\left(\frac{1}{4H-1}+1 \right)(1+t^2),
\end{align*}
since $\alpha_j \in \{0,1-2H, 2(1-2H)\}$ and $H \in (\frac{1}{4},\frac{1}{2})$. Hence,
\begin{equation}
\label{bound-g}
	n! \| \widetilde{g}_{n,R}(\cdot;t,s)\|_{\cP_0^{\otimes n}}^2
	\le  (t-s)^2 t^{2n} C_{H}^n (1+t^2)^{n-1} \dfrac{1}{n!}R,
\end{equation}
where $C_{H}>0$ is a constant that depends on $H$.
Relation \eqref{tight-eq} follows from \eqref{increm-F}, \eqref{bound-g} and Stirling's formula.

\medskip
{\bf Step 2. (finite-dimensional convergence)}
%\subsection{Finite-dimensional convergence}
Let $Q_R(t)=R^{-1/2}F_R(t)$. We have to show that for any $m\geq 1$, $0\leq t_1<\ldots<t_m \leq T$,
\[
\big(Q_R(t_1),\ldots, Q_R(t_m)\big) \stackrel{d}{\to}
\big(\cG_1(t_1),\ldots,\cG_m(t_m)\big)
\]
Using the same argument as in the proof of Theorem 1.3.(iii) (Step 2) of \cite{BY}, it is enough to prove that for any $i,j=1,\ldots,m$,
\[
{\rm Var}\Big(\langle DF_R(t_i),-DL^{-1} F_R(t_j)\rangle_{\cP_0} \Big)\leq CR.
\]

To estimate this variance, we use Proposition B.1 of \cite{BY2022}. It remains to show that
% for any $0\leq s<t \leq T$,
\begin{align*}
	%\cA(t,s)=&
& \int_{[-R,R]^4} \int_{\bR^6} \left\| D_z u_{\theta}(t_j,x_1)- D_{z+z'} u_{\theta}(t_j,x_1) \right\|_4 \left\| D_w u_{\theta}(t_j,x_2) - D_{w+w'} u_{\theta}(t_j,x_2) \right\|_4 \nonumber \\
	& \left\| D_{z,y}^2 u_{\theta}(t_i,x_3)- D_{z,y+y'}^2 u_{\theta}(t_i,x_3)- D_{z+z',y}^2 u_{\theta}(t_i,x_3)+ D_{z+z',y+y'}^2 u_{\theta}(t_j,x_3) \right\|_4 \nonumber \\
	& \left\| D_{w,y}^2 u_{\theta}(t_i,x_4)- D_{w,y+y'}^2 u_{\theta}(t_i,x_4)- D_{w+w',y}^2 u_{\theta}(t_i,x_4)+ D_{w+w',y+y'}^2 u_{\theta}(t_i,x_4) \right\|_4 \nonumber \\
	& \qquad \qquad \qquad |y'|^{2H-2}|z'|^{2H-2}|w'|^{2H-2} dydy' dzdz' dwdw' dx_1 dx_2 dx_3 dx_4 \leq CR.
\end{align*}
This can be proved using the same argument as for $\cA^* \leq CR $ in Section \ref{sec:QCLT}. We omit the details.

\appendix

\section{Moment comparison using hypercontractivity}
\label{section-hyper}

In this section, we provide a moment comparion result for the solution of a general SPDE, in the spirit of Theorem B.1 of \cite{BCC}.

\medskip

Let $W=\{W(\varphi);\varphi \in \cH\}$ be an isonormal Gaussian process,
associated to a Hilbert space $\cH$. Assume that either one of the following conditions hold:\\
(i) $\cH$ consists  of functions (or distributions) on $\bR_{+} \times \bR^d$,
i.e. $W$ is {\em time-dependent}; and\\
(ii)$\cH$ consists of functions (or distributions) on $\bR^d$, i.e. $W$ is {\em time-independent}.\\

Let $\cal L$ be a second-order pseudo-differential operator of $\bR_{+} \times \bR^d$ and $u_{\theta}$ be the solution of the SPDE:
\begin{equation}
\label{eq-SDE-hyper}
	{\cal L}u(t,x)=\sqrt{\theta}u(t,x) \dot{W}, \quad t>0,x\in \bR^d
\end{equation}
with (deterministic) initial condition.
By definition, the {\em (mild Skorohod) solution} to \eqref{eq-SDE-hyper} is an adapted square-integrable process $u_{\theta}=\{u_{\theta}(t,x);t>0,x\in \bR^d\}$ which satisfies
\begin{align*}
	u_{\theta}(t,x)=w(t,x)+\sqrt{\theta}\int_0^t \int_{\bR^d}G(t-s,x-y)u_{\theta}(s,y)W(\delta s,\delta y),
\end{align*}
if the noise $W$ is time-dependent, respectively
\begin{align*}
	u_{\theta}(t,x)=w(t,x)+\sqrt{\theta}\int_0^t \int_{\bR^d}G(t-s,x-y)u_{\theta}(s,y)W(\delta y)ds,
\end{align*}
if the noise $W$ is time-independent, provided that these integrals are well-defined.
Here $W(\delta s,\delta y)$ (respectively $W(\delta y)$) denotes the Skorohod integral with respect to $W$,
$G$ is the fundamental solution of $\cal L$ on $\bR_{+} \times \bR^d$,
and $w$ is the solution of the deterministic equation ${\cal L}u=0$ on
$\bR_{+} \times \bR^d$, with the same initial condition as \eqref{eq-SDE-hyper}.

\begin{lemma}
\label{Lemma-hypercontractivity}
Suppose that for any $\theta>0$, equation \eqref{eq-SDE-hyper} has a unique solution $U_{\theta}$ and $\bE|U_{\theta}(t,x)|^p<\infty$ for any $t>0$, $x\in \bR^d$ and $p>1$.

(a) For any $\theta>0$, $q> p>1$, $t>0$, $x \in \bR^d$ and $h \in \bR^d$, we have:
	\begin{align*}
		\left\| U_{\frac{p-1}{q-1}\theta}(t,x+h) - U_{\frac{p-1}{q-1}\theta}(t,x) \right\|_q \le \left\| U_{\theta}(t,x+h) - U_{\theta}(t,x) \right\|_p.
	\end{align*}
In particular, for any $\theta>0$, $p>2$, $t>0$, $x \in \bR^d$ and $h \in \bR^d$, we have:
\begin{align*}
		\left\| U_{\theta}(t,x+h) - U_{\theta}(t,x) \right\|_p \le \left\| U_{(p-1)\theta}(t,x+h) - U_{(p-1)\theta}(t,x) \right\|_2.
	\end{align*}

(b) Assume that for any $\theta>0$, $t>0$ and $x \in \bR^d$, $U_{\theta}(t,x)$ is Malliavin differentiable of order $k$, its Malliavin derivative is a function in $\cH$, and $\bE|D_z U_{\theta}(t,x)|^p<\infty $ for any $z$ and $p>1$. Then, for any $q> p>1$, $t>0$, $x \in \bR^d$, $l\in \bN$, for any real numbers $\{a^{(j)}: 1 \le j \le l\}$ and for any $\{z_i^{(j)}:1 \le i \le k, 1 \le j \le l \}$ (chosen in $\bR_{+}\times \bR^d$ if the noise is time-dependent, or in $\bR^d$ if the noise is time-independent), we have:
	\begin{align*}
		\left\| \sum_{j=1}^l a^{(j)} D^k_{z_1^{(j)},\ldots,z_k^{(j)}} U_{\frac{p-1}{q-1}\theta}(t,x) \right\|_q
		\le \left( \dfrac{p-1}{q-1} \right)^{k/2} \left\| \sum_{j=1}^l a^{(j)} D^k_{z_1^{(j)},\ldots,z_k^{(j)}} U_{\theta}(t,x) \right\|_p.
	\end{align*}
In particular, for any $\theta>0$, $p>2$, $t>0$ and $x \in \bR^d$, we have:
\begin{align*}
		\left\| \sum_{j=1}^l a^{(j)} D^k_{z_1^{(j)},\ldots,z_k^{(j)}} U_{\theta}(t,x) \right\|_p
		\le \left( \dfrac{1}{p-1} \right)^{k/2} \left\| \sum_{j=1}^l a^{(j)} D^k_{z_1^{(j)},\ldots,z_k^{(j)}} U_{(p-1)\theta}(t,x) \right\|_2.
	\end{align*}
\end{lemma}

\begin{proof}
From the proof of Theorem B.1 of \cite{BCC}, for any $\tau>0$, $t>0$ and $x \in \bR^d$, we have:
\begin{align}
\label{eq-T}
	T_{\tau} (U_{\theta}(t,x)) = U_{e^{-2\tau}\theta}(t,x) \quad \mbox{a.s.}
\end{align}
where $(T_{t})_{t \ge 0}$ is the Ornstein-Uhlenbeck semigroup (see relation (B.6) of \cite{BCC}).

(a) Using \eqref{eq-T} and the fact that $T_{\tau}$ is a linear operator, we have:
\begin{align*}
	T_{\tau} (U_{\theta}(t,x+h) - U_{\theta}(t,x)) = U_{e^{-2\tau}\theta}(t,x+h) - U_{e^{-2\tau}\theta}(t,x).
\end{align*}
Hence, using the hypercontractivity property \eqref{OU-hyper} of the OU semigroup, we obtain:
\begin{align*}
	\left\| U_{e^{-2\tau}\theta}(t,x+h) - U_{e^{-2\tau}\theta}(t,x) \right\|_{q(\tau)}
	=& \left\| T_{\tau} (U_{\theta}(t,x+h) - U_{\theta}(t,x)) \right\|_{q(\tau)} \\
	\le& \left\| U_{\theta}(t,x+h) - U_{\theta}(t,x) \right\|_p,
\end{align*}
where $q(\tau) = e^{2\tau} (p-1)+1$. We choose $\tau$ such that $q(\tau) = q$, which means that $e^{2\tau} = \frac{q-1}{p-1}$.

(b) By \eqref{eq-T} and \ref{Prop-OU},
\begin{align*}
	D^k_{z_1^{(j)},\ldots,z_k^{(j)}} U_{e^{-2\tau}\theta}(t,x)
	= D^k_{z_1^{(j)},\ldots,z_k^{(j)}} \left( T_{\tau} (U_{\theta}(t,x)) \right)
	= e^{-k\tau} T_{\tau} (D^k_{z_1^{(j)},\ldots,z_k^{(j)}} U_{\theta}(t,x)).
\end{align*}
Since $T_{\tau}$ is a linear operator,
\begin{align*}
	\sum_{j=1}^l a^{(j)} D^k_{z_1^{(j)},\ldots,z_k^{(j)}} U_{e^{-2\tau}\theta}(t,x)
	= e^{-k\tau} T_{\tau} \left( \sum_{j=1}^l a^{(j)} D^k_{z_1^{(j)},\ldots,z_k^{(j)}} U_{\theta}(t,x) \right).
\end{align*}
Using the hypercontractivity property \eqref{OU-hyper} of the OU semigroup, we obtain:
\begin{align*}
	\left\| \sum_{j=1}^l a^{(j)} D^k_{z_1^{(j)},\ldots,z_k^{(j)}} U_{e^{-2\tau}\theta}(t,x) \right\|_{q(\tau)}
	=& e^{-k\tau} \left\| T_{\tau} \left( \sum_{j=1}^l a^{(j)} D^k_{z_1^{(j)},\ldots,z_k^{(j)}} U_{\theta}(t,x) \right) \right\|_{q(\tau)} \\
	\le& e^{-k\tau} \left\| \sum_{j=1}^l a^{(j)} D^k_{z_1^{(j)},\ldots,z_k^{(j)}} U_{\theta}(t,x) \right\|_p,
\end{align*}
where $q(\tau) = e^{2\tau} (p-1)+1$. We choose $\tau$ such that $q(\tau) = q$, which means that $e^{2\tau} = \frac{q-1}{p-1}$.
\end{proof}

\begin{remark}
{\rm Using the same argument as above, one can prove the following extension of Lemma \ref{Lemma-hypercontractivity}. Consider a family $\mathcal L^{(1)}, \ldots, \mathcal L^{(l)}$ of operators, and let $U_{\theta}^{(j)}$ be the unique solution of equation \eqref{eq-SDE-hyper} with operator $\mathcal L$ replaced by $\mathcal L^{(j)}$, and some initial data $w^{(j)}$. Then
\begin{align*}
		\left\|  \sum_{j=1}^l a_j U^{(j)}_{\frac{p-1}{q-1}\theta_j}(t,x_j) \right\|_q \le \left\| \sum_{j=1}^l a_j U^{(j)}_{\theta_j}(t,x_j)) \right\|_p,
	\end{align*}
\begin{align*}
		\left\| \sum_{j=1}^l a^{(j)} D^k_{z_1^{(j)},\ldots,z_k^{(j)}} U^{(j)}_{\frac{p-1}{q-1}\theta}(t,x_j) \right\|_q
		\le \left( \dfrac{p-1}{q-1} \right)^{k/2} \left\| \sum_{j=1}^l a^{(j)} D^k_{z_1^{(j)},\ldots,z_k^{(j)}} U^{(j)}_{\theta}(t,x_j) \right\|_p.
	\end{align*}
}
\end{remark}

\section{Stochastic Volterra equations with white noise in time}
\label{section-appA}

In this section, we study two models involving parametric families of stochastic Volterra equations driven by the Gaussian noise $\fX$ with covariance \eqref{cov-fX}, with $H \in (\frac{1}{4},\frac{1}{2})$. This study will allow us to develop some properties of the solution $V_{\theta}$ of the (hAm) model \eqref{eq-V} with noise $\fX$ and delta initial velocity,
which are needed in the sequel.

\medskip

The stochastic heat and wave equations with noise $\fX$ and affine function $\sigma(u)=au+b$ multiplying the noise
were studied in \cite{BJQ15}. Several facts from \cite{BJQ15} will be needed here. For instance, from Theorem 2.9 of \cite{BJQ15} and Minkowski's inequality, we have: for any $p\geq 2$,
\begin{equation}
\label{BDG}
\left\|\int_0^T \int_{\bR}S(t,x)\fX(dt,dx)\right\|_p^2 \leq C_{p,H} \int_0^T \int_{\bR^2} \|S(t,x)-S(t,y)\|_p^2 |x-y|^{2H-2}dxdydt,
\end{equation}
where $C_{p,H}>0$ is a constant which depends on $p$ and $H$.

Let $T>0$ be arbitrary. For each $t\in [0,T]$, let $\G_t$ be a deterministic non-negative function on $\bR$. Consider the following functions:
\begin{align*}
J_1(t)&:= \int_{\bR^2}|\G_t(x)-\G_t(x+h)|^2 |h|^{2H-2}dxdh\\
J_2(t)&:= \int_0^t J_3(s)J_4(t-s)ds \\
J_3(t)&:=\int_{\bR^3}|\big(\G_{t}(x+h)-\G_t(x)\big)-
\big(\G_{t}(x+h+k)-\G_t(x+k)\big)|^2 |h|^{2H-2}
|k|^{2H-2}dhdkdx \\
J_4(t)&:=\int_{\bR}\G_t^2(x)dx.
\end{align*}

We impose the following assumption:
\medskip

{\bf Assumption A.} $J_i(t)<\infty$ for any $t\in [0,T]$ and  $i=1,2,3,4$, and
\[
Q_i(T):=\int_0^T J_i(t)dt<\infty \quad \mbox{for $i=1,2$}.
\]

\medskip

Assumption A holds when $\G_t(x)=G_t(x)$ is the fundamental solution of the wave equation (given by \eqref{def-G}), or when $\G_t(x)=g_t(x):=(2\pi t)^{-1/2} \exp(-\frac{x^2}{2t})$ is the fundamental solution of the heat equation: by Lemmas 3.1 and 3.3 of \cite{BJQ15},
\begin{align*}
J_1(t)&=C \int_{\bR} |\cF \G_t(\xi)|^2 |\xi|^{1-2H}d\xi=
\begin{cases}
		C t^{2H} \quad \mbox{for wave equation} \\
		C t^{H-1} \quad \mbox{for heat equation}
	\end{cases}
\\
J_3(t)&=C \int_{\bR} |\cF \G_t(\xi)|^2 |\xi|^{2(1-2H)}d\xi=
\begin{cases}
		C t^{4H-1} \quad \mbox{for wave equation} \\
		C t^{2H-3/2} \quad \mbox{for heat equation}
	\end{cases}
\end{align*}
\[
J_4(t)=
\begin{cases}
		C t \quad \mbox{for wave equation} \\
		C t^{-1/2} \quad \mbox{for heat equation}
	\end{cases}
\quad \mbox{and} \quad
J_2(t)=
\begin{cases}
		C t^{4H+1} \quad \mbox{for wave equation} \\
		C t^{2H-1} \quad \mbox{for heat equation}
	\end{cases}
\]
Here $C$ is a constant that depends on $H$ and may be different in each of its appearances.

\medskip

Let $p\geq 2$ be arbitrary. For any $r \in [0,T]$, let $\cX_r$ be the space of predictable processes $\{X(t,x);t \in [r,T],x\in \bR\}$ such that $\|X\|_{\cX_r}:=\|X\|_{\cX_{r,1}}+\|X\|_{\cX_{r,2}}<\infty$, where
\begin{align*}
\|X\|_{\cX_{r,1}}&=\sup_{(t,x)\in [r,T] \times \bR}\|X(t,x)\|_p\\
\|X\|_{\cX_{r,2}}&=\sup_{(t,x)\in [r,T] \times \bR} \left(
\int_r^t \int_{\bR^d}\G_{t-s}^2(x-y)\|X(s,y)-X(s,y+h)\|_{p}^2 |h|^{2H-2} dh dy ds \right)^{1/2}.
\end{align*}
%By Lemma A.2 of \cite{HHLNT1}, $\cX_r$ is a Banach space.
\begin{theorem}
\label{gen-th}
Suppose Assumption A holds.
Let $\Lambda$ be an arbitrary set.

a) For any $r \in [0,T]$ and $\alpha \in \Lambda$, the family of stochastic Volterra equations:
\begin{equation}
\label{gen-eq}
X(t,x)=X_0^{(r,\alpha)}(t,x)+\int_r^t \int_{\bR}\G_{t-s}(x-y)X(s,y)\fX(ds,dy),
\end{equation}
with $t \in [r,T]$ and $x\in \bR$, has a unique solution $X^{(r,\alpha)}$ in $\cX_r$, provided that $X_0^{(r,\alpha)}\in \cX_r$.

b) If in addition,
\begin{equation}
\label{cond-X0}
\sup_{r \in [0,T]} \sup_{\alpha \in \Lambda} \|X_0^{(r,\alpha)}\|_{\cX_r}<\infty,
\end{equation}
then
\begin{equation}
\label{X-ra}
\sup_{r \in [0,T]} \sup_{\alpha \in \Lambda} \|X^{(r,\alpha)}\|_{\cX_r} <\infty.
\end{equation}
\end{theorem}

\begin{proof} a) For any $n\geq 0$, we define the Picard iterations:
\begin{equation}
\label{picard}
X_{n+1}(t,x)=X_0^{(r,\alpha)}(t,x)+\int_r^t \int_{\bR} \G_{t-s}(x-y)X_n(s,y)\fX(ds,dy),
\end{equation}
for all $t \in [r,T]$ and $x\in \bR$. Note that $X_n=X_n^{(r,\alpha)}$ depends on $(r,\alpha)$. To simplify the writing, we omit writing the upper index $(r,\alpha)$ in the first part of the argument.

The recurrence relation \eqref{picard} also holds for $n=-1$,
letting $X_{-1}(t,x)=0$. For any $n\geq 0$ and $t \in [r,T]$, we define
\begin{align*}
V_n(t)&=\sup_{x \in \bR}\|X_{n}(t,x)-X_{n-1}(t,x)\|_p^2 \\
W_n(t)&=\sup_{x \in \bR}
\int_r^t \int_{\bR^d}\G_{t-s}^2(x-y) \\
& \|
\big(X_{n}(s,y)-X_{n-1}(s,y)\big)-
\big(X_{n}(s,y+h)-X_{n-1}(s,y+h)\big)\|_p^2 |h|^{2H-2} dhdyds.
\end{align*}
As in the proof of Theorem 3.8 of \cite{BJQ15}, for any $t \in [r,T]$ and $n\geq 0$, we have:
\begin{align*}
V_{n+1}(t) & \leq 2 C_{p,H} \left( \int_r^t V_n(s) J_1(t-s)ds+ W_n(t)\right)\\
W_{n+1}(t) &\leq 2 C_{p,H} \left( \int_r^t V_n(s) J_2(t-s)ds+ \int_r^t W_n(s) J_1(t-s)ds\right),
\end{align*}
where $C_{p,H}$ is the constant from \eqref{BDG}. Letting $f_n(t)=V_n(t)+W_n(t)$, we have: for $n \geq 1$,
\[
f_{n+1}(t) \leq 2C_{p,H}(C_{p,H}+1) \int_r^t \big( f_n(s)+f_{n-1}(s)\big)ds, \quad \mbox{for all} \ t \in [r,T].
\]
Note that this is precisely the recurrence relation appearing in Lemma 3.10 of \cite{BJQ15}. To apply this lemma, we need to show that $f_0$ and $f_1$ are uniformly bounded. For $n=0$,
\[
\sup_{t \in [r,T]}\sqrt{f_0(t)} \leq \sup_{t \in [r,T]} \sqrt{V_0(t)}+\sup_{t \in [r,T]} \sqrt{W_0(t)}=\|X_0^{(r,\alpha)}\|_{\cX_r}=:M_0
\]
For $n=1$, using the recurrence relations above, we have:
\begin{align*}
V_{1}(t) & \leq 2C_{p,H} \left( \int_r^t V_0(s) J_1(t-s)ds+W_0(t)\right)\leq 2 C_{p,H} \|X_0^{(r,\alpha)}\|_{\cX_r}^2 \big( Q_1(t-r)+1 \big)\\
W_1(t) & \leq 2C_{p,H} \left(  \int_r^t V_0(s) J_2(t-s)ds +  \int_r^t W_0(s) J_1(t-s)ds\right)\\
& \leq 2C_{p,H} \|X_0^{(r,\alpha)}\|_{\cX_r}^2 \big( Q_1(t-r)+Q_2(t-r) \big),
\end{align*}
and hence
\[
\sup_{t \in [r,T]} f_1(t) \leq 2C_{p,H} \|X_0^{(r,\alpha)}\|_{\cX_r}^2
\big( 2Q_1(T-r)+Q_2(T-r) +1\big)=:M_1.
\]
Let $M=M_0+M_1$. Applying now Lemma 3.10 of \cite{BJQ15}, we infer that there exists a sequence $(a_n)_{n\geq 0}$ of positive numbers with the property $\sum_{n\geq 0}a_n^{1/p}<\infty$ for any $p>1$, such that
\[
\sup_{t \in [r,T]} f_n(t) \leq M a_n \quad \mbox{for all} \quad n\geq 0.
\]
It follows that
$
\|X_n-X_{n-1}\|_{\cX_r} =\sup_{t \in [r,T]}\sqrt{V_n(t)}+\sup_{t \in [r,T]}\sqrt{W_n(t)}\leq 2 \sqrt{Ma_n}$.
Hence $(X_n)_{n\geq 0}$ is a Cauchy sequence in $\cX_r$. Its limit $X$ is the unique solution of \eqref{gen-eq} in $\cX_r$.

b) We now prove the last statement regarding the uniform bound in $(r,\alpha)$. We include the upper indices $(r,\alpha)$ in this part. We have:
\[
\|X_n^{(r,\alpha)}\|_{\cX_r} \leq \|X_0^{(r,\alpha)}\|_{\cX_r}+ \sum_{k=1}^{n}\|X_k^{(r,\alpha)}-X_{k-1}^{(r,\alpha)}\|_{\cX_r} \leq \|X_0^{(r,\alpha)}\|_{\cX_r} +2 (M^{(r,\alpha)})^{1/2} \sum_{n\geq 1}a_n^{1/2}.
\]
Relation \eqref{X-ra} follows letting $n \to \infty$, using the fact that $M^{(r,\alpha)}$ is uniformly bounded for all $r \in [0,T]$ and $\alpha \in \Lambda$, due to condition \eqref{cond-X0}.
\end{proof}

\begin{example}
\label{ex-v}
{\rm As an application of Theorem \ref{gen-th}, we consider the (hAm) model \eqref{eq-V} with noise $\fX$ and Dirac delta initial velocity. For any $r \in [0,T]$ and $z \in \bR$ fixed, the solution $V_{\theta}^{(r,z)}$ satisfies the integral equation \eqref{eq-V-int}. This is precisely the stochastic Volterra equation \eqref{gen-eq} with initial condition $X_0^{(r,z)}(t,x)=G_{t-r}(x-z)$ and $\G_t=\sqrt{\theta}G_t$. In this application of Theorem \ref{gen-th}, $\alpha=z$ and $\Lambda=\bR$. Condition \eqref{cond-X0} clearly holds since:
\begin{align*}
\|X_0^{(r,z)}\|_{\cX_{r,1}}^2&=\sup_{(t,x)\in [r,T] \times \bR}G_{t-r}^2(x-z) \leq \frac{1}{4} \\
\|X_0^{(r,z)}\|_{\cX_{r,2}}^2 &=\theta \sup_{(t,x)\in [r,T] \times \bR}
\int_r^t \int_{\bR^2} G_{t-s}^2(x-y) |G_{s-r}(y-z)-G_{s-r}(y+h-z)|^2 |h|^{2H-2}dhdy ds \\
& \leq \frac{1}{4}\theta C_{H,1}\sup_{t \in [r,T]}\int_r^t (s-r)^{2H}ds=\theta C_{H,1}'(T-r)^{2H+1}.
\end{align*}
We conclude that
\begin{equation}
\label{vrz-bound}
\sup_{r \in [0,T]} \sup_{z \in \bR}\|V_{\theta}^{(r,z)}\|_{\cX_r}<\infty.
\end{equation}
}
\end{example}

\begin{remark}
{\rm Theorem \ref{gen-th} cannot be applied to the (pAm) model with Dirac delta initial condition, since $\sup_{t \in [0,T]} \sup_{x\in\bR}g_t(x)=\infty$, where $g_t(x)=(2\pi t)^{-1/2}e^{-x^2/t}$ is the heat kernel. For any $x\not=0$, the function $t \mapsto g_t(x)$ attains its maximum at $t=x^2$, and if $|x|\leq T$, this maximum value is $c|x|^{-1}$ for some constant $c>0$.}
\end{remark}

\medskip

We study now a second parametric family of stochastic Volterra equations, which will be useful for treating the increments of $V_{\theta}$. We introduce the following assumption.

\medskip

{\bf Assumption B.} For any $i=1,3,4$, $J_i(t)<\infty$ for any $t\in [0,T]$ and
\[
Q_i(T):=\int_0^T J_i(t)dt<\infty.
\]

\medskip

Let $p\geq 2$ be arbitrary. For any $r \in [0,T]$, let $\cY_r$ be the set of all processes $\{Y(t,x,z,h);t \in [r,T], (x,z,h) \in \bR^3\}$ such that the map $(\omega,t,x,z,h) \mapsto Y(\omega,t,x,z,h)$ is $\cP \times \cB(\bR^2)$-measurable and
$\|Y\|_{\cY_r}:=\|Y\|_{\cY_{r,1}}+\|Y\|_{\cY_{r,2}}<\infty$, where
\begin{align*}
\|Y\|_{\cY_{r,1}}&=\sup_{t \in [r,T]} \sup_{z\in \bR} \left(\int_{\bR^2}\|Y(t,x,z,h)\|_p^2 |h|^{2H-2}dxdh \right)^{1/2}\\
\|Y\|_{\cY_{r,2}}&=\sup_{t\in [r,T]} \sup_{z\in \bR} \left(\int_{\bR^3}\|Y(t,x,z,h)-Y(t,x+k,z,h)\|_p^2 |h|^{2H-2}
|k|^{2H-2}dxdh dk \right)^{1/2}.
\end{align*}
Here $\cP$ is the predictable $\sigma$ field on $\Omega \times \bR_{+}\times \bR$.

\begin{theorem}
\label{theoremF}
Suppose Assumption B holds.

a) For any $r \in [0,T]$, the family of stochastic Volterra equations:
\begin{equation}
\label{eq-Xr}
X(t,x,z,h)=X_0^{(r)}(t,x,z,h)+\int_r^t \int_{\bR}\G_{t-s}(x-y)X(s,y,z,h)\fX(ds,dy),
\end{equation}
with $t \in [r,T]$ and $(x,z,h) \in \bR^3$, has a unique solution $X^{(r)}$ in $\cY_r$, provided that $X_0^{(r)} \in \cY_r$.

 b) If in addition,
\[
M:=\sup_{r \in [0,T]}\|X_0^{(r)}\|_{\cY_{r}}<\infty,
\]
then
\begin{equation}
\label{M-sup-r}
\sup_{r \in [0,T]}\|X^{(r)}\|_{\cY_{r}}\leq C_{T,H,p}M,
\end{equation}
where $C_{T,H,p}>0$ is a constant that depends on $(T,H,p)$.
\end{theorem}

\begin{proof}
a) We fix $r \in [0,T]$.
To illustrate the main ideas, assume first that a solution $X^{(r)}$ exists. Denote $X^{(r,z,h)}(t,x):=X^{(r)}(t,x,z,h)$.

We use the notation $a \les b$ if $a \leq C b$ and $C>0$ is a constant that depends on $p$ and $H$.
By the BDG inequality \eqref{BDG} and triangular inequality,
\begin{align*}
& \big\|X^{(r,z,h)}(t,x)\big\|_p^2 \les \big|X_0^{(r,z,h)}(t,x)\big|^2+\\
&
\int_r^t \int_{\bR^2} \big\|\G_{t-s}(x-y)X^{(r,z,h)}(s,y)-
\G_{t-s}(x-y-k)X^{(r,z,h)}(s,y+k)\big\|_p^2 \, |k|^{2H-2}dkdyds \\
&\les \big|X_0^{(r,z,h)}(t,x)\big|^2+ \int_r^t \int_{\bR} \big(\G_{t-s}(x-y) -\G_{t-s}(x-y-k)\big)^2 \big\|X^{(r,z,h)}(s,y)\big\|_p^2 \, |k|^{2H-2}dkdyds\\
& \qquad \qquad \qquad \quad +\int_r^t \int_{\bR} \G_{t-s}^2(x-y-k) \big\|X^{(r,z,h)}(s,y)-X^{(r,z,h)}(s,y+k)\big\|_p^2 \, |k|^{2H-2}dkdyds.
\end{align*}
We multiply by $|h|^{2H-2}$ and we integrate $dx dh$ on $\bR^2$.
For the second term, we use
\[
\int_{\bR^2}\big(\G_{t-s}(x-y)-\G_{t-s}(x-y-k)\big)^2 |k|^{2H-2}dxdk=J_1(t-s),
\]
and for the third term, we use $\int_{\bR}\G_{t-s}^2(x-y-k)dx=J_4(t-s)$. We obtain:
\begin{align*}
& \int_{\bR^2}\big\|X^{(r,z,h)}(t,x)\big\|_p^2 \, |h|^{2H-2}dx dh \les \int_{\bR^2}\big|X_0^{(r,z,h)}(t,x)\big|^2 \,|h|^{2H-2}dx dh \\
& \quad + \int_r^t J_1(t-s) \left( \int_{\bR^2}\big\|X^{(r,z,h)}(s,y)\big\|_p^2 \, |h|^{2H-2}dydh \right) ds \\
& \quad + \int_r^t J_4(t-s)  \left(\int_{\bR^3} \big\|X^{(r,z,h)}(s,y)-X^{(r,z,h)}(s,y+k)\big\|_p^2
\, |h|^{2H-2}|k|^{2H-2} dy dh dk\right) ds.
\end{align*}
Taking the supremum over $z \in \bR$, we obtain that for any $t \in [r,T]$,
\begin{equation}
\label{alpha-eq}
\alpha^{(r)}(t) \les \big\|X_0^{(r)}\big\|_{\cY_{r,1}}^2+\int_r^t \big(J_1(t-s)\alpha^{(r)}(s)+J_4(t-s) \beta^{(r)}(s)\big)ds,
\end{equation}
where
\begin{align*}
\alpha^{(r)}(t)&= \sup_{z\in \bR}\int_{\bR^2} \big\|X^{(r,z,h)}(t,x)\big\|_p^2  \, |h|^{2H-2}dxdh\\
\beta^{(r)}(t) &=\sup_{z\in \bR}\int_{\bR^3}
\big\|X^{(r,z,h)}(t,x)-X^{(r,z,h)}(t,x+k)\big\|_p^2
\, |h|^{2H-2}|k|^{2H-2} dx dh dk.
\end{align*}

To make this argument work, relation \eqref{alpha-eq} should be paired with a similar inequality for $\beta^{(r)}(t)$. We show below how to obtain this. Again, by BDG inequality \eqref{BDG} and triangle inequality,
\begin{align*}
& \big\|X^{(r,z,h)}(t,x)-X^{(r,z,h)}(t,x+k)\big\|_p^2 \les \big|X_0^{(r,z,h)}(t,x)-X_0^{(r,z,h)}(t,x+k)\big|^2+\\
& \quad \int_r^t \int_{\bR^2}\big\|\big(\G_{t-s}(x-y)-\G_{t-s}(x+k-y)\big) X^{(r,z,h)}(s,y)-\\
& \quad \big(\G_{t-s}(x-y-w)-\G_{t-s}(x+k-y-w) \big)X^{(r,z,h)}(s,y+w)\big\|_p^2 |w|^{2H-2} dwdyds\\
& \les |X_0^{(r,z,h)}(t,x)-X_0^{(r,z,h)}(t,x+k)|^2+\int_r^t \int_{\bR^2} \big|\big(\G_{t-s}(x-y)-\G_{t-s}(x+k-y)\big)-\\
& \quad \quad \quad \big(\G_{t-s}(x-y-w)-\G_{t-s}(x+k-y-w)\big)\big|^2 \, \big\|X^{(r,z,h)}(t,x)\big\|_p^2 \,|w|^{2H-2}dwdyds+ \\
& \int_r^t \int_{\bR^2} \big| \G_{t-s}(x-y-w)-\G_{t-s}(x+k-y-w)\big|^2 \,
\big\|X^{(r,z,h)}(s,y)-X^{(r,z,h)}(s,y+w)\big\|_p^2 \\
& \quad \quad \quad |w|^{2H-2} dwdyds.
\end{align*}
We multiply by $|h|^{2H-2}|k|^{2H-2}$ and we integrate $dx dh dk$ on $\bR^3$. For the second term, we use
\begin{align*}
& \int_{\bR^3}\big|\big(\G_{t-s}(x-y)-\G_{t-s}(x+k-y)\big)-
\big(\G_{t-s}(x-y-w)-\G_{t-s}(x+k-y-w)\big)\big|^2 \\ & \quad \quad \quad |k|^{2H-2}|w|^{2H-2} dw dx dk=J_3(t-s),
\end{align*}
and for the third one,
$\int_{\bR^2} \big| \G_{t-s}(x-y-w)-\G_{t-s}(x+k-y-w)\big|^2 \,
|k|^{2H-2}dxdk=J_1(t-s)$.
We obtain:
\begin{align*}
& \int_{\bR^3}\big\|X^{(r,z,h)}(t,x)-X^{(r,z,h)}(t,x+k)\big\|_p^2 \, |h|^{2H-2}|k|^{2H-2} dxdhdk \les \\
& \int_{\bR^3}\big|X_0^{(r,z,h)}(t,x)-X_0^{(r,z,h)}(t,x+k)\big|^2 \,|h|^{2H-2}|k|^{2H-2} dxdhdk+\\
& \quad \int_r^t J_3(t-s) \left(\int_{\bR^2}
\big\|X^{(r,z,h)}(t,x)\big\|_p^2 \,|h|^{2H-2} dydh\right)ds+\\
& \quad \int_r^t J_1(t-s) \left(\int_{\bR^3}
\big\|X^{(r,z,h)}(s,y)-X^{(r,z,h)}(s,y+w)\big\|_p^2 |h|^{2H-2} |w|^{2H-2} dydwdh \right)  ds.
\end{align*}
Taking the supremum over $z \in \bR$, we obtain that for any $t \in [r,T]$,
\begin{equation}
\label{beta-eq}
\beta^{(r)}(t) \les \big\|X_0^{(r)}\big\|_{\cY_{r,2}}^2+\int_r^t \Big(J_3(t-s)\alpha^{(r)}(s)+J_1(t-s) \beta^{(r)}(s)\Big)ds.
\end{equation}

We now prove the existence of solution of equation \eqref{eq-Xr}, for fixed $r \in [0,T]$.  For $n\geq 0$, we define the Picard iterations:
\[
X_{n+1}(t,x,z,h)=X_0^{(r)}(t,x,z,h)+\int_r^t \int_{\bR}\G_{t-s}(x-y)X_n(s,y,z,h)\fX(ds,dy),
\]
for all $t \in [r,T]$ and $x,z,h \in \bR$. Denote $X_n^{(r,z,h)}(t,x):=X_n^{(r)}(t,x,z,h)$. Letting $X_{-1}^{(r,z,h)}=0$, we see that the following recurrence relation holds for any $n\geq 0$:
\[
\big(X_{n+1}^{(r,z,h)}-X_{n}^{(r,z,h)}\big)(t,x)=
\int_r^t \int_{\bR}\G_{t-s}(x-y)\big(X_n^{(r,z,h)}-X_{n-1}^{(r,z,h)}\big)(s,y)
\fX(ds,dy).
\]

For any $n\geq 0$ and $t \in [r,T]$, we define
\begin{align*}
\alpha_n^{(r)}(t) &= \sup_{z\in \bR} \int_{\bR^2}\big\|(X_n^{(r,z,h)}-X_{n-1}^{(r,z,h)})(t,x)\big\|_p^2 \,|h|^{2H-2}dxdh \\
\beta_n^{(r)}(t) &=\sup_{z\in \bR} \int_{\bR^3}\big\|(X_n^{(r,z,h)}-X_{n-1}^{(r,z,h)})(t,x)-
(X_n^{(r,z,h)}-X_{n-1}^{(r,z,h)})(t,x+k)\big\|_p^2 \\
& \quad \quad \quad \quad \quad  |h|^{2H-2}
|k|^{2H-2}dxdh dk.
\end{align*}

Similarly to \eqref{alpha-eq} and \eqref{beta-eq}, we obtain that for any $n\geq 0$ and $t \in [r,T]$,
\begin{align*}
\alpha_{n+1}^{(r)}(t) \leq C_{p,H}\int_r^t \big(J_1(t-s)\alpha_n^{(r)}(s)+J_4(t-s) \beta_n^{(r)}(s)\big)ds \\
\beta_{n+1}^{(r)}(t) \leq C_{p,H}\int_r^t \Big(J_3(t-s)\alpha_n^{(r)}(s)+J_1(t-s) \beta_n^{(r)}(s)\Big)ds,
\end{align*}
where $C_{p,H}$ is a constant depending on $p$ and $H$. Denote $f_n^{(r)}=\alpha_n^{(r)}+\beta_n^{(r)}$ and $J=J_1+J_3+J_4$.
Then, for any $n \geq 0$ and $t \in [r,T]$,
\[
f_{n+1}^{(r)}(t) \leq C_{p,H} \int_r^t f_n^{(r)}(s)J(t-s)ds.
\]

For the initial term, we have:
\begin{equation}
\label{M-bound1}
M_0^{(r)}:=\sup_{t \in [r,T]} \big(\alpha_{0}^{(r)}(t)+\beta_{0}^{(r)}(t)\big) \leq \|X_0^{(r)}\|_{\cY_{r}}^2.
%\sup_{t \in [r,T]} \alpha_0^{(r)}(t)+\sup_{t \in [r,T]} %\beta_0^{(r)}(t)=\|X_0^{(r)}\|_{\cY_{r,1}}^2+
%\|X_0^{(r)}\|_{\cY_{r,2}}^2 \leq 2 \|X_0^{(r)}\|_{\cY_{r}}^2.
\end{equation}

By Assumption A, $Q(T):=C_{p,H}\int_0^T J(t)dt<\infty$. By Lemma 15 of \cite{dalang99} (an extension of Gronwall Lemma), there exists a sequence $(a_n)_{n\geq 0}$ of non-negative real numbers with the property $\sum_{n\geq 1}a_n^{1/q}<\infty$ for any $q\geq 1$, such that
\[
f_n^{(r)}(t) \leq M_0^{(r)}a_n \quad \mbox{for all $t \in [r,T]$}.
\]
More precisely,
\[
a_n=a_n(p,H,T)=Q(T)^n P(S_n \leq T),
\]
where $S_n=\sum_{i=1}^n X_i$ and $(X_i)_{i\geq 1}$ are i.i.d. random variables with values in $[0,T]$ of law $J(\cdot)/\int_0^T J(t)dt$. It follows that:
\begin{equation}
\label{M-bound2}
\|X_n^{(r)}-X_{n-1}^{(r)}\|_{\cY_r}=\sup_{t\in [r,T]} \big(\alpha_n^{(r)}(t)\big)^{1/2}+
 \sup_{t\in [r,T]} \big(\beta_n^{(r)}(t)\big)^{1/2}\leq 2 (M_0^{(r)})^{1/2}a_n^{1/2}.
\end{equation}
From this, we deduce that
$\sum_{n\geq 0}\|X_n^{(r)}-X_{n-1}^{(r)}\|_{\cY_r} <\infty$.
Therefore, $(X_n^{(r)})_{n\geq 0}$ is a Cauchy sequence in $\cY_{r}$. Its limit $X^{(r)}$ is a solution of equation \eqref{eq-Xr}. Uniqueness follows by standard methods.

b) From \eqref{M-bound1} and \eqref{M-bound2}, we obtain that
$\|X_n^{(r)}-X_{n-1}^{(r)}\|_{\cY_r} \leq 2 \|X_0^{(r)}\|_{\cY_{r}} a_n^{1/2}$. Hence,
\[
\|X_n^{(r)}\|_{\cY_r} \leq \|X_0^{(r)}\|_{\cY_r}+\sum_{k=1}^n \|X_k^{(r)}-X_{k-1}^{(r)}\|_{\cY_r}\leq
\Big(1+2 \sum_{k=1}^n a_k^{1/2}\Big)\|X_0^{(r)}\|_{\cY_r}.
\]
Letting $n \to \infty$, we deduce that
$\|X^{(r)}\|_{\cY_r}\leq \Big(1+2 \sum_{n\geq 1} a_n^{1/2}\Big)\|X_0^{(r)}\|_{\cY_r}$ for any $r \in [0,T]$. Relation \eqref{M-sup-r} follows with
\[
C_{p,H,T}=1+2 \sum_{n\geq 1} a_n^{1/2}
\]

\end{proof}

\begin{example}
\label{TheoremF-ex}
{\rm As an application of Theorem \ref{theoremF} we consider the (hAm) model \eqref{eq-V} with noise $\fX$ and Dirac delta initial velocity. For any $r \in [0,T]$ and $z \in \bR$ fixed, the solution $V_{\theta}^{(r,z)}$ satisfies the integral equation \eqref{eq-V-int}. Then
\[
X^{(r)}(t,x,z,h)=X^{(r,z,h)}(t,x)=V_{\theta}^{(r,z)}(t,x)-
V_{\theta}^{(r,z+h)}(t,x)
\]
is the unique solution of the stochastic Volterra equation \eqref{eq-Xr} with $\G_t:=\sqrt{\theta}G_t$ and initial condition:
\[
X_0^{(r)}(t,x,z,h)=G_{t-r}(x-z)-G_{t-r}(x-z-h).
\]
In this case, the functions $J_1,J_3,J_4$ are replaced, respectively, by
\[
\bar{J}_{1}(t)=\theta C_{H,1}t^{2H}, \quad \bar{J}_{3}(t)=\theta C_{H,3}t^{4H-1}, \quad \bar{J}_{4}(t)=\theta C_{H,4}t,
\]
%respectively $\bar{J}(t)=\theta(C_{H,1}t^{2H}+C_{H,3}t^{4H-1}+C_{H,4}t)$,
and the constant $a_n$ is replaced by
\begin{align*}
\overline{a}_n&=\overline{a}_n(p,H,T,\theta)=\bar{Q}(T)^n P(S_n\leq T)\\
&=C_{p,H}^n\theta^n\big(C_{H,1}'T^{2H+1}+C_{H,3}'T^{4H}+C_{H,4}'T^2\big)^nP(S_n\leq T),
\end{align*}
where $\bar{Q}(T):=C_{p,H}\int_0^T \bar{J}(t)dt$ and $\bar{J}=\bar{J}_1+\bar{J}_3+\bar{J}_4$.
%=C_{p,H}\theta(C_{H,1}'T^{2H+1}+C_{H,3}'T^{4H}+C_{H,4}'T^2)$.
Here $S_n=\sum_{i=1}^n X_i$ where $(X_i)_{i\geq 1}$ are i.i.d. random variables on $[0,T]$ of law $\bar{J}(\cdot)/\int_0^T \bar{J}(T)$ (which does not depend on $\theta$). As for the bound for the initial condition, we have:
\begin{align*}
\|X_0^{(r)}\|_{\cY_{r,1}}^2&=\sup_{t \in [r,T]} \sup_{z\in \bR} \int_{\bR^2} \big|G_{t-r}(x-z)-G_{t-r}(x-z-h)\big|^2 \, |h|^{2H-2} dxdh \\ & =C_{H,1}\sup_{t \in [r,T]}(t-r)^{2H}=C_{H,1}(T-r)^{2H}\\
\|X_0^{(r)}\|_{\cY_{r,2}}^2&=\sup_{t \in [r,T]} \sup_{z\in \bR} \int_{\bR^3} \big|\big(G_{t-r}(x-z)-G_{t-r}(x-z-h)\big)-\\
& \quad \quad \quad
\big(G_{t-r}(x+k-z)-G_{t-r}(x+k-z-h)\big)\big|^2 \, |h|^{2H-2} |k|^{2H-2}dxdh dk \\ & =C_{H,3}\sup_{t \in [r,T]}(t-r)^{4H-1}=C_{H,3}(T-r)^{4H-1},
\end{align*}
and hence
\[
M:=\sup_{r \in [0,T]}\|X_0^{(r)}\|_{\cY_{r}}= \sup_{r \in [0,T]}\big(C_{H,1}^{1/2}(T-r)^{H}+C_{H,3}^{1/2}(T-r)^{2H-1/2}\big) \leq C_{H}^{1/2} (T^{H}+T^{2H-1/2}),
\]
where $C_H=\max(C_{H,1},C_{H,3})$.
By Theorem \ref{theoremF}, we infer that:
\begin{align}
\label{V-int1}
& \sup_{0\leq r\leq t \leq T} \sup_{z\in \bR} \int_{\bR^2}\big\|V_{\theta}^{(r,z)}(t,x)-
V_{\theta}^{(r,z+h)}(t,x)\big\|_p^2 |h|^{2H-2}dxdh \leq C_{V, p,H,T,\theta}\\
\nonumber
& \sup_{0\leq r\leq t \leq T} \sup_{z\in \bR} \int_{\bR^3}\big\|V_{\theta}^{(r,z)}(t,x)-
V_{\theta}^{(r,z+h)}(t,x)-V_{\theta}^{(r,z)}(t,x+k)+
V_{\theta}^{(r,z+h)}(t,x+k)\big\|_p^2 \\
\label{V-int2}
& \qquad \quad \quad \quad \qquad |h|^{2H-2}
|k|^{2H-2}dxdh dk  \leq C_{V,p,H,T,\theta},
\end{align}
where
\begin{equation}
\label{TheoremF-const}
C_{V,p,H,T,\theta}= 2 C_H(T^{2H}+T^{4H-1}) \Big(1+2\sum_{n\geq 1}\overline{a}_n(p,H,T,\theta)^{1/2}\Big)^2.
\end{equation}
}
\end{example}

\begin{remark}
\label{Rmk-norm-dim}
{\rm Theorem \ref{theoremF} also holds if we consider processes $\{Y(t,x,z,\pmb{h})\}$ with multi-parameter $\pmb{h}=(h_1,\ldots,h_m)\in \bR^m$ instead of $h$, and we replace the integral $|h|^{2H-2}dh$ on $\bR$ by the integral $\prod_{i=1}^{m}|h_i|^{2H-2}d\pmb{h}$ on $\bR^m$ in the definitions of the norms $\|\cdot\|_{\cY_{r,1}}$ and  $\|\cdot\|_{\cY_{r,2}}$,
In addition, Theorem \ref{theoremF} holds if we remove $h$ from the parametrization, i.e. we consider processes $\{Y(t,x,z)\}$, and we drop the integral $|h|^{2H-2}dh$ from the definitions of the norms $\|\cdot\|_{\cY_{r,1}}$ and  $\|\cdot\|_{\cY_{r,2}}$.
%(which yields a different result than Theorem \ref{gen-th}).
% (since we do not have $\G$ in $\|\cdot\|_{\cY_{r,2}}$).
}
\end{remark}

\section{Integral inequalities}
\label{section-appC-ineq}

In this section, we give some basic inequalities related to the function $\cF G_t(\xi)=\frac{\sin(t|\xi|)}{|\xi|}$.

\begin{lemma} \label{Lem-ineq-integral}
Let $\varphi \in L^1(\bR)$. For $\beta \in (0,2)$ and $t,s \in \bR$, we have
\begin{align*}
	\int_{\bR} \dfrac{|\sin(t|x|) \sin(s|x|)|}{|x|^2} |\varphi(x)| |x|^{\beta} dx
	\le (1+|ts|) \|\varphi\|_{L^1(\bR)}.
\end{align*}
\end{lemma}

\begin{proof}
We denote by $I_1$, $I_2$ the integrals over the regions $|x|\leq 1$, respectively $|x|>1$. In these regions, we use the inequalities
$|\sin(x)| \le |x|$, respectively $|\sin(x)| \le 1$. Then,
\begin{align*}
I_1 & \leq |ts|\int_{|x|\leq 1}|\varphi(x)||x|^{\beta}dx\leq  |ts|\int_{|x|\leq 1}|\varphi(x)|dx\\
I_2 & \leq \int_{|x|> 1}|\varphi(x)||x|^{\beta-2}dx\leq  \int_{|x|> 1}|\varphi(x)|dx.
\end{align*}
%\begin{align*}
%	& \int_{\bR} \dfrac{|\sin(t|x|) \sin(s|x|)|}{|x|^2} %\varphi(x) |x|^{\beta} dx \\
%	=& \int_{|x|\ge 1} \dfrac{|\sin(t|x|) \sin(s|x|)|}{|x|^2} %\varphi(x) |x|^{\beta} dx
%	+ \int_{|x| \le 1} \dfrac{|\sin(t|x|) \sin(s|x|)|}{|x|^2} %\varphi(x) |x|^{\beta} dx \\
%	\le& \int_{|x|\ge 1} \varphi(x) |x|^{\beta-2} dx
%	+ |ts| \int_{|x| \le 1} \varphi(x) |x|^{\beta} dx \\
%	\le& \int_{|x|\ge 1} \varphi(x) dx
%	+ |ts| \int_{|x| \le 1} \varphi(x) dx
%	\le (1+|ts|) \|\varphi\|_{L^1(\bR)}.
%\end{align*}
\end{proof}

The next result is obtained by a similar method.

\begin{lemma}
\label{Lem-ineq-integral-1}
For $\beta \in (0,1)$ and $t,s \in \bR$, we have
\begin{align*}
	\int_{\bR} \dfrac{|\sin(t|x|) \sin(s|x|)|}{|x|^2} |x|^{\beta} dx
	\le \dfrac{2}{1-\beta} + \dfrac{2|ts|}{1+\beta}.
\end{align*}
\end{lemma}

%\begin{align*}
%	& \int_{\bR} \dfrac{|\sin(t|x|) \sin(s|x|)|}{|x|^2} %|x|^{\beta} dx \\
%	=& \int_{|x|\ge 1} \dfrac{|\sin(t|x|) \sin(s|x|)|}{|x|^2} %|x|^{\beta} dx
%	+ \int_{|x| \le 1} \dfrac{|\sin(t|x|) \sin(s|x|)|}{|x|^2} %|x|^{\beta} dx \\
%	\le& \int_{|x|\ge 1} |x|^{\beta-2} dx
%	+ |ts| \int_{|x| \le 1} |x|^{\beta} dx
%	= \dfrac{2}{1-\beta} + \dfrac{2|ts|}{1+\beta}
%\end{align*}

%\noindent {\footnotesize{\em Acknowledgement.} Wangjun Yuan %gratefully acknowledges the financial support of ERC %Consolidator Grant 815703 "STAMFORD: Statistical Methods for %High Dimensional Diffusions"}

\end{document}